\documentclass[11pt]{amsart}

\usepackage{amsmath, amsthm, amssymb, mathrsfs}
\usepackage[hidelinks]{hyperref}
\usepackage[parfill]{parskip}    \parskip = 0.2cm    
\usepackage[margin=1.12in]{geometry} 
\usepackage{tikz}
\usepackage{tikz-cd}
\usetikzlibrary{matrix}
\usepackage{comment} 

\usepackage{mathtools}

\usepackage{mathabx}

\graphicspath{ {images/} }
\usepackage{import}

\numberwithin{equation}{section}

\title[{\bf Substitutional cut and project sets }]{When is a cut and project set substitutional? }
\date{\today}

\author{Edmund Harriss}
\address{University of Arkansas}
\email{}

\author{Henna Koivusalo}
\address{School of Mathematics, Fry Building, Woodland Road, Bristol BS8 1UG, United Kingdom}
\email{henna.koivusalo@bristol.ac.uk }
\urladdr{https://people.maths.bris.ac.uk/~te20281/index.html}

\author{James J.\ Walton} 
\address{School of Mathematical Sciences, Mathematical Sciences Building, University Park, Nottingham, NG7 2RD, United Kingdom}
\email{James.Walton@nottingham.ac.uk}
\urladdr{https://www.nottingham.ac.uk/mathematics/people/james.walton}

\theoremstyle{plain}
\newtheorem{theorem}{Theorem}[section]
\newtheorem{lemma}[theorem]{Lemma}

\newtheorem{proposition}[theorem]{Proposition}
\newtheorem{corollary}[theorem]{Corollary}
\newtheorem*{infmain}{Informal statement of Theorem \ref{thm:main}}
\newtheorem*{infpoly}{Informal statement of Theorem \ref{thm:polytopal}}

\theoremstyle{definition}
\newtheorem{definition}[theorem]{Definition}
\newtheorem{remark}[theorem]{Remark}
\newtheorem{example}[theorem]{Example}
\AtEndEnvironment{example}{\null\hfill\exend}
\newtheorem{notation}[theorem]{Notation}

\newcommand{\exend}{\hfill \ensuremath{\Diamond}}

\newcommand{\R}{\mathbb R}

\newcommand{\Z}{\mathbb Z}
\newcommand{\Q}{\mathbb Q}

\newcommand{\N}{\mathbb N}

\newcommand{\sH}{\mathscr{H}}

\newcommand{\sing}{\mathscr{S}}
\newcommand{\nonsing}{\mathscr{N}}

\newcommand{\clin}{\varkappa}

\newcommand{\E}{\mathbb E}

\renewcommand{\epsilon}{\varepsilon}

\newcommand{\intr}{\mathrm{int}}
\newcommand{\cl}{\mathrm{cl}}

\DeclareMathOperator{\Id}{\mathrm{Id}}
\DeclareMathOperator{\rk}{rk}

\DeclareMathOperator{\tot}{\E}               
\DeclareMathOperator{\phy}{\E_\vee}          
\DeclareMathOperator{\intl}{\E_<}            
\DeclareMathOperator{\cps}{\Lambda}          
\DeclareMathOperator{\cP}{\mathcal{P}} 
\DeclareMathOperator{\cQ}{\mathcal{Q}} 
\DeclareMathOperator{\cU}{\mathcal{U}} 
\DeclareMathOperator{\cV}{\mathcal{V}}
\DeclareMathOperator{\cT}{\mathcal{T}}

\newcommand{\sub}{\sigma}

\DeclarePairedDelimiter\ang{\langle}{\rangle} 

\newcommand{\xra}[1]{\overset{#1}{\rightsquigarrow}}

\newcommand{\xrla}[1]{\overset{#1}{\leftrightsquigarrow}}
\DeclareMathOperator{\LD}{\scalebox{0.8}{\(\xra{\mathrm{LD}}\)}} 
\DeclareMathOperator{\MLD}{\scalebox{0.8}{\(\xrla{\mathrm{MLD}}\)}}

\newcommand{\xsqsubset}[1]{\overset{#1}{\sqsubset}}
\DeclareMathOperator{\LI}{\scalebox{0.8}{\(\xsqsubset{\mathrm{LI}}\)}}
\newcommand{\xsimeq}[1]{\overset{#1}{\simeq}}
\DeclareMathOperator{\LIs}{\scalebox{0.8}{\(\xsimeq{\mathrm{LI}}\)}}


\begin{document}

\maketitle
\thispagestyle{empty}


\begin{abstract}
Cut and project sets are obtained by projecting an irrational slice through a lattice to a lower dimensional subspace. Under standard conditions, the resulting pattern has no translational periods even though it retains some regularity of the lattice. Cut and project sets are one of the archetypical examples of patterns featuring aperiodic order, the other construction methods being by substitution and matching rules. Many early examples of aperiodic tilings, including the famous Penrose and Ammann--Beenker tilings, have a description from all of these methods. In this article we answer the following question, in the case of a Euclidean total space: what property of the cut and project data characterises when the resulting cut and project sets may also be defined by a substitution rule?
\end{abstract}

\section{Introduction}

A {\it tiling} $\cT$ of Euclidean space $\R^d$ is a collection $(T_i)_{i\in\N}$ of topologically regular subsets (i.e., equal to the closures of their interiors) that intersect only on their boundaries and cover all of $\R^d$. A {\it Delone set} in $\R^d$ is a collection of points that is relatively dense and uniformly discrete. We use the term pattern to refer to both tilings and Delone sets (or variants of these where tiles or points may carry labels). The field of Aperiodic Order \cite{AOI} studies patterns that have no translational periods but feature a good amount of structure, in some sense. The main methods for producing aperiodic order are by matching rules, substitutions and cut and project.

The matching rule tilings are colections of shapes that can tile the plane but only non-periodically \cite{Berger66,Gardner89,Robinson71,Kari96,JR21}, such as the famous Penrose tiles and the recently discovered Hat einstein \cite{SMKG24}. The discovery of such tilesets in the '60s and '70s invigorated the study of tilings and patterns, with a quest to discover general methods that could create such tiles. Two powerful methods were developed, hierarchic/substitution methods that expanded and replaced tiles to create larger patches and projection methods that brought tilings down from higher dimensions. Many classical examples of matching rule tilings, examples, such as the Penrose tiling and the Hat tiling that can be generated in both of these ways.

Let us first focus on tile substitutions (or, used somewhat less commonly, Delone set substitution \cite{LW03}). Symbolic substitutions are classical in Dynamics \cite{Morse21,BI94}, and symbolic substitutions have relevance in tiling theory and vice versa \cite{PF02}. Loosely speaking, they are defined by inflate and replace rules that produce arbitrarily large patches of patterns. Proving that a given matching rule enforces hierarchy (often described by some substitution) is one of the more common approaches to proving aperiodicity of a matching rule tiling. For example, the Penrose tilings \cite{Pen79} and many other classical examples have a description from substitutions as well as from matching rules.

We note at this point that a substitution rule \(\sigma\) defines not a single tiling, but rather a collection of finite patches that, in turn, defines a collection of admitted tilings that are all \emph{locally isomorphic} (at least in the translationally repetitive/primitive case), that is, they all have the same finite patches, up to translation. In the nonperiodic case, there are uncountably many such tilings, even modulo translation equivalence, which together form a \emph{tiling space} or \emph{translational hull} \(\Omega\) \cite{Sad08book, AP98}. The tiling space is a topological dynamical system, when equipped with a natural topology and the \(\R^d\)-action of translation. The original substitution acts continuously on \(\Omega\) and, as a result of aperiodicity, is a homeomorphism on \(\Omega\). That is, it is invertible: each tiling \(\cT \in \Omega\) has a unique substitutive pre-image \(\cT' = \sigma^{-1}(\cT) \in \Omega\). This result is known as unique composition, or recognisability; see \cite{Sol98} for primitive stone inflation rules or \cite{Wal25} for a more general statement in the non-repetitive case that extends beyond tilings, to general so-called \(L\)-sub FLC pattern spaces.

The third method is by via cut and project schemes. Cut and project sets are obtained by taking an irrational slice through a lattice and projecting it to a lower dimensional subspace. More formally, a cut and project scheme \(\mathcal{S}\) consists of the following data: two (positive-dimensional) transversal subspaces \(\phy\), \(\intl < \R^k\), respectively called the \emph{physical} and \emph{internal} space, and a lattice \(\Gamma < \R^k\). We restrict here to a Euclidean \emph{total space} \(\R^k\) (although it is also of interest to allow non-Euclidean internal and total space). Let \(\pi_\vee \colon \R^k \to \phy\) be the projection with respect to the decomposition \(\R^k = \phy + \intl\). Then, given a \emph{window} \(W \subset \intl\), we define the \emph{cut and project set}:
\[
\cps = \cps(\mathcal{S},W) \coloneqq \pi_\vee (\Gamma \cap (W + \phy)) .
\]
In other words, we \emph{cut} the lattice \(\Gamma\) with the \emph{strip} \(W + \phy\), and then \emph{project} it to the physical space. For further details, see Section \ref{sec:cutandprojectDef}. 

Throughout, we use the standard assumptions that \(\pi_\vee\) is injective on \(\Gamma\) and that the projection of \(\Gamma\) to the internal space is dense (otherwise, certain reductions can be performed). We also take the projection of the lattice to internal space to be injective, which ensures non-periodicity. It is not hard to then show that \(\cps\) is a Delone set if \(W\) is compact and topologically regular (which we always will do in this article).

Just as how substitutions do not merely define single tilings, a cut and project scheme \(\mathcal{S}\) with window \(W\) more naturally defines a whole translational hull \(\Omega = \Omega(\mathcal{S},W)\) of Delone sets which, due to the above restrictions, are all locally isomorphic, non-periodic and repetitive. This set of patterns is defined by taking non-singular translates of the lattice (those not intersecting the boundary of the strip), then cutting and projecting this, and completing this collection by adding any limiting patterns corresponding to the singular translates.

The cut and project method was first given by Meyer in 1972 as a tool to study certain problems in harmonic analysis and Diophantine approximation \cite{Meyer72}. Just as many of the early substitution patterns were proven to have matching rule descriptions and vice versa, many cut and project patterns can also be generated from these other descriptions. For example, it was shown by de Bruijn in 1981 that this is true of Penrose tilings \cite{Bruijn81}.

Comparing the patterns that can be generated by these different methods is of fundamental importance to the theory of Aperiodic Order. It was shown by Goodman--Strauss in 1998 that essentially every substitution tiling can be enforced, in a certain sense, as a matching rule tiling \cite{Goodman98} (also see \cite{Mozes89} for similar considerations in the special case of Wang tiles). For a certain class of cut and project schemes, B\'edaride and Fernique \cite{BF17} found a characterisation of the existence of matching rules (also see \cite{Burkov88,Levitov88,Socolar90,BF15,FL24} for earlier and supplementary results). The problem of characterising those substitution tilings with a `nice' cut and project description is closely related to the elusive Pisot conjecture \cite{Lee23,ABBL15,AH14}, towards which so far only partial progress has been made \cite{Barge18,Queffelec87,ABBL15}. 

The current article tackles the reverse problem: characterising precisely when cut and project sets are also substitutional. Some past work exploring this problem from various perspectives include \cite{BJS91,Pleasants00,Har04,MMP21,KL23}. Our work extends a programme of establishing precise criteria on the cut and project data that correspond to natural properties of the resulting patterns. For polytopal windows, it is known that the complexity function \(p(r)\) (giving the number of patches of each radius \(r\)) grows polynomially, with power determined algebraically by the ranks of intersections of the projected lattice with supporting subspaces of the window, see \cite{Jul10, KoiWal21, Wal24}. The property of linear repetitivity \cite{LP02} requires, additionally, that the lattice is in some sense positioned in a badly approximable way with respect to the physical space, see \cite{HKW14, KoiWalII, Wal24}.

Many insights and results of the current article are rooted in the work of \cite{Har04,Harriss:OCST}. However, a key departure from this is a great simplification and clarification on what is meant by `substitutional'. Whilst there are some simple notions in the case of 1-dimensional tilings, when moving to higher dimensions it is less obvious how one defines the notion of being substitutional, in a way that encompasses examples such as the Penrose tilings, where either one must allow for non-stone inflations (where inflated tiles do not have support exactly equal to their replacement), or one needs to apply a local rule so as to work with so-called \emph{mutually locally derivable} (MLD) tilings, the Robinson triangle tilings, which do permit a stone-inflation rule. The notion of a local derivation is fundamental in Aperiodic Order \cite{AOI} and will be explained further in Section \ref{sec:setup}. Cut and project sets are Delone sets, rather than tilings, so implicitly we already perform such MLD transformations when comparing to the case of tilings. Thus, it seems natural to work with a notion of substitutional that is invariant under MLD equivalence.

In fact, the notion of a local derivation is integral to our working definition of being substitutional here. We aim to characterise those cut and project sets which we can locally derive from some scaled pattern that is locally indistinguishable from it (that is, in the repetitive case here, in the same local isomorphism class). So as to avoid confusing the set-ups and terms, we have chosen to call such cut and project sets \emph{substitutional}.

Thus, we have chosen the following definition for a Delone set \(\cps \subset \R^d\) (or tiling) to be substitutional with respect to some linear isomorphism \(L \colon \R^d \to \R^d\): there is a Delone set \(\cps' \subset \R^d\) such that 
\begin{itemize}
\item there is a local derivation from \(L(\cps')\) to \(\cps\) and
\item \(\cps'\) is locally indistinguishable from \(\cps\).
\end{itemize}
For more detail, see Section \ref{sec:substitutionalDef}. Our definition is very closely related to the standard notion of being generated by substitution for tilings, as we will demonstrate in Section \ref{sec:compare tilings}. When \(L\) is expansive, by recognisability \cite{Sol98, Wal25} (and since our cut and project sets are non-periodic and repetitive), the first condition is equivalent to the stronger condition that ``there is a local derivation from $L(\Lambda')$ to $\Lambda$, and vice versa'' i.e., \(L(\cps')\) is MLD to \(\cps\). Similarly, the second condition is equivalent to \(\cps\) and \(\cps'\) being locally isomorphic i.e., they are locally indistinguishable (see Definition \ref{def:local indistinguishable}) from each other.

We are now ready to state our main theorem. For the precise statement, see Theorem \ref{thm:main}.

\begin{infmain}
The cut and project sets determined by the data as above are substitutional with respect to some linear expansion \(L \colon \intl \to \intl\) if and only if there is a linear map \(M \colon \R^k \to \R^k\) satisfying: 
\begin{enumerate}
	\item \(M(\Gamma) = \Gamma\);
	\item \(M(\phy)=\phy\) with $M|_{\phy} = L\);
	\item \(M(\intl)=\intl\) and 
	\item a translate of $W$ is {\bf  constructable} from $M(W)$. In other words, $W$ can be written using finitely many translates of $M(W)$, and the set operations: intersection, union and complement. The allowed translations are the projected elements of $\Gamma$. 
\end{enumerate}
\end{infmain}

We note that our results throughout also allow for our windows to be multicoloured, even though we only speak of the single-coloured version in this introduction, for simplicity.

The first three conditions depend only on the cut and project scheme \(\mathcal{S} = (\R^k,\phy,\intl,\Gamma)\), which need to be preserved by the linear automorphism \(M\). The final condition says that \(W\) is also fixed by \(M\), up to `mutual constructability'. The notion of constructability is explained in more detail in Section \ref{sec:windowconstructDef}. It is shown that, when the above theorem is satisfied, we necessarily also have that \(M(W)\) is constructable from \(W\). This condition is slightly weaker (owing to our notion of substitutional being MLD invariant) to the window being defined by a certain type of GIFS (graph iterated function system). That this notion is related to the notion of a cut and project Delone (multi)set being substitutive is already known, see for instance \cite{Lee23}. Whilst the window being a special form of attractor for a GIFS is not necessary for our substitutional property, we show that it is closely related: up to MLD equivalence, the window must (up to a translation) be a GIFS of the special form described in Corollary \ref{cor:sub=>GIFS}.

This theorem has many corollaries. For instance, we quickly derive an explicit characterisation of the possible physical spaces \(\phy\) and lattices \(\Gamma\) which allow for substitutional cut and project sets. Namely, if \(\Gamma = \Z^2\) (which is without loss of generality, up to reparameterisation), it is necessary and sufficient for the slop of \(\phy\) to be a quadratic irrational, see Theorem \ref{thm:2-to-1}. We are also able to show that there are essentially (that is, up to reparametrisation) only countably many substitutional cut and project sets (with Euclidean total space), see Theorem \ref{thm:countably many}. We also demonstrate that variations of the above substitutionality definition and assumptions on the cut and project sets lead to quantifiable changes in the statements. These allow us to, for instance, classify exactly what expansive substitution maps \(\sub \colon \Omega \to \Omega\) the translational hull supports (Theorem \ref{thm:classification}) and which exact positionings of the lattice leads to LIDS \cite{AOI} patterns, loosely speaking, those which substitute to \emph{themselves} under an expansive substitution map.

We are able to derive a much simpler (and more easily checked) criteria on the window \(W\) when restricting to the case that \(W\) is polytopal. The main result here states the following; for a precise statement, see Theorem \ref{thm:polytopal}.

\begin{infpoly}
Let $\cps$ be a cut and project set determined by the data above, where $W$ is a polytope, and let \(L \colon \phy \to \phy\) be expansive. Then \(\cps\) is substitutional with respect to some power of \(L\) if and only if there is a linear map \(M \colon \R^k \to \R^k\) satisfying the conditions (1--3) from Theorem \ref{thm:main} (see above), and \(W\) satisfies: 
\begin{itemize}
\item the supporting hyperplanes of $W$ are invariant under some power of $M$, and 
\item the differences of the vertices of $W$ can be expressed as projected elements of $\tfrac 1M\Gamma$ for some $M\in \N$. 
\end{itemize}
\end{infpoly}

The techniques for proving these theorems build on a vast range of methodology from the literature. It is not too hard to show that if an $M$ such as in the statement of both of the main theorems exists, the corresponding cut and project set is substitutional. A key clue for this direction of the proof are analyses of local derivations and local isomorphisms in the cut and project context which correspond, respectively, to translations and reconstructions of the window. These techniques are particularly inspired by \cite[Chapter 7]{AOI} and \cite{BJS91}. The converse direction is less obvious. Finding such an $M$ when $\cps$ is substitutional, we need to lift the substitutional property of the cut and project set into the ambient space $\R^k$. Here, properties of local inflation-deflation symmetries (LIDS) on Delone sets play an important role. These techniques utilise \cite[Chapter 5]{AOI}. Throughout, the acceptance domain formalism is important, as is the stabiliser rank computations for pattern analysis in the polytopal case (based on e.g., \cite{KoiWal21, KoiWalII}). We summarise the crucial background findings on acceptance domains, local derivations and local isomorphisms, and LIDS, in Section \ref{sec:techniques for general windows}. Many of the consequences of the main theorem rely on the torus parametrisation of the translational hull \(\Omega\) (see Theorem \ref{thm:classification}), that is, the maximal equicontinuous factor of \(\Omega\), a tool used frequently for cut and project schemes see e.g., \cite{FHK02}.

\subsection*{How to read this article} In Section \ref{sec:setup} we give the basic definitions: Delone sets, local derivability and local isomorphisms, substitutionality, cut and project sets and constructability of windows, and the torus parametrisation of the local isomorphism class of cut and project sets. We focus on the case of Delone sets, but many of the definitions also apply for tilings. In Section \ref{sec:main results} we give the main results, including the main theorems, Theorems \ref{thm:main} and \ref{thm:polytopal}, together with examples and discussion to set them into context. Section \ref{sec:techniques for general windows} contains preliminary results on acceptance domains, LIDS and local derivability and isomorphisms, which are applied in Section \ref{sec:proof for general windows} to prove Theorem \ref{thm:main} and many of the accompanying results from Section \ref{sec:main results}. Section \ref{sec:proof polytopal} contains the proof for the main theorem for polytopal windows Theorem \ref{thm:polytopal}, which utilises techniques from \cite{KoiWalII} to describe the possible vertices and supporting hyperplanes of $W$. In Section \ref{sec:L-sub versus classical} we compare the definitions and results from the current work to the traditional definitions for tilings.

\section{Problem set-up, notation, and basic properties}\label{sec:setup}

In this section we define Delone (multi)sets and recall some basic properties, such as notions of {\it local derivability} and {\it indistinguishability}. Many of these properties have analogous definitions for tilings. In Section \ref{sec:L-sub versus classical} we will explain how to compare the properties defined in the Delone context, to the corresponding tiling properties in principle, and demonstrate that the results from this and later sections apply to tilings just as well as Delone sets. In particular, we show that the definition of a Delone set being {\it substitutional} given below in Subsection \ref{sec:substitutionalDef} is analogous to the classical definition of tilings or Delone (multi)sets being generated by substitution rules.

In Subsection \ref{sec:cutandprojectDef} we define Euclidean cut and project sets and go on to define and investigate other properties required to state our main theorems, in particular, in Subsections \ref{sec:windowconstructDef} and \ref{sec:torus parametrisation} we relate local derivability and local indistinguishability precisely to the cut and project set up.

\subsection{Delone (multi)sets}
Throughout, let \(E \cong \R^d\) be some \(d\)-dimensional vector space (\(d \geq 1\)), which will be the ambient space of our patterns. Fix a norm on \(E\). We let \(B(x,r)\) denote the closed ball of radius \(r\) centred at \(x \in E\). A \textbf{Delone set} is a subset \(\cps \subset E\) which is both \textbf{uniformly discrete} (there exists some \(r > 0\) so that every \(r\)-ball in \(E\) intersects at most one point of \(\cps\)) and \textbf{relatively dense} (there exists some \(R > 0\) so that every \(R\)-ball intersects at least one point of \(\cps\)).

Given \(\mathbf{\Lambda} = (\cps_1, \ldots, \cps_\ell) \subset E^\ell\), we denote its \textbf{support} by  \(\cps = \mathrm{supp}(\mathbf{\Lambda}) \coloneqq \bigcup_{i=1}^\ell \Lambda_i\). We call \(\mathbf{\Lambda}\) a \textbf{Delone multiset} if \(\cps\) is a Delone set\footnote{It is usual to further demand that each \(\Lambda_i\) is Delone, although here it seems more natural to not demand this, in analogy with tilings (where not all tile types necessarily appear relatively densely).}. Although a Delone multiset is a subset of the \(\ell\)-fold product of the ambient space \(E\), we visualise it as a finite collection of uniformly discrete sets in \(E\), where the points of \(\cps_i \subseteq \cps\) are those labelled with `colour' \(i\). No complications arise from points of different colours occupying the same point of \(E\), we can simply view that point as having multiple colours. We may define the translate \(\mathbf{\Lambda} + x = (\Lambda_i+x)_{i=1}^\ell\) of a Delone (multi)set, and say it is \textbf{non-periodic} if \(\cps \neq \cps + x\) for all \(x \in E \setminus \{\mathbf{0}\}\). Similarly, if \(L \colon E \to E\) is a linear automorphism, we may define \(L\cps = (L \Lambda_i)_{i=1}^\ell\). Henceforth, we will drop the bold font on \(\mathbf{\Lambda}\) and work with Delone sets and multisets interchangeably.

Given a Delone (multi)set \(\cps\), a patch is given by taking a finite selection of points in \(\cps\). More precisely, for a bounded subset \(U \subset E\), we may define the patch \(\cps\ang{U} \subset \cps\) by \((\cps\ang{U})_i \coloneqq \Lambda_i \cap U\). A \textbf{patch} in \(\cps\) is then simply a patch of the form \(\cps\ang{U}\) for bounded \(U \subset E\) i.e., a finite selection of points from \(\cps\) (recording their colours, in the multiset case).

In what follows we will use the word `pattern', to encompass both Delone (multi)sets and tilings. For the pattern \(\cP = \cps\) which, for simplicity, we assume is a Delone (multi)set for now, \(x \in E\) and \(r \geq 0\), we define
\begin{equation}\label{eq:patch notation}
\cP[x,r] \coloneqq \cP\ang{B(x,r)} - x .
\end{equation}
If \(x \in \mathrm{supp}(\cP)\), we in particular call \(\cP[x,r]\) an \textbf{\(r\)-patch}. The square brackets indicate that an equivalence relation has been introduced: since we `recentre', by translating \(x\) over the origin in Equation \ref{eq:patch notation}, \(\cP[x,r] = \cP[y,r]\) can be interpreted as `the \(r\)-patches centred at \(x\) and \(y\) in \(\cP\) are the same, modulo translation from \(x\) to \(y\)'. More generally, given a bounded subset \(U \subset E\), we define a \textbf{\(U\)-patch} to be a patch of the form \(\cP[x,U] \coloneqq \cP\ang{U+x}-x\), where again (for later convenience) we restrict to \(x \in \mathrm{supp}(\cP)\) (for Delone sets). Restricting \(r\)- or \(U\)-patches to be centred on \(\mathrm{supp}(\cP)\) is a convenience that allows the following simpler definitions (otherwise, we need to modify them slightly):

\begin{definition}
A pattern \(\cP\) has \textbf{finite local complexity} (\textbf{FLC}) if, for each \(r \geq 0\), there are only finitely many \(r\)-patches in \(\cP\). Equivalently, for all bounded \(U \subset E\), there are only finitely many \(U\)-patches.
\end{definition}

\begin{definition}
A pattern \(\cP\) is \textbf{repetitive} if, for each \(r \geq 0\), there is some \(R=R_r > 0\) so that, for each \(r\)-patch \(P\) in \(\cP\) and all \(x \in E\), there is some \(y \in B(x,R)\) with \(\cP[y,r] = P\). In other words, all patches appear relatively densely.
\end{definition}

We henceforth assume that all patterns are FLC. In fact, for most of the paper (from Section \ref{sec:main results} onwards), all patterns will also be repetitive.

\subsection{Local derivability and indistinguishability}
We will use the notions of local derivability and indistinguishability to give a highly practical (especially in the context of cut and project sets) definition of patterns being `substitutional'.

\begin{definition}\label{def:LD}
A pattern \(\cQ\) is \textbf{locally derivable} (\textbf{LD}) \textbf{from} a pattern \(\cP\) if there is some \textbf{derivation radius} \(c \geq 0\) satisfying the following: whenever \(\cP[x,c] = \cP[y,c]\) for $x,y\in E$, then \(\cQ[x,0] = \cQ[y,0]\), in which case we write \(\cP \LD \cQ\). If \(\cP\) is locally derivable from \(\cQ\) and vice versa, we call them \textbf{mutually locally derivable} and write \(\cP \MLD \cQ\).
\end{definition}

A pattern \(\cQ\) being LD from \(\cP\) may be thought of as saying that one may redecorate \(\cP\) to obtain \(\cQ\) using a rule that is `local', in the sense it can only inspect regions to a bounded radius in \(\cP\). One may also consider that, in this case, `\(\cP\) is detailed enough so that its patches, to some bounded radius, determine \(\cQ\)'.

\begin{remark}\label{rem:LD on supports or punctures}
It is easy to show that we can increase the `\(0\)' for the patches in \(\cQ\) in Definition \ref{def:LD} with any \(r \geq 0\), by increasing the derivation radius by \(r\). In particular, if \(\cQ = \cps\) is a Delone (multi)set and \(r > 0\) is such that all points of \(E\) are within distance \(r\) of a point of its support, clearly in Definition \ref{def:LD} it is necessary and sufficient for a local derivation \(\cP \LD \cQ\) that, for any \(x\) and \(y \in \Lambda\) (rather than all of \(E\)), there is some \(c \geq 0\) so that if \(\cP[x,c] = \cP[y,c]\) then \(\cQ[x,r] = \cQ[y,r]\).
\end{remark}

\begin{definition}\label{def:local indistinguishable}
A pattern \(\cP\) is \textbf{locally indistinguishable} (\textbf{LI}) \textbf{from} a pattern \(\cQ\), and we write \(\cP \LI \cQ\), if every patch of \(\cP\) appears, up to translation, in \(\cQ\). If \(\cP\) is locally indistinguishable from \(\cQ\) and vice versa, we call them \textbf{locally isomorphic} and write \(\cP \LIs \cQ\).
\end{definition}

\begin{remark}
Note that, in contrast to \cite{AOI}, we do not use the terms \emph{locally isomorphic to} and \emph{locally indistinguishable from} interchangeably, with the latter being only a one-sided relation here. This non-transitive version is very useful, though, for example in defining the hull of a pattern.
\end{remark}

The proof of the following is a simple exercise:

\begin{lemma}\label{lem:LI => properties}
If \(\cP \LI \cQ\) and \(\cQ\) is FLC then \(\cP\) is also FLC. If, additionally, \(\cQ\) is repetitive, then so is \(\cP\) and \(\cQ \LI \cP\), that is, \(\cP \LIs \cQ\).
\end{lemma}

\begin{definition}
For an FLC pattern \(\cP\) we let its (\textbf{translational, continuous}) \textbf{hull} \(\Omega_{\cP}\) be the space of patterns locally indistinguishable from \(\cP\), with the local topology.
\end{definition}

See \cite[Section 5.4]{AOI} for the definition of the local topology. In short, it may be derived from a uniformity, or a metric, which considers patterns to be `close' if they agree on a large ball about the origin, up to a small translation. One may equivalently take \(\Omega_{\cP}\) as the completion of the orbit \(\cP + E\) of translates of \(\cP\). It follows quickly from the definitions that \(\cP \LI \cQ\) if and only if \(\Omega_{\cP} \subseteq \Omega_{\cQ}\).

For a proof of the below, in a generalised setting, see \cite{Wal25}.

\begin{lemma}\label{lem: LD over hull}
If \(\cP \LD \cQ\), with derivation radius \(c\), then this may be uniquely extended to a surjective factor map \(f \colon \Omega_{\cP} \to \Omega_{\cQ}\) with \(f(\cP) = \cQ\). This map satisfies the following: given \(\cU\), \(\cV \in \Omega_{\cP}\), if \(\cU[x,c] = \cV[y,c]\) then \((f\cU)[x,0] = (f\cV)[y,0]\). In particular, each element of \(\Omega_{\cQ}\) is LD, with derivation radius \(c\), from an element of \(\Omega_{\cP}\).
\end{lemma}

More generally, for spaces of patterns \(\Omega\) and \(\Omega'\), we say that a map \(f \colon \Omega \to \Omega'\) is a \textbf{local derivation map} if it satisfies the conclusion of the above, that is, there exists some \(c \geq 0\) for which, for all \(\cU\), \(\cV \in \Omega\), if \(\cU[x,c] = \cV[y,c]\) then \(f(\cU)[x,0] = (f\cV)[y,0]\). In particular, for each \(\cU \in \Omega\), we have \(\cU \LD f(\cU)\). This is the geometric analogue of a `sliding block code' (between two orbit closures) from Symbolic Dynamics.

A local derivation map \(f \colon \Omega_{\cP} \to \Omega_{\cQ}\), as induced by \(\cP \LD \cQ\) from Lemma \ref{lem: LD over hull}, is injective (and thus is a homeomorphism) if and only if \(\cP \MLD \cQ\), in which case the inverse \(f^{-1}\) is the derivation map induced by \(\cQ \LD \cP\). For further details, see, for instance, \cite{Wal25}.

\subsection{Substitutional patterns}\label{sec:substitutionalDef}
We now use the notions of LD and LI to give an intrinsic definition of when a pattern is `substitutional' and derive some simple consequences. The idea of this definition is simple: when dealing with sufficiently nice patterns generated by an `inflate, replace' substitution rule, for any such pattern one always has a predecessor pattern, also generated by the substitution, for which applying substitution recovers the original pattern (in other words, the inflated `super-pattern' produces the original when we apply the replacement rule to it). Conversely, such a predecessor pattern determines a substitution rule generating these patterns. We make the existence of such a predecessor the defining property. In the below, the `inflate' is given by the map \(L\), and the `replace' (on the inflated structure) is given by the local derivation:

\begin{definition}\label{def:L-sub}
Let \(L \colon E \to E\) be a linear isomorphism and \(\cP\) a pattern in \(E\). We say that \(\cP\) is \textbf{\(L\)-substitutional} (or \textbf{\(L\)-sub}, for short) if there is another pattern \(\cP'\), called a \textbf{predecessor}, for which \(L\cP' \LD \cP\) and \(\cP' \LI \cP\). That is:
\begin{enumerate}
	\item \(\cP\) is locally derivable from \(L \cP'\) and
	\item \(\cP'\) is locally indistinguishable from \(\cP\).
\end{enumerate}
\end{definition}

We will refer to any linear automorphism \(L \colon E \to E\) on the ambient space of a pattern as an `inflation', whether or not it is expansive. However, in most places in this paper (except when considering symmetries), we have in mind that \(L\) is expansive:

\begin{definition}
A linear map \(L \colon E \to E\) is called \textbf{expansive} if every eigenvalue of \(L\) is strictly greater than \(1\) in modulus. Similarly, a linear map is called \textbf{contractive} if all its eigenvalues are strictly less than \(1\) in modulus.
\end{definition}

One may show that being expansive is equivalent to the property that, with respect to some appropriate norm on \(E\), there is some \(\lambda > 1\) with \(\|L(x)\| \geq \lambda\|x\|\) for all \(x \in E\). Given a fixed, expansive \(L\), we always assume such a \(\lambda > 1\) and norm are taken (the norm, otherwise, does not affect properties of relevance in this paper, so this may be done freely).

We say that \(\sub \colon \Omega \to \Omega\) is a \textbf{substitution map} on a space of patterns \(\Omega\) if \(\sub = S \circ L\), where \(L \colon \Omega \to L\Omega\) is the expansion map \(\cU \mapsto L\cU\) and \(S \colon L\Omega \to \Omega\) is a local derivation map. For an \(L\)-sub pattern \(\cP\), with predecessor \(\cP'\), by Lemma \ref{lem: LD over hull} we always have an induced \textbf{subdivision map} \(S \colon \Omega_{L\cP'} \to \Omega_{\cP}\), which is a surjective local derivation map. Often, we at least have \(\cP' \LIs \cP\)  (see Definition \ref{def:strongly L-sub} below, Theorem \ref{thm:recognisability}, but also generally in the repetitive case), so that \(\Omega_{L\cP'} = \Omega_{L\cP}\) and thus subdivision defines a map \(S \colon \Omega_{L\cP} \to \Omega_{\cP}\) i.e., we may apply subdivision to \emph{any} inflated pattern from the hull of \(\cP\). We may then define the (necessarily surjective) substitution map as \(\sub \coloneqq S \circ L \colon \Omega_{\cP} \to \Omega_{\cP}\) i.e., the map which `inflates, then subdivides'. If all elements of \(\Omega\) are locally isomorphic (i.e., in the repetitive case), clearly each \(\cP \in \Omega\) is \(L\)-sub, with predecessor any element of \(\sub^{-1}(\cP)\).

If one requires the predecessor to be literally equal to the original in Definition \ref{def:L-sub} we obtain the following notions, see \cite[Definition 5.15]{AOI} (and \cite{Sol98}, where the notion of a pseudo self-affine pattern is considered).

\begin{definition}
We call \(\cP\) \textbf{pseudo self-affine} (or say that \(\cP\) has a \textbf{local scaling property}), with respect to a linear map \(L \colon E \to E\), if \(L\cP \LD \cP\). We say that \(\cP\) has a \textbf{local inflation deflation symmetry} (or is \textbf{LIDS}, for short), with respect to \(L\), if \(L\cP \MLD \cP\).
\end{definition}

If, instead, we drop the local derivation, we define the following notion:

\begin{definition}
We say that a linear map \(L \colon E \to E\) is a \textbf{symmetry} of \(\cP\) if \(L\cP \LI \cP\).
\end{definition}

When we have such a symmetry, we obtain a continuous map \(L \colon \Omega_{\cP} \to \Omega_{L\cP} \subseteq \Omega_{\cP}\), given by \(\cU \mapsto L\cU\). For repetitive patterns, \(L\cP \LI \cP\) automatically implies that \(L\cP \LIs \cP\) (see Lemma \ref{lem:LI => properties}), and thus \(L\) defines a homeomorphism \(L \colon \Omega_{\cP} \to \Omega_{\cP}\), with inverse \(L^{-1}\), which is also a symmetry. Thus, \(L\) is in fact a `symmetry' of the LI-class of \(\cP\) (the hull of \(\cP\), in the repetitive case), rather than of just \(\cP\) itself, and the set of symmetries of \(\cP\) (or \(\Omega_{\cP}\)) is a group. Note, however, that a symmetry need not fix \(\cP\) itself, nor indeed any element of the hull, see Example \ref{ex:Penrose} below.

\begin{example}
Trivially, any pattern \(\cP\) is \(L\)-sub for \(L = \Id\). The notion is more relevant when \(L\) is expansive. Of course, \(L = \Id\) is also a symmetry of any pattern \(\cP\).
\end{example}

\begin{example}\label{ex:Penrose} We include here the familiar example of the Penrose tiling for comparison, and to demonstrate how the properties above relate to the set-up of a Penrose tiling. For definitions relating to this example, see \cite[Chapters 5-6]{AOI} and Section \ref{sec:L-sub versus classical}. 

Let \(\cP\) be a Penrose rhomb tiling, and \(L\) be inflation by the golden ratio. There is a {\it tile substitution}, using the inflation \(L\), which acts on tilings: for any Penrose tiling \(\cQ \in \Omega_{\cP}\), we have its substitution \(\sub(\cQ) \in \Omega_{\cP}\), given by first inflating \(\cQ\) by \(L\) then applying a rule which replaces inflated tiles by certain patches (for details, see \cite[Section 6.2]{AOI}). Any Penrose tiling has a `predecessor', so there is another Penrose tiling \(\cP' \in \Omega_{\cP}\) with \(\sub(\cP') = \cP\). In particular, \(L\cP' \LD \cP\). In fact, by recognisability (see below) and aperiodicity, \(L\cP' \MLD \cP\). That is, the supertiling \(L\cP'\) may be locally derived from \(\cP\) \cite{Sol98}.

A linear map \(L\) is a symmetry of \(\cP\) if \(L \in D_{10}\), although no Penrose tiling is literally fixed under rotation by \(2\pi/10\).
\end{example}

\begin{example}\label{ex:sub Delone}
Let \(\cps'\) be a Delone set in \(E\) and \(C \subset E\) a finite set. Suppose that \(\cps \coloneqq L\cps' + C = \{Lx+c \mid x \in \cps, c \in C\}\) is a Delone set and satisfies \(\cps' \LI \cps\). Clearly \(L\cps' \LD \cps\) (we just replace each point \(Lx \in L\cps'\) with the cluster \(Lx+C\), which is a local rule), so \(\cps\) is \(L\)-sub. When \(\cps \LIs \cps'\) (for instance, when \(\cps\) is repetitive), we will see below that we may define a substitution map \(\sub \colon \Omega \to \Omega\) on the hull \(\Omega = \Omega_{\cps}\). This is simply the substitution rule \(\sub(X) = LX + C\), for any Delone set \(X \in \Omega\). One may also consider a similar setup for Delone multisets, where the replacement clusters \(C_i\) (which are coloured) can depend on the colour(s) \(i\) present at each \(Lx \in \mathrm{supp}(L\cps')\).

The above is similar to the standard notion of a substitution Delone (multi)set \cite{LW03}, but is also essentially universal, as discussed in Section \ref{sec:L-sub versus classical}, up to enhancing by locally derivable labels.
\end{example}

\begin{example}\label{ex:periodic L-sub}
Let \(\cP\) be (a translate of) the periodic unit square lattice in \(E = \R^d\) and \(L(x) = \lambda x\) for \(\lambda \in \N\). Then \(\cP\) is \(L\)-sub, where we may take \(\cP' = \cP\) (or any translate of \(\cP\)). Indeed, trivially, \(\cP \LIs \cP'\). And, since \(\lambda \in \N\), it is not hard to see that \(L\cP' \LD \cP\); the local derivation is given by replacing each inflated square by its covering of \(\lambda^2\) unit squares (and perhaps applying a fixed translation), which is clearly a local derivation. So \(\cP\) is \(L\)-sub. However, \(L\cP'\) is not LD from \(\cP\) for \(\lambda \neq 1\). Indeed, the translational symmetries of \(L\cP\) are a strict subset of those of \(\cP\) in this case and the translational symmetries of two MLD patterns must be equal. The symmetries of \(\cP\) are given by the dihedral group \(D_4\).
\end{example}

It will be crucial in many of our proofs for the LD of Definition \ref{def:L-sub} to be an MLD, and that \(\cP\) and \(\cP'\) are in fact locally isomorphic:

\begin{definition}\label{def:strongly L-sub}
We say that \(\cP\) is \textbf{strongly \(L\)-substitutional} (or \textbf{\(L\)-Sub}, for short) if there exists some \(\cP'\) for which \(L(\cP') \MLD \cP \LIs \cP'\).
\end{definition}

Note that we use an upper case in Sub as shorthand for being \emph{strongly} \(L\)-substitutional. Recall that \(\cP\) is called \textbf{aperiodic} if its hull \(\Omega_{\cP}\) consists of only non-periodic patterns. In rather general circumstances, patterns which are \(L\)-sub are automatically \(L\)-Sub, when aperiodic:

\begin{theorem}[\cite{Wal25}]\label{thm:recognisability}
Let \(\cP\) be an \(L\)-sub FLC pattern in \(E\), where \(L \colon E \to E\) is expansive, with predecessor \(\cP'\). Then \(\cP \LIs \cP'\), so we have a substitution map \(\sub \colon \Omega \to \Omega\). Moreover, \(\cP\) is strongly \(L\)-Sub, that is \(L\cP' \MLD \cP\) or equivalently \(\sub \colon \Omega_{\cP} \to \Omega_{\cP}\) is injective, if and only if \(\cP\) is aperiodic.
\end{theorem}

The above is an extension of the classical result from \cite{Sol98}, on recognisability or unique decomposition in the geometric setting (see also \cite{Mos92,Mos96,BSTY19,BPRS25} for recognisability results in the symbolic setting). In fact, for the repetitive case at hand, it may be derived from a combination of the results from \cite{Sol98,Sol07,PriSol01} (along with showing that a power of substitution must have a fixed point). In short, aperiodicity is equivalent to there being a \emph{unique} inflated pattern \(L\cU' \in L\Omega\) subdividing to a given \(\cU \in \Omega\), in which case \(L\cU'\) can be derived from \(\cU\) using a local rule.

The upshot from all of the above is, that in all reasonable cases of FLC \(L\)-sub patterns (namely, when \(L\) is expansive, or \(\cP\) is repetitive), we do have a subdivision map \(S \colon \Omega_{L\cP} \to \Omega_{\cP}\) and substitution map \(\sub \colon \Omega_{\cP} \to \Omega_{\cP}\). Moreover, if we additionally assume aperiodicity and that \(L\) is expansive, we have that \(L\)-sub and \(L\)-Sub are equivalent. Similarly, in this case, being pseudo self-affine is equivalent to LIDS, with respect to \(L\). However, we preserve the distinctions, as the definition of \(L\)-sub seems the most natural extension of the idea of a tiling (or Delone set) being substitutional in the classical sense. Compare to Examples \ref{ex:sub Delone} and \ref{ex:periodic L-sub}; surely, we would like to consider translates of a periodic square lattice to be `substitutional', even though they are not strongly so. Moreover, the weaker definition arises also in considering non-expansive linear maps \(L\). Hence, the relationship between these ideas is best understood through the original definition for \(L\)-sub, even though in the proofs we will rely on the implied stronger property \(L\)-Sub. 

\begin{lemma}\label{lem:L-sub stable under LIs}
If \(\cP\) is \(L\)-sub, \(L\)-Sub, or has \(L\) as a symmetry, and \(\cP \LIs \cQ\) then \(\cQ\) has the same property.
\end{lemma}

\begin{proof}
If \(\cP \LIs \cQ\) then \(\Omega_{\cP} = \Omega_{\cQ}\). Subdivision \(S \colon \Omega_{L\cP'} \to \Omega_{\cP}\) is surjective, so there exists some \(L\cQ' \in \Omega_{L\cP'}\) with \(S(L\cQ') = \cQ\). Since \(S\) is a local derivation map, we have that \(L\cQ' \LD \cQ\). Moreover, since \(L\cQ' \in \Omega_{L\cP'}\), equivalently \(\cQ' \in \Omega_{\cP'}\), we have \(\cQ' \LI \cP' \LI \cP \LIs \cQ\), that is, \(\cQ' \LI \cQ\), as required. If \(\cP\) is also \(L\)-Sub, then subdivision defines an MLD \(L\cQ' \MLD \cQ\). It is not hard to see that, for patterns \(\cU \LI \cV\) and a local derivation map \(f\) (as in Lemma \ref{lem: LD over hull}), we also have \(f(\cU) \LI f(\cV)\). Indeed, the finite patches of \(f(\cU)\) are determined, under \(f\), by those of \(\cU\), and similarly for \(f(\cV)\), which contain at least all of those of \(f(\cU)\) since \(\cU \LI \cV\). We see that, if \(\cU \LIs \cV\), then \(f\cU \LIs f\cV\). In our case, starting with \(\cP \LIs \cQ\) and applying the local map \(S^{-1}\), we see that \(L\cP' \LIs L\cQ'\), equivalently, \(\cP' \LIs \cQ'\). Thus \(\cQ \LIs \cP \LIs \cP' \LIs \cQ'\) so \(\cQ \LIs \cQ'\), as required. The proof for \(L\) being a symmetry is similar (and simpler).
\end{proof}

In particular, for the hull of a repetitive pattern \(\cP\), all elements \(\cQ \in \Omega_{\cP}\) have \(\cQ \LIs \cP\) (Lemma \ref{lem:LI => properties}) and thus are also \(L\)-sub; of course, in this case, one may take any \(\cQ' \in \sub^{-1}(\cQ)\) as a predecessor for \(\cQ\). Iterating substitution, and choosing pre-images, we may thus always construct a hierarchy of patterns \(\cQ_i \in \Omega_{\cP}\)
\[
\cdots \LD L^2 \cQ_2 \LD L \cQ_1 \LD \cQ_0 = \cQ \LD L^{-1}\cQ_{-1} \LD L^{-2}\cQ_{-2} \LD \cdots
\]
where each LD is induced by subdivision \(S \colon \Omega_{L^i\cP} \to \Omega_{L^{i-1}\cP}\) in the obvious way, and we have \(\sub(\cQ_{i+1}) = \cQ_i\).

Being \(L\)-sub or $L$-Sub is also invariant under MLD equivalence. Note, however, that whilst being pseudo self-affine (or LIDS) are MLD-invariant, they are not LI invariant. On the other hand, symmetries are not MLD invariant.

\begin{lemma}\label{lem:L-sub stable under MLD}
If \(\cP\) is \(L\)-sub, \(L\)-Sub, LIDS or pseudo self-affine with respect to \(L\), and \(\cQ \MLD \cP\), then \(\cQ\) has the same property.
\end{lemma}

\begin{proof}
By Lemma \ref{lem: LD over hull}, we may extend the MLD to a derivation map of hulls \(g \colon \Omega_{\cP} \to \Omega_{\cQ}\), where \(\cU \MLD g\cU\) for all \(\cU \in \Omega_{\cP}\), where \(g\cP = \cQ\). Let \(\cQ' \coloneqq g(\cP') \in \Omega_{\cQ}\), so \(\cQ'\LI \cQ\) and \(\cP' \MLD \cQ'\). Since \(L\cQ' \MLD L\cP' \LD \cP \MLD \cQ\), we have \(L\cQ' \LD \cQ\), so \(\cQ\) is \(L\)-sub. If \(\cP\) is \(L\)-Sub, then we have \(\cP' \LIs \cP\). Using a similar argument to the last proof, we have \(\cQ' \LIs \cQ\) (as \(g\cP' = \cQ'\) and \(g\cP = \cQ\), so \(\cQ' \LIs \cP' \LIs \cP \LIs \cQ\)). Since \(L\cQ' \MLD L\cP' \MLD \cP \MLD \cQ\), we see that \(L\cQ' \MLD \cQ\) and \(\cQ\) is \(L\)-Sub, as required. The proof for LIDS/pseudo self-affine with respect to \(L\) are simpler, as in this case we just have that \(\cP' = \cP\) and thus we take \(\cQ' = \cQ\).
\end{proof}

\subsection{Euclidean cut and project sets}\label{sec:cutandprojectDef}
We now define the main object of study in this paper, cut and project sets. 

\begin{definition}
Let \(\tot\) be a \(k\)-dimensional vector space, called the {\bf total space}. We let \(\phy\) be a dimension \(d > 0\) subspace of \(\tot\), called the {\bf physical space}, and \(\intl\) be an \(n \coloneqq k-d > 0\) dimensional subspace of \(\tot\) that is complementary with \(\phy\), called the {\bf internal space} (so \(\phy \cap \intl = \{\mathbf{0}\}\) and \(\tot = \phy + \intl\)). Let \(\Gamma\) be a full rank lattice of \(\tot\).

Denote by \(\pi_\vee, \pi_<\) the projections to \(\phy\) and \(\intl\), respectively, with respect to the decomposition \(\tot=\intl + \phy\). For any subset \(A\subset \tot\), use the shorthands \(A_\vee \coloneqq \pi_\vee(A)\) and \(A_< \coloneqq \pi_<(A)\), and similarly \(a_\vee \coloneqq \pi_\vee(a)\) and \(a_< \coloneqq \pi_<(a)\) for any \(a \in \tot\). We always assume that \(\Gamma\) is irrationally positioned relative to the the physical and internal spaces, in the following sense:
\begin{itemize}
\item \(\pi_\vee|_{\Gamma} \colon \Gamma \to \phy\) is injective;
\item \(\pi_<|_{\Gamma} \colon \Gamma\to\intl\) is injective;
\item \(\Gamma_<\subset \intl\) is dense.
\end{itemize} 
The tuple \(\mathcal{S} = (\tot, \phy, \intl, \Gamma)\) is called a {\bf \(k\)-to-\(d\) cut and project scheme}. 

We define the {\bf star map} by \(e^\star \coloneqq \pi_<\circ \pi_\vee^{-1}(e)\) on elements \(e \in \Gamma_\vee\), which defines a (discontinuous) isomorphism \(\star \colon \Gamma_\vee \to \Gamma_<\).
\end{definition}

\begin{remark}
In the above definition, injectivity of $\pi_\vee|_\Gamma$ guarantees that the star map is well defined. Injectivity of \(\pi_<|_\Gamma\) is equivalent to non-periodicity of the resulting patterns.
\end{remark}

To produce a cut and project set from a scheme, we choose a window in its internal space:

\begin{definition} \label{def: cps}
A subset is \textbf{topologically regular} if it is the closure of its interior. A ({\bf multi-coloured}) {\bf window} of a cut and project scheme \(\mathcal{S}\) is a finite collection \(W = (W_i)_{i=1}^\ell\) of \textbf{window components} \(W_i \subset \intl\), which are compact and topologically regular.

We say that \(W\) is in {\bf non-singular position} if each \(\Gamma_< \cap \partial W_i = \emptyset\). In this case, we define the associated the {\bf cut and project set} as the Delone (multi)set \(\cps(\mathcal{S},W) = (\Lambda_i)_{i=1}^\ell\) where
\[
\Lambda_i \coloneqq \{e \in \Gamma_\vee \mid e^\star \in W_i\} = (\Gamma \cap (\phy + W_i))_\vee\ .
\]
More generally, we say that a Delone (multi)set \(\cps'\) is {\bf generated by the cut and project scheme} \(\mathcal{S}\) (and window \(W\)) if \(\cps' \in \Omega(\mathcal{S},W) \coloneqq \Omega_{\cps}\), where \(\cps = \cps(\mathcal{S},W+s)\) for some (equivalently, any) non-singular translation \(W+s\) of \(W\).

Given a window, we define the set \(\sing \subset \tot\) of \textbf{singular points} as those \(x \in \tot\) for which \(x_< \in \bigcup_{g \in \Gamma_<} \partial W + g\). We let \(\nonsing = \tot \setminus \sing\) denote the \textbf{non-singular points}. Note that, in particular, for \(x \in \intl\) we have that \(x \in \sing\) if and only if \(x\in \bigcup_{g \in \Gamma_<} \partial W + g\), otherwise \(x \in \nonsing\).
\end{definition}

Frequently one does not need the window to be multi-coloured. In this case, we take \(W\) itself as simply a compact, topologically regular subset of \(\intl\), and take \(\cps\) as an (uncoloured) Delone set. As we will see in Proposition \ref{prop:reducing to one colour}, reducing to the single-colour case often results in little loss of generality. However, allowing for multi-coloured windows can have practical advantage in applying our results directly to given examples, such as schemes that are naturally defined using partitions of the window but also those whose internal space is the product of a Euclidean space with a finite Abelian group, see Example \ref{ex:penrose cps} below.

Notice that, since we assume \(W\) is topologically regular, \(\nonsing\) is a countable union of open dense sets in \(\intl\) and is thus dense (similarly, \(\sing\) is also dense). This does not require that the window is \textbf{regular}, that is, the measure \(|\partial W| = 0\). In the regular case, \(|\sing_<| = 0\), so \(\nonsing_<\) is full measure in \(\intl\), and the same holds for the singular and non-singular points in \(\tot\).

\begin{remark}
Since we restrict to repetitive patterns, generating the hull using non-singular translates of the window, the \(W_i\) being topologically regular is not a particularly restrictive condition. Indeed, otherwise, replacing each \(W_i\) with \(\cl(\intr(W_i))\) does not change \(\Omega(\mathcal{S},W)\) but makes the window topologically regular. One advantage is that, for a fixed scheme \(\mathcal{S}\), there is now only one choice of window \(W\), up to translation, which produces \(\Omega(\mathcal{S},W)\), see Lemma \ref{lem: LI <=> equal windows}.
\end{remark}

\begin{remark} \label{rem: fixing ns cps}
All cut and project sets as defined above are repetitive (this follows by Lemma \ref{lem:acceptance domains indicate patches} by the same proof as \cite[Proposition 7.5]{AOI}) and aperiodic. So the elements of \(\Omega = \Omega(\mathcal{S},W)\) are all locally isomorphic. In particular, by Lemma \ref{lem:L-sub stable under LIs}, we have that all elements of \(\Omega\) are \(L\)-sub (or \(L\)-Sub) if and only if some particular element \(\cps \in \Omega\) is.
\end{remark}

\begin{example} \label{ex:penrose cps}
The vertices of the Penrose tiling are given as a cut and project set with an internal space \(\R^2 \times C_5\), where \(C_5 = \Z / 5\Z\), see \cite[Example 7.11]{AOI}. The window is a union of four pentagons, occupying the components of the internal space corresponding to non-trivial elements of \(C_5\). Considering each pentagon to have a different colour does not affect the MLD class. Lattice elements may be used to translate each window component to \(W_i - (\gamma_i)_<\), all in the trivial component \(\R^2 \times \{[0]\}\). The result will be equal to a Penrose pattern but where each of the sub-Delone sets of different colours are non-trivially translated relative to each other. However, doing this does not affect the MLD class. So (most of) our results apply directly to examples of this form: when the internal space is isomorphic to \(\R^n \times G\) for some finite Abelian group \(G\), a process like this can always be done.
\end{example}

\subsection{Constructability between windows}\label{sec:windowconstructDef}
Constructability between windows was introduced in \cite{BJS91} (also see \cite[Remark 7.6]{AOI}). We recall the definition here with some expository detail.

Informally, given a fixed cut and project scheme \(\mathcal{S}\), we na\"{i}vely say that a window \(W'\) is \emph{constructable} from \(W\) if we can give \(W'\) as a finite Boolean expression in terms of \(W\) and its translates under \(\Gamma_<\). However, this definition lacks important nuance, as the next example demonstrates:

Consider the windows \(W' = [0,1]^2\), the unit square in \(\intl = \R^2\), and \(W = W' \cup (W' + (1,0)) \cup (W' + (0,1))\). That is, \(W\) is an `L' shape. Suppose that \((1,0) \in \Gamma_<\). Then it is almost the case that \(W' = W \cap (W+(1,0))\), although not quite: this intersection contains an extra \(1\)-dimensional segment \([(1,1),(1,2)] \subset \partial W\). We could attempt to remove this segment using further Boolean operations on \(\Gamma_<\) translates of \(W\) and its complement, but it can well be the case that there is no element in \(\Gamma_<\) that would fully remove this segment without also removing interior points of \(W'\). However, once the windows are moved into non-singular position, no lattice points ever hit this segment, so the definition of constructability should be blind to such inconsequential features.

One route to a more precise definition is to consider equality between the windows only on the non-singular points. An alternative, neater way, and the one given in \cite{BJS91}, is to redefine the standard Boolean operations by always taking the closure of interior at each step, which removes awkward features such as in teh example above. Another advantage of this approach is that this makes acceptance domains (see Definition \ref{def:acceptance domains}) topologically regular and thus also windows. This is the definition we give underneath.

\begin{notation}\label{not:boolean operations}
For \(U \subset \intl\) we define \(\clin(U) \coloneqq \mathrm{cl}(\intr(U))\). We denote:
\[
U \wedge V \coloneqq \clin(U \cap V), \ U \vee V \coloneqq \clin(U \cup V), \ \neg U \coloneqq \clin(\intl \setminus U) .
\]
Let \(\mathscr{B}\) be the collection of topologically regular subsets of \(\intl\). It is not hard to check that \(\clin(U) \in \mathscr{B}\) for all subsets \(U\), that \(\clin = \Id\) on \(\mathscr{B}\) and that, for \(U\), \(V \in \mathscr{B}\), we have \(U \vee V = U \cup V\) and \(\neg U = \cl(\intl \setminus U)\). Moreover, one may check that this makes \(\mathscr{B} = (\mathscr{B},\wedge,\vee,\neg)\) a Boolean algebra.

Given \(W = (W_i)_{i=1}^\ell\), with each \(W_i \in \mathscr{B}\), let \(\mathscr{B}(W)\) be the sub-Boolean algebra of \(\mathscr{B}\) generated by \(\{W_i+g \mid i \in \{1,\ldots,\ell\}, g \in \Gamma_<\} \subset \mathscr{B}\). That is, \(\mathscr{B}(W)\) consists of all \(U \in \mathscr{B}\) which may be constructed as a finite expression in terms of the subsets \(W_i+g\), for \(g \in \Gamma_<\), and the operations \(\wedge\), \(\vee\) and \(\neg\).
\end{notation}

\begin{definition}\label{def: constructable}
Fix a cut and project scheme \(\mathcal{S}\) and two windows \(W\), \(W' \subset \intl\). We say that \(W'\) \textbf{is constructable from} \(W\) if \(\mathscr{B}(W') \subseteq \mathscr{B}(W)\). We call the windows \textbf{mutually constructable} if \(W'\) is constructable from \(W\) and vice versa, that is, \(\mathscr{B}(W) = \mathscr{B}(W')\).
\end{definition}

To unpack the above definition somewhat, for simplicity suppose that both \(W\) and \(W'\) are single-coloured (so we may just take \(W = W_1\), \(W' = W_1'\)). If \(W'\) is constructable from \(W\) then, by definition, \(\mathscr{B}(W') \subseteq \mathscr{B}(W)\). This is equivalent to \(W' \in \mathscr{B}(W)\): Indeed, \(W' \in \mathscr{B}(W')\), so if \(\mathscr{B}(W') \subseteq \mathscr{B}(W)\) then \(W' \in \mathscr{B}(W)\). Conversely, if \(W' \in \mathscr{B}(W)\) then also \(W'+g \in \mathscr{B}(W)\) for each \(g \in \Gamma_<\), since all terms of a given Boolean expression may be translated by \(g\) to provide another for the translate. Since these generate \(\mathscr{B}(W')\), we have \(\mathscr{B}(W') \subseteq \mathscr{B}(W)\).

Now, \(W' \in \mathscr{B}(W)\) simply means that \(W'\) can be given a finite expression in terms of the translates \(W+g\), for \(g \in \Gamma_<\), and the Boolean operations above. Using the De Morgan Laws (and that \(\vee = \cup\) on \(\mathscr{B}\)), such an expression may always be re-arranged to an equation of the form
\begin{equation}\label{eq:construction}
W' = \bigcup_{i=1}^m X_i \ \ \ \text{ for } \ \ \ X_i = \bigwedge_{j=1}^{m_i} W_{ij},
\end{equation}
where each \(W_{ij} = W+(g_{ij})_<\) or \(W_{ij} = \neg W+(g_{ij})_<\), where the \(g_{ij} \in \Gamma\). Intuitively, we may construct \(W'\) from intersections, unions and \(\Gamma_<\) translates of \(W\), up to negligible topological defects (more precisely, see also Lemma \ref{lem:Booleans same on NS}). The case of a multi-coloured window is similar: in this case, we now take each \(W_{ij} = W_m + (g_{ij})_<\) or \(W_{ij} = \neg W_m + (g_{ij})_<\) for some index \(m\) in Equation \ref{eq:construction}, to construct each window component \(W'_p\).

\subsection{The torus parametrisation}\label{sec:torus parametrisation}
For a hull \(\Omega = \Omega(\mathcal{S},W)\) of cut and project sets, we now recall the \textbf{torus parameterisation}, \(\tau \colon \Omega \to \mathbb{T}\), where \(\mathbb{T} = \tot / \Gamma\) . This map is defined as follows. For each \(x \in \nonsing\) (equivalently, such that \((\Gamma + x)_< \cap \partial W = \emptyset\)), we associate a \textbf{non-singular} cut and project set \(\cps_x \in \Omega\), defined by  
\[
\cps_x=\pi_\vee((\Gamma + x)\cap (\phy + W) )
\]
(and similarly for a coloured window, where \((x+\gamma)_\vee \in \cps_x\) is given colour \(i\) when \((x+\gamma)_< \in W_i\)). We then define \(\tau(\cps_x) \coloneqq [x]\in \mathbb T\). All other elements \(\cps \in \Omega\), which are called \textbf{singular}, may be written as a limit of non-singular patterns \(\cps_{x_n} \to \cps\), where \([x_n] \to [x]\) for some \([x] \in \mathbb{T}\), and we may set \(\tau(\cps) := [x]\), which can be shown to be well-defined. Clearly, for \(\gamma \in \Gamma\), we have \(x \in \nonsing\) if and only if \(x + \gamma \in \nonsing\). Hence, the singular and non-singular points \(x\in\tot\) correspond to elements of \(\mathbb{T}\), whose pre-images under \(\tau\) are the singular and non-singular elements of \(\Omega\), respectively. Note that \(\tau\) defines a surjective factor map, where \(\phy\) acts on \(\mathbb{T}\) by \(x \cdot [u] = [u+x]\). That is, \(\tau(\cps+x) = \tau(\cps) + [x]\) for any \(x \in \phy\). Moreover, it is not hard to see that, for \(s \in \intl\) with \(W-s\) non-singular, we have that \(\tau(\cps(\mathcal{S},W-s)) = \tau(\cps_s) = [s]\), since \((\gamma+s)_< \in W\) if and only if \(\gamma_< \in W-s\).

The idea of the above may be stated quite intuitively: if \(\tau(\cps) = [x]\), this means that \(\cps\) may be considered as the cut and project set given by cutting and projecting the translated lattice \(\Gamma + x\). One just needs to bear in mind that, if \([x]\) is singular, we need to be careful in selecting points from the translated lattice when they project to the boundary of the window. In particular, there are multiple different such cut and project sets for this translate, depending on which lattice points hitting the window boundary to include or exclude; in other words, depending on how we approach the singular pattern with non-singular ones. The non-singular points of \(\mathbb{T}\) are precisely those with single pre-images under \(\tau\), and the singular points are precisely those with multiple pre-images. For more detail, see e.g. \cite{Hun15}. 

\section{Characterisation of cut and project schemes that are substitutional} \label{sec:main results}

We can now state our main theorems. Recall Definition \ref{def:strongly L-sub} of \(L\)-Sub, shorthand for strongly \(L\)-substitutional. Recall that since cut and project sets are aperiodic, by Theorem \ref{thm:recognisability},  $L$-Sub is equivalent to the weaker  notion of being $L$-substitutional, when \(L\) is expansive. Further, by Lemma \ref{lem:L-sub stable under LIs}, for any element from the hull of locally isomorphic patterns is also $L$-sub. Hence, the result below also characterises when a hull of Euclidean cut and project sets satisfies \(L\)-sub.

\begin{theorem}\label{thm:main}
Let \(\mathcal{S}=(\tot, \phy, \intl, \Gamma)\) be a Euclidean cut and project scheme with window \(W\), and let $L: \phy \to \phy$ be linear. The cut and project sets in \(\Omega(\mathcal{S},W)\) are \(L\)-Sub if and only if there is a linear map \(M \colon \tot \to \tot\) satisfying the following:
\begin{enumerate}
	\item \(M(\Gamma) = \Gamma\);
	\item \(M(\phy) = \phy\) with \(L = M|_{\phy}\);
	\item \(M(\intl) = \intl\), and we define \(A \coloneqq M|_{\intl}\);
	\item a translate \(W-s\) of \(W\), for some \(s \in \intl\), is mutually constructable from \(A(W)\).
\end{enumerate}
If \(L\) is expansive, then \(A=M|_{\intl}\) is contractive and we may replace (4) with the weaker condition
\begin{enumerate}
	\item[(4a)] a translate \(W-s\) of \(W\), for some \(s \in \intl\), is constructable from \(A(W)\).
\end{enumerate}
\end{theorem}

The proof of this result will be given in Section \ref{sec:proof for general windows}. A few remarks are in order. 

\begin{remark}
\begin{itemize}
\item[(a)] Similar results on sufficiency of properties similar to (1--4) have appeared before in the literature \cite{Har04, KL23, AFHI11}; more of the work is in formally proving they are also necessary. 
\item[(b)] Note that Theorem \ref{thm:main} applies directly to single- and multi-coloured windows. Also, as discussed in Example \ref{ex:penrose cps}, it is indirectly but easily applied to the case of internal spaces with multiple Euclidean components, such as seen in the standard Penrose cut and project scheme, through a simple (MLD) conversion to the single-component setting considered here. 
\item[(c)] For readers interested in defining explicit substitution rules, note that the construction of (a translate of) \(W\) from \(A(W)\) determines the actual substitution rule on given patterns (how points of inflated patterns are replaced with clusters), as explained through Proposition \ref{prop:LD from construction} and Theorem \ref{thm: MLD <=> constructable}. See also Theorem \ref{thm:classification} below, which describes the corresponding substitution map on the hull. 
\item[(d)] We will see later that, the window $W$ of an $L$-Sub cut and project set need not be (even up to translation) the attractor of a GIFS. However, once we decompose the window into enough acceptance domains, see Definition \ref{def:acceptance domains}, it is an attractor of a GIFS, see Corollary \ref{cor:sub=>GIFS}.
\end{itemize}
\end{remark}

Given a Euclidean total space, lattice and physical space, it is natural to ask if this triple permits an internal space and window making (expansive) substitutional patterns. This is answered in the result below, where we call a linear automorphism \(M \colon \tot \to \tot\) \textbf{hyperbolic} if \(\tot = A+B\) is the sum of \(M\)-invariant subspaces, with \(M\) expansive on \(A\) and contractive on \(B\). Note that the condition that \([\phy]=\phy/\Gamma\) is dense in \(\mathbb{T}\) is equivalent to our requirement that \(\Gamma_<\) is dense in \(\intl\) in all cut and project schemes, but without reference to a particular \(\intl\).

\begin{theorem}\label{thm:existence}
Let \(\tot\), \(\phy\) and \(\Gamma\) be given, with \([\phy]\) dense in \(\mathbb{T} = \tot / \Gamma\) and \(\Gamma \cap \phy = \{\mathbf{0}\}\). The following are equivalent: 
\begin{itemize}
\item[(1)]There is an internal space $\intl$, expansive $L\colon \phy\to \phy$, and a window $W$ so that the cut and project set corresponding to $W$ and the scheme $(\tot, \phy, \intl, \Gamma)$ is $L$-Sub. 
\item[(2)] There exists a hyperbolic linear automorphism \(M \colon 
\tot \to \tot\), with \(M(\Gamma) = \Gamma\), for which \(\phy\) is 
\(M\)-invariant, and \(M|_{\phy}\) is expansive. 
\end{itemize}
In this case, $L=M|_{\phy}$ and $\intl$ is the contracting subspace of $M$ with $A=M|_{\intl}$. Moreover, there is a finite collection of lattice points $Z\subset \Gamma$ and such a window $W$, invariant in the sense that 
\begin{equation}\label{eq:Winvariant}
W=\bigcup_{z\in Z}A(W)+z_<. 
\end{equation}
\end{theorem}

\begin{remark}\label{rem:IFS}
Since $A$ is a contraction, the collection of maps $\{x\mapsto A(x)+z_<\mid z\in Z\}$ is known as an {\bf iterated function system} and for any choice of finite $Z$, formally the existence of a unique, non-empty, compact set $W$ satisfying \eqref{eq:Winvariant} follows from standard theory, see e.g. \cite[Theorem 9.1]{Falconer03}. A particular choice of $Z$, and hence $W$, is made in the proof of Theorem \ref{thm:existence} in Section \ref{sec:proof for general windows}. 
\end{remark}

The next theorem is a classification result, on all possible substitution maps \(\sub \colon \Omega \to \Omega\) on a hull of Euclidean cut and project sets. We assume from here on that all substitution maps are injective. Note that injectivity is automatic when \(L\) is expansive (see Theorem \ref{thm:recognisability}), which is our main interest. Injectivity is also automatic in the case of a substitution map defined by a pattern which is (strongly) \(L\)-Sub, with respect to a given predecessor: In this case, \(\sub = S \circ L\) where, for each \(L\cps \in L\Omega\), we have that \(L\cps \MLD S(L\cps) = \sub(\cps)\).

To state the theorem, we need to introduce some notation. Given an automorphism \(M \colon \tot \to \tot\) with \(M(\Gamma) = \Gamma\) (as provided by Theorem \ref{thm:main}), we let \(M_\mathbb{T} \colon \mathbb{T} \to \mathbb{T}\) denote the induced automorphism \(M_\mathbb{T}[x] \coloneqq [Mx]\) on the torus. Given also some \([u] \in \mathbb{T}\), we let \(T_u = T_u^M \colon \mathbb{T} \to \mathbb{T}\) denote the map \(T_u[x] = M_\mathbb{T}[x] + [u]\). Recall also the torus parametrisation $\tau: \Omega \to \mathbb T$ from Section \ref{sec:torus parametrisation}.

\begin{theorem}\label{thm:classification}
For a Euclidean cut and project scheme \(\mathcal{S}\) with window $W$, suppose that the hull \(\Omega(\mathcal{S},W)\) consists of \(L\)-Sub patterns. Let \(M \colon \tot \to \tot\) be an automorphism satisfying (1--3) of Theorem \ref{thm:main}. 

Then, every substitution map \(\sub \colon \Omega \to \Omega\) with inflation $L=M|_{\phy}$ satisfies \(\tau \circ \sub = T_u \circ \tau\) for some \([u] \in \mathbb{T}\). For all such \(u\), we have that \(\nonsing + u = \nonsing\) and \(\sing + u = \sing\). Further, exactly those $u\in \mathbb T$ for which \(W-u_<\) is constructable from \(A(W)\), correspond to a substitution \(\sub \colon \Omega \to \Omega\) with \(\tau \circ \sub = T_u \circ \tau\).

Finally, the element \([u] \in \mathbb{T}\) uniquely characterises \(\sub\), in the sense that if \(\tau \circ \sub' = T_u \circ \tau\) for another substitution map \(\sub' \colon \Omega \to \Omega\), then \(\sub' = \sub\).

\end{theorem}

\begin{remark}
\begin{itemize}
\item[(a)] Theorem \ref{thm:classification} means that every substitution map on a hull \(\Omega\) of a cut and project sets determines, and is determined by, a certain automorphism followed by a translation on the torus. These translations can also be characterised: they are the \(s \in \intl\) of (4) in Theorem \ref{thm:main}, up to an arbitrary translation in the physical space.

\item[(b)] When $L$ is expansive, the automorphism $M$ is unique and hyperbolic, and hence $T_u$ is also uniquely defined for all $[u]\in \mathbb T$. 

\item[(c)] Unlike in Theorem \ref{thm:main}, in Theorem \ref{thm:classification} the initial translation of the window affects the statement; changing it reparametrises the torus parametrisation \(\tau \colon \Omega \to \mathbb{T} = \tot / \Gamma\), see Section \ref{sec:torus parametrisation}. 

\end{itemize}
\end{remark}

The next results concerns a characterisation of the case where \(L\) is a symmetry. We require that when $L$ is a symmetry, then every  \(\cps \in \Omega\) also has \(L\cps \in \Omega\), without a further local derivation. Thus, the following is unsurprising:

\begin{theorem}\label{thm:symmetries}
Consider a Euclidean cut and project scheme \(\mathcal{S}\) and linear automorphism \(L \colon \phy \to \phy\). Then \(L\) is a symmetry of \(\Omega(\mathcal{S},W)\) if and only if there is a linear automorphism \(M \colon \tot \to \tot\) satisfying (1--3) of Theorem \ref{thm:main} and
\begin{enumerate}
	\item[(4b)] \(W\) is a translate of \(A(W)\).
\end{enumerate}
\end{theorem}

Note that, in contrast to the previous results in this section, the above is not invariant under MLD equivalence, unless it is an MLD equivalence which respects not just translations but also general rigid motions, see the notion of S-MLD equivalence in \cite{BJS91}.

\begin{example}\label{ex:flip symmetry}
Consider \(M \colon \tot \to \tot\), given by \(M(x) = -x\). Then clearly \(M\) satisfies (1--3) of Theorem \ref{thm:main}. Moreover, the restriction of \(M\) to \(\phy\) and \(\intl\) are also the maps \(x \mapsto -x\). Thus, \(L(x) = -x\) is a symmetry of the elements of \(\Omega(\mathcal{S},W)\) if and only if \(W\) is a translate of \(-W\). This holds, for instance, for all codimension \(1\) cut and project sets with interval windows. Indeed, these have arbitrarily large palindromic patches.
\end{example}

In this aperiodic setting, with $L$ expansive, recall that LIDS is equivalent to being pseudo self-affine. This  means that a pattern substitutes to itself, for some substitution rule with inflation \(L\). More precisely, for \(\cps \in \Omega\) to be LIDS with respect to \(L\), we require that \(L\cps \MLD \cps\) without any LI replacement step. Thus, in place of (4) in Theorem \ref{thm:main}, instead of dropping the construction we now drop the translation relating \(W\) to \(A(W)\). However, due to singular points, we generally need to use a sufficiently high power:

\begin{theorem}\label{thm:PSA}
Consider a hull of Euclidean cut and project sets \(\Omega = \Omega(\mathcal{S},W)\) which consists of \(L\)-Sub patterns. Thus (1--4) of Theorem \ref{thm:main} hold with respect to a map \(M \colon \tot \to \tot\). Moreover, suppose that the fibres of the torus parametrisation \(\tau \colon \Omega \to \mathbb{T}\) are all finite. Then, for each \(s \in \intl\), the following are equivalent:
\begin{enumerate}
	\item[(A)] for all \(v \in \phy\) and \(\cps \in \Omega\) with \(\tau(\cps) = [v+s]\), we have that \(\cps\) is LIDS with respect to some power \(L^m\) of the inflation;
	\item[(B)] there exists some \(v \in \phy\) and some \(\cps \in \Omega\) with \(\tau(\cps) = [v+s]\), for which \(\cps\) is LIDS with respect to some power \(L^m\) of the inflation;
	\item[(4c)] \(W-s\) is mutually constructable from \(A^p(W-s)\) for some \(p \in \N\).
\end{enumerate}
If \(L\) is expansive, we may replace (4c) above with the weaker condition
\begin{enumerate}
	\item[(4d)] \(W-s\) is constructable from \(A^p(W-s)\) for some \(p \in \N\).
\end{enumerate}
\end{theorem}

The condition that the torus parametrisation has finite fibres holds in many cases of interest, such as for the case discussed in the following subsection, of polytopal windows. It seems feasible that it should always hold when \(\Omega\) consists of \(L\)-sub patterns for \(L\) expansive. This condition is not needed in showing that (A) or (B) implies (4c), and we are able to take \(p=m\) in this case. However, in proving (A) from (4c) (or (4d), when \(L\) is expansive), this condition is needed and generally we require \(m\) to be a multiple of \(p\), when \(W-s\) is in singular position (and thus corresponds to multiple elements of \(\Omega\)). Consideration of simple examples, such as the Fibonacci cut and project scheme, show why such a higher power is necessary (see Example \ref{ex:Fibonacci}).

\subsection{Results for polytopal windows}\label{sec:results for polytopal windows}
In the case that \(W\) is polytopal, we may describe the constructability conditions in the results above in more detail. We begin with the definition of polytopal and the corresponding supporting hyperplanes. 

\begin{definition} \label{def:polytopal}
Let \(W \subset \intl\) be compact and topologically regular. An (\textbf{affine}) \textbf{hyperplane} \(H \subset \intl\) is translation of a codimension \(1\) subspace. We call a hyperplane \(H \subset \intl\) \textbf{supporting} if there exists some \(w_H \in H \cap W\) and an open neighbourhood \(U_H \subset \intl\) of \(w_H\) so that \(W \cap U_H = H^+ \cap U_H\), where \(H^+\) is a closed half-space for \(H\) (i.e., the closure of a connected component of \(\intl \setminus H\)). We call \(W\) \textbf{polytopal} if \(\partial W\) is contained in a finite collection \(\sH\) of supporting hyperplanes. A multi-coloured window \(W = (W_i)_{i=1}^\ell\) is called \textbf{polytopal} if each \(W_i\) is polytopal, in which case we let \(\sH\) be the union of the supporting hyperplanes of the \(W_i\).

For an affine hyperplane \(H\) we let \(V(H) \coloneqq H-H\), that is, the subspace translate of \(H\) containing the origin. We let \(\sH_0 \coloneqq V(\sH) = \{V(H) \mid H \in \sH\}\) denote the collection of \textbf{supporting subspaces} of the window i.e., the subspaces parallel to its faces.
\end{definition}

\begin{remark}
In Definition \ref{def:polytopal}, each \(H \cap W\), for \(H \in \sH\), is a `face' of the polytope (or possibly containing a union of faces). Then, selecting \(w_H\) in the interior of a face, in a small neighbourhood about \(w_H\) the polytopal window will be the same as a closed half-space. We note that any supporting hyperplane \(H\) of \(W_i \subset \intl\) must be a member of \(\sH\) or else the hyperplanes of \(\sH\) would not cover \(\partial W_i\). 

We leave it as an exercise to check that Definition \ref{def:polytopal} is equivalent to \(W\) being a finite union of convex polytopes, where a {\bf convex polytope} is a convex hull of finitely many points with a non-empty interior. In particular, the requirement that \(W\) is compact, equal to the closure of its interior and has boundary contained in a finite union of affine hyperplanes means that \(W\) is a finite union of convex polytopes given as closures of some connected components of \(\intl \setminus (\bigcup_{H \in \sH} H)\).
\end{remark}

Assume now that the necessary conditions (1--3) from Theorem \ref{thm:main} hold for some fixed linear automorphism \(M \colon \tot \to \tot\), that is:
\begin{enumerate}
	\item \(M(\Gamma) = \Gamma\);
	\item \(M(\phy) = \phy\), where we denote \(L \coloneqq M|_{\phy} \colon \phy \to \phy\);
	\item \(M(\intl) = \intl\), where we denote \(A \coloneqq M|_{\intl} \colon \intl \to \intl\).
\end{enumerate}
The existence of such an \(M\) concerns only the scheme \(\mathcal{S}\) but not the window. We now also assume that \(W\) is polytopal (possibly multi-coloured) and that \(L\) is expansive. Without loss of generality in the results below, we will assume that \(W\) has a vertex over the origin, which simplifies some statements. One immediate consequence of this assumption is that $W$ is in a singular position. 

\begin{definition}
We say that \(W\) is \textbf{face direction invariant} (with respect to \(A\)), or \textbf{FD-invariant} for short, if \(A(V) = V\) for each \(V \in \sH_0\).
\end{definition}

If we merely have \(A(V) \in \sH_0\) for all \(V \in \sH_0\), then we at least have FD-invariance for some power (since \(A\) is a linear automorphism so acts injectively on the set of subspaces of \(\intl\), and \(\sH_0\) is finite). As we will see below, our main result classifies those cut and project sets which are substitutional with respect to some power \(L^m\) of the expansion, and so it makes no difference to the characterisation whether we take $A$ or some $A^m$ FD-invariant.

\begin{definition}
We call elements of \(\Gamma_<\) \textbf{integral} points. Similarly, if \(g = q \gamma_<\) for \(q \in \Q\) and \(\gamma \in \Gamma\), that is, \(g \in \Q\Gamma_< = \bigcup_{N \in \N} \frac{1}{N} \Gamma_<\), then we call \(g\) \textbf{rational}.

The polytope \(W\) (or each polytope \(W_i\), in the multi-coloured case) has a decomposition into \(k\)-dimensional cells, including the vertices (\(0\)-cells) and edges (\(1\)-cells). Let \(\mathrm{Ver}(W)\) denote the set of vertices of \(W\) (in the multi-coloured case, this is the union of vertices of each \(W_i\)). We say that \(W\) is \textbf{rational} if \(v - v' \in \Q\Gamma_<\) for each \(v\), \(v' \in \mathrm{Ver}(W)\); that is, there is some \(N \in \N\) so that \(v - v' \in \frac{1}{N} \Gamma_<\) for each \(v\), \(v' \in \mathrm{Ver}(W)\).
\end{definition}

Of course, if \(W\) is connected, it is necessary and sufficient for rationality that each edge \(e = [v_1,v_2] \subset W\) of adjacent vertices in \(W\) has \(v_2 - v_1\) rational. Generally, since we assume that \(\mathbf{0} \in \mathrm{Ver}(W)\), being rational is equivalent to each element of \(\mathrm{Ver}(W)\) being rational. The above definition is given to be invariant under translating the window. 

Our main characterisation for substitutional polytopal cut and project schemes may now be stated:

\begin{theorem} \label{thm:polytopal}
Assume that conditions (1--3) of Theorem \ref{thm:main} hold, where \(L \colon \phy \to \phy\) is expansive, and that \(W\) is polytopal. Then the following are equivalent:
\begin{enumerate}
	\item there is some \(m \in \N\) so that all elements of \(\Omega(\mathcal{S},W)\) are \(L^m\)-sub;
	\item \(W\) is FD-invariant with respect to \(A^p\), for some \(p \in \N\), and \(W\) is rational.
\end{enumerate}
\end{theorem}

By inspecting the proofs, the powers \(m\) and \(p\) in the statement can be controlled. 

The result below generalises the well-known case of Sturmian words (see also Theorem \ref{thm:2-to-1}), in the 2--to--1 case to general dimensions and general polytopal windows. Recall the torus parametrisation $\tau\colon \Omega\to \mathbb T$ from Section \ref{sec:torus parametrisation}. 

\begin{theorem}\label{thm:polytopal PSA}
Assume that the elements of \(\Omega = \Omega(\mathcal{S},W)\), for polytopal \(W\), are \(L^m\)-sub for some \(m \in \N\). Then \(\cps \in \Omega(\mathcal{S},W)\) is LIDS (equivalently, pseudo self-affine, or fixed by a substitution map \(\sub \colon \Omega \to \Omega\)), with respect to some power \(L^p\) of the expansion, if and only if \(\tau(\cps) = [u]\) for \(u_< \in \Q\Gamma_<\).
\end{theorem}

Recall that we have chosen our polytopal windows to have a vertex over the origin in the above statement. So, when the hull admits an expansive substitution, those Delone (multi)sets \emph{fixed} by one are precisely those given by cutting and projecting translations of the lattice which are rational, in the internal space component. More specifically (see Proposition \ref{prop:polytopal PSA1}), that there is a single invertible substitution \(\sub \colon \Omega \to \Omega\) so that all LIDS points of \(\Omega\) are fixed under some power of \(\sub\), perhaps followed by a translation in the physical space.

\begin{example}\label{ex:Fibonacci}
Denote \(\varphi = \frac{1+\sqrt{5}}{2}\) and \(\varphi'\) its algebraic conjugate, \(\frac{1-\sqrt{5}}{2}\). Consider the Fibonacci cut and project scheme, with \(\tot = \R^2\) and \(\Gamma = \Z^2\) with \(\phy = \ang{(\varphi,1)}\) and \(\intl = \ang{(\varphi',1)}\) the expanding and contracting subspaces, respectively, of \(M = 
\begin{pmatrix}
1 & 1 \\
1 & 0
\end{pmatrix}
\),
 which restricts to \(L \colon \phy \to \phy\) given by \(x \mapsto \varphi x\) and \(A \colon \intl \to \intl\) given by \(x \mapsto \varphi' x\). A standard choice of window is \(W = [0,1]^2_<\), the projection to the internal space of the unit square, which is the line segment from \((e_1)_<\) to \((e_2)_<\). This does not have a vertex over the origin (as assumed in the above theorem), but it does after translating it to, say, \(W = [0,1]^2 -(e_1)_< \in \Gamma_<\), which now has the origin \(v_1 = \mathbf{0}\) as a vertex (the other vertex being \(v_2 = (e_2-e_1)_<\)). Since \((e_1)_< \in \Gamma_<\), this does not affect the statement of Theorem \ref{thm:polytopal PSA} (essentially, we just shift the torus parametrisation in the physical space component, taking a new origin of the lattice).

Thus, for example, the cut and project sets \(\cps\in \Omega\) with \(\tau(\cps) = [\mathbf{0}]\) should be LIDS. These are singular, since both \(\mathbf{0}\) and \(e_2-e_1 \in \Gamma\) project to \(\partial W\). Precisely, the singular patterns are given by excluding exactly one of these two, then cutting and projecting the remainder of the lattice. This results in two sequences of long and short intervals (denoted \(a\) and \(b\)) which are fixed under the second power of the standard Fibonacci substitution \(\sub\), given by \(a \mapsto ab\), \(b \mapsto a\), but not the first power.

Not all LIDS points are singular, since the singular points are integral translates (in the internal space) but from Theorem \ref{thm:polytopal PSA} we also have the rational translates. For example, consider the non-singular translates \(\Gamma + x_i\) for \(x_1 = c\), \(x_2 = c + (e_1)/2\) and \(x_3 = c + (e_2)/2\), where \(c = \frac{1}{2}(v_1+v_2) = \frac{1}{2}(e_2-e_1)_< \) is the centre of \(W\). The corresponding cut and project sets satisfy \(\tau(\cps_i) = [x_i]\), where \((x_i)_<\) is rational. Notice that the \(\Gamma + x_i\) are the only other translates of the lattice which are invariant under the map \(x \mapsto -x\) and, consequently, each \(\cps_i\) is palindromic. Taking the third power (no lower one is sufficient), one may check that \(W - s_i\) is constructable from \(A^3(W - s_i)\), where \(s_i = (x_i)_<\). And, indeed, each \(\cps_i\) is fixed under \(\sub^3\), up to a translation in the physical space. In fact, each \(\cps_i\) is given by the fixed point of the alternative (palindromic) substitution \(a \mapsto ababa\), \(b \mapsto aba\), where \(\cps_1\) has the origin on the boundary of two \(a\) tiles, \(\cps_2\) has the origin contained in the centre of an \(a\) tile and \(\cps_3\) has the origin in the centre of a \(b\) tile.
\end{example}

\subsection{Countability of substitutional Euclidean cut and project schemes}
In this, and the following subsections, we will derive some simple consequences of the theorems above.

\begin{definition}
We say that two (Euclidean) cut and project schemes \(\mathcal{S} = (\tot,\phy,\intl,\Gamma)\) and \(\mathcal{S}' = (\tot',\tot_\vee',\tot_<',\Gamma')\) are \textbf{isomorphic} if there is a linear isomorphism \(f \colon \tot \to \tot'\) for which \(f(\phy) = \tot_\vee'\), \(f(\intl) = \tot_<'\) and \(f(\Gamma) = \Gamma'\).
\end{definition}

Clearly, by choosing an isomorphism \(\Gamma \cong \Z^k\), every Euclidean cut and project scheme is isomorphic to \((\R^k,\phy,\intl,\Z^k)\) for some choice of physical and internal space. That is, modulo isomorphism, we may take \(\Gamma\) as the standard integer lattice. Note that if \(\mathcal{S}\) is isomorphic, under \(f\), to \(\mathcal{S}'\), then the elements of \(\Omega(\mathcal{S},W)\) are equal to those of \(\Omega(\mathcal{S'},f(W))\), up to a linear rescaling of the physical space (compare with Lemma \ref{lem:rescaling of cps is cps}). So, if we only wish to consider cut and project sets up to rescaling by a linear automorphism, we need only consider cut and project schemes up to isomorphism. Note that the property of being substitutional is preserved by linear automorphisms, and the map $L$ only changes to $f\circ L\circ f^{-1}$. 

In the proof of the next theorem we will need the following notion.  

\begin{definition}\label{def:GIFS}
For each  $i,j\in \{1, \dots, \ell\}$, let $f_{i,j}: \R^d\to \R^d$ be a contraction. Assume that the tuple $(W_1, \dots, W_\ell)\subset \R^{r\ell}$ of compact non-empty sets satisfies that for each $i=1, \dots, \ell$, 
\begin{equation}\label{eq:GIFSdef}
W_i=\bigcup_{j=1}^{\ell} f_{(i,j)}(W_j). 
\end{equation}
We then call the tuple $(W_1, \dots, W_\ell)$ {\bf attractor of a graph directed iterated function system (GIFS)}. Given the maps $f_{(i,j)}$, the existence of a unique tuple of compact, non-empty sets $(W_1, \dots, W_\ell)$ satisfying \eqref{eq:GIFSdef} is guaranteed by standard theory (see e.g. \cite[Theorem 1]{MW88}. 
\end{definition}

\begin{theorem}\label{thm:countably many}
Up to isomorphism, there are only a countable number of cut and project schemes \(\mathcal{S}\) for which there is some window \(W\) making the elements of \(\Omega(\mathcal{S},W)\) all \(L\)-sub for some expansive \(L\). For each such scheme \(\mathcal{S}\), there are only countably many such windows \(W\), up to translation.
\end{theorem}

\begin{proof}
Consider a scheme \(\mathcal{S}\), which we may assume, by the above, has \(\tot = \R^k\) and \(\Gamma = \Z^k\). By Theorem \ref{thm:existence}, we must have that \(\phy\) and \(\intl\) are the expanding and contracting subspaces, respectively, of some hyperbolic linear map \(M\) on \(\tot\), with \(M(\Z^k) = \Z^k\), which thus corresponds to a matrix in \(\mathrm{GL}(k,\Z)\), of which there are countably many.

Now fix a cut and project scheme \(\mathcal{S}\) satisfying the above, thus the necessary conditions (1--3) of Theorem \ref{thm:main} hold, with respect to \(M\). We also obtain a contraction \(A \colon \intl \to \intl\). In the polytopal case, since the possible set of locations \(\Q\Gamma_<\) of vertices is countable, it quickly follows that there are only countably many polytopal windows making \(\Omega\) consist of \(L\)-sub patterns. 

For general windows, things are slightly more complicated. However, we will see in Corollary \ref{cor:sub=>GIFS} that,  \(W\) is a finite union of subsets \(W_1'\), \ldots, \(W_\ell' \in \mathscr{B}(W)\), so that for every $i=1, \dots, \ell$, 
\[
W'_i=\bigcup_{j=1}^{\ell} A(W_j)+g_j, 
\]
where for each $j=1, \dots, \ell$, the map \(x \mapsto A(x) - g_j\) is a contraction, with \(g_j \in \Gamma_<\). Clearly, the tuple $(W_1', \dots, W_{\ell}')$ is then the attractor of a GIFS and hence uniquely characterised by the above equality. Since \(\Gamma_<\) is countable, there are only countably many such choices of maps, so there are only countably many attractors of GIFS of this form. It follows that there are only countably many such windows $W$, up to translation.
\end{proof}

\subsection{Characterisation of substitutional cut and project sets in \(2\)--to--\(1\) case}
Theorems \ref{thm:main} and \ref{thm:polytopal} give a full characteristion of those cut and project sets which are substitutional. However, they are both by nature somewhat non-constructive, relying on the statement `there exists a linear map \(M\)'. In this section we demonstrate that they can, however, be used to characterise substitutional cut and project sets in concrete terms. Versions of the following result are well-known (see, for example \cite[Chapter 6]{PF02}), but depending on the exact set-up, there are minor discrepancies in formulation. In particular, our notion of being substitutional allows for adding locally derivable labels when defining the substitution rule (see Example \ref{ex:non-Sturm sub}).

\begin{theorem}\label{thm:2-to-1}
Consider a cut and project scheme \(\mathcal{S}\) with \(k=2\), \(d=1\), \(\Gamma = \Z^2\) and \(\phy = \ang{(1,\alpha)}_\R\) for some irrational \(\alpha \in \R\). Then, if the cut and project sets of \(\Omega = \Omega(\mathcal{S},W)\) are \(L\)-sub with \(L\) expansive, for some \(\intl\) and \(W\), then \(\alpha\) must be a quadratic irrational. Conversely, if \(\alpha\) is a quadratic irrational, there is a unique choice of \(\intl\), and there are choices of windows \(W\), for which the elements of \(\Omega\) are \(L\)-Sub, with \(L\) expansive. In this latter case, \(W\) may be taken, for example, to be any union of closed intervals with rational endpoints. 

In this case, \(\cps \in \Omega\) is pseudo self-affine, fixed by some (expansive) substitution \(\sub \colon \Omega \to \Omega\), if and only if \(\tau[\cps] = [u]\) for \(u_< \in \Q\Gamma_<\).
\end{theorem}

\begin{proof}
Assume first that such choices \(W\), and \(\intl\) exist for which \(\cps\) is \(L\)-sub (equivalently, \(L\)-Sub, using that \(L\) is expansive). We know from Theorem \ref{thm:main} that there is a linear automorphism \(M\colon \R^2 \to \R^2\) which preserves \(\Gamma\) and \(\E_\vee\). In particular, \((1,\alpha)\) is an eigenvector for an integer matrix, which makes $\alpha$ a quadratic irrational.

Now assume that \(\alpha\) is a quadratic irrational, and denote \(\ell \coloneqq \phy\). By Theorem \ref{thm:2-to-1}, we need to find \(M \in \mathrm{GL}(2,\Z)\) with \(M\ell = \ell\), with \(M |_{\phy}\) expansive. This makes $\phy$ an eigendirection of $M$, and we will then choose $\intl$ to be another one. 

Without loss of generality, we may take \(0 < \alpha < 1\): Otherwise, it is not hard to see that there is some \(T \in \mathrm{GL}(2,\Z)\) for which \(\ell' = T\ell\) intersects the line segment between \((1,0)\) and \((1,1)\). In this case, \(\ell' = \ang{(1,\alpha')}\) where \(0 < \alpha' < 1\) remains a quadratic irrational (since it has a rational expression in terms of \(\alpha\)). The below argument constructs \(M \in \mathrm{GL}(2,\Z)\) with \(M\ell' = \ell'\) and so may use \(T^{-1}MT \in \mathrm{GL}(2,\Z)\) with \(T^{-1}MT\ell = \ell\).

The continued fraction expansion of \(\alpha\) is eventually periodic. Let us first investigate the case when \(\alpha\) has a periodic continued fraction expansion, so \(\alpha=[\overline{a_0, \dots, a_n}]\), where the notation \(\overline{a_0, \dots, a_n}\) means that the finite word \(a_0, \dots, a_n\) is repeated indefinitely. Set the notation
\[
A\coloneqq \left(\begin{array}{cc}
1&1\\
0&1
\end{array}\right)
\ \textrm{and} \
B\coloneqq \left(\begin{array}{cc}
0&1\\
1&0
\end{array}\right)\ , \ \text{ thus }
A^{a_j}B=\left(\begin{array}{cc}
a_j&1\\
1&0
\end{array}\right)\ .  
\]
It is a matter of a simple computation to check that if a line \(k\) has the slope \(s\) (i.e., \(k = \ang{(1,s)}_{\R}\)), then \(A^{a_j}B k\) has the slope \(\frac{1}{a_j+s}\). If we now let 
\[
M = A^{a_0} B A^{a_1} B \dots B A^{a_n} B 
\]
and apply the observation on the slopes above repeatedly to \(\alpha\), we see that the slope of \(M\ell\) reduces to
\[
\frac{1}{a_0+\frac{1}{a_1+\frac{1}{\ddots+\frac{1}{a_n+\alpha}}}}=\alpha\ ,
\]
using \(\alpha=[\overline{a_0, \dots, a_n}]\). Thus, \(M\ell = \ell\). The matrix \(M\) also preserves \(\Gamma\) as it is an integer matrix of determinant \(\pm 1\), and hence is the matrix required.

We now check the case for an \(\alpha\) with an eventually periodic continued fraction expansion \(\alpha = [c_0\dots c_m,\overline{a_0, \dots, a_n}]\). Denote by \(\beta = [\overline{a_0, \dots, a_n}]\) the irrational corresponding to the periodic part of the continued fraction expansion, and denote by \(\ell'\) the line of slope \(\beta\). Analogously to the argument above, applying \(P \coloneqq A^{c_0} B \dots B A^{c_m} B\) to \(\ell'\) produces a line of slope 
\[
\frac{1}{c_0+\frac{1}{c_1+\frac{1}{\ddots+\frac{1}{c_n+\beta}}}} = \alpha\ . 
\]
That is, \(P\ell' = \ell\). Hence, choosing \(M' = A^{a_0} B A^{a_1} B \dots B A^{a_n} B\) and \(M = PM'P^{-1}\) gives a matrix for which \(M(\ell) = PM'P^{-1}(\ell) = PM'(\ell') = P(\ell') = \ell\). Notice that \(M'\) and \(P\) are integer matrices of determinant \(\pm 1\) and in particular, \(M\) preserves \(\Gamma\). Since \(M(\ell) = \ell\), with the latter irrational (i.e., not intersecting non-zero lattice points), it is clear that the eigenvector \((1,\alpha)\) has a real, irrational eigenvalue \(\lambda\), in particular \(|\lambda| \neq 1\). Replacing \(M\) with \(M^{-1}\), if necessary, it follows that \(M |_{\phy}\) is expansive, so the second claim now follows from Theorem \ref{thm:existence}.

If we now take \(W \subset \intl\) to be polytopal, it is a union of closed intervals. The only supporting subspace is \(\{\mathbf{0}\}\) so the scheme is FD-invariant (as is always the case in codimension \(1\)). Thus, \(\Omega(\mathcal{S},W)\) consists of \(L^m\)-sub elements, for some \(m \in \N\), if and only if \(W\) is rational, by Theorem \ref{thm:polytopal} (in fact, if all the boundary points of \(W\) are integral, it is not hard to show that we may take \(m=1\) here, as clearly \(W\) is constructable from \(A(W)\) in this case). Equivalently, if we choose a translate of \(W\) so that some boundary point is rational, all vertices of \(W\) are rational.

The final statement, on the pseudo self-affine elements, follows directly from Theorem \ref{thm:polytopal PSA}, at least when we choose one vertex over the origin. If, instead, vertices are merely in \(\Q\Gamma_<\), then the result clearly still holds, as translating a boundary point over the origin just changes the torus parametrisation by addition of an element in \(\Q\Gamma_<\).
\end{proof}

In the Sturmian case (where the window is a projection of \([0,1]^2\)), similar but slightly different results have appeared before, that connect the property of being substitutive with the slope being quadratic irrational. For instance, in \cite[Theorem 4.1]{PT11}, the slope (taken with \(\alpha \in [0,1]\)) needs to be a Sturm number i.e., with Galois conjugate not in \([0,1]\). In \cite[Theorem 1.5(iii)]{Lamb:OTCPM}, the slope \(\mu > 1\) (or its inverse, if \(0 < \mu < 1\)) needs to be a reduced quadratic irrational, meaning that its Galois conjugate is in \((-1,0)\).

We get the more satisfying condition of the slope just needing to be a quadratic irrational. Part of the explanation of the difference is that, in our result (which also needs to respect the geometric structure and not merely the symbolic sequence), we may choose an appropriate projection to the physical space i.e., choose the internal space, which is determined as the contracting eigenspace of \(M \in \mathrm{GL}(2,\Z)\). In the result of \cite{Lamb:OTCPM}, the orthogonal projection is used instead. The below explains this discrepancy with an explicit example:

\begin{example}\label{ex:non-Sturm sub}
Consider the cut and project scheme \(\mathcal{S}\) with \(\tot = \R^2\), \(\Gamma = \Z^2\) and \(M\) the linear map defined by the matrix \(\begin{pmatrix}
2 & -1 \\
1 & -1
\end{pmatrix} \in \mathrm{GL}(2,\Z)
\). This has expanding subspace \(\phy = \ang{(\alpha,1)}_\R\) and contracting subspace \(\intl = \ang{(\alpha',1)}_\R\), with eigenvalues \(\varphi\) and \(\varphi'\), respectively, where \(\alpha = \frac{3+\sqrt{5}}{2}\), \(\alpha' = \frac{3-\sqrt{5}}{2}\), \(\varphi = \frac{1+\sqrt{5}}{2}\) and \(\varphi' = \frac{1-\sqrt{5}}{2}\).

\begin{figure}
\includegraphics[width=\textwidth,page=2]{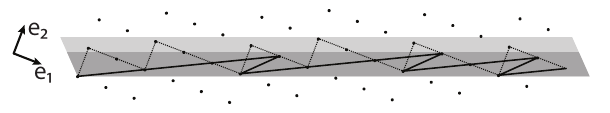}
\caption{The cut and project scheme arising from Example \ref{ex:non-Sturm sub}. }
\label{fig:backtracking}
\end{figure}

According to Theorem \ref{thm:2-to-1}, this cut and project scheme produces substitutional patterns, when equipped with a (for example) canonical window. As usual, for a strip defined by a (translate of) the canonical window, the lifted cut and project set defines a `staircase' of points connected by a sequence of moves by \(e_1 = (1,0)\) and \(e_2 = (0,1)\) (see Figure \ref{fig:backtracking}). However, the (inverse of) the slope, which is \(\alpha > 1\) has Galois conjugate \(\alpha' > 0\), so this is not a reduced quadratic irrational. As a result, unlike the standard case of a projected point set with a sequence of two gap lengths, we instead get three gap lengths:
\[
a = (1,0)_\vee \ , b = (1,1)_\vee \ \text{ and } c = (0,-1)_\vee .
\]
Note, in particular, that an \(e_2\) move in the total space projects down to a `backwards' move in the physical space. This backtracking behaviour is also seen upon applying the matrix \(M\), as also illustrated in Figure \ref{fig:backtracking}; compare to Figure \ref{fig:Fibonacci}.

Despite these complications, with the natural projection \(\pi_\vee\) (i.e., the one determined by \(\intl\)), the hull \(\Omega(\mathcal{S},W)\) consists of substitutional patterns. However, it is not quite so simple as these being generated, symbolically, by a substitution on the alphabet \(\{a,b,c\}\). We must enhance these labels (but not arbitrarily: we are constrained to use local rules, so the resulting patterns are still MLD and essentially equivalent to the original Delone sets).

Explicitly, this can be realised as follows. First, each Delone set is MLD to the tiling of \(a\), \(b\), \(c\) interval tiles, as above. Although already determined by their lengths, let each geometric tile also carry its label \(a\), \(b\) or \(c\). To these tilings, apply the local derivation which replaces the label of a \(b\) tile by a \(b'\), if the original \(b\) tile is followed by another \(b\). Then, replace any \(c\) tile that is followed by a \(b'\) with a \(c'\). Finally, for each \(b\) tile that is followed by a \(c\), replace its label with \(b^*\). On the enhanced alphabet \(\mathcal{A} = \{a,b,b',b^*,c,c'\}\), we define the following symbolic substitution:
\[
\sub =
a \mapsto b a \ , b \mapsto b c \ , b' \mapsto a \ , b^* \mapsto b c' \ , c \mapsto b' \ , c' \mapsto b .
\]
The geometric realisation of the above (which defines tile lengths using the left-eigenvector of the substitution matrix in the standard way) defines the same translational hull as those given by \(\Omega(\mathcal{S},W)\), after applying the MLD redecoration to tilings with labels in \(\mathcal{A}\). A useful translate of the window here is \(W = ([0,1]^2 + (3,1))_<\), since then the contracted interval window has the same lower endpoint as the original. That the above symbolic substitution defines the same patterns can then be checked by a simple (but lengthy) analysis of the acceptance domains for occurrences for each tile type, for which we omit the details.
\end{example}

\subsection{The codimension 1 case}
Some of the ideas above can be extended to the general codimension 1 case, that is, when \(k = d+1\) (and, without loss of generality, \(\Gamma = \Z^{d+1}\)). By Theorem \ref{thm:main}, if a Euclidean cut and project scheme generates \(L\)-sub Delone sets (with \(L\) expansive), we need a matrix \(M \in \mathrm{GL}(d+1,\Z)\) which has the physical space as an expanding subspace and \(\intl\) as the remaining (necessarily contractive) eigenline. So we need a matrix \(M \in \mathrm{GL}(d+1,\Z)\) with characteristic polynomial that has all but one of its eigenvalues strictly greater than one in modulus, where the projection of the expanding subspace to \(\mathbb{T} = \tot / \Gamma\) is dense. Conversely, given an integer polynomial \(p = x^{d+1} + c_dx^d + \cdots + c_1 x + c_0\), one can form its companion matrix, which has minimal and characteristic polynomial equal to \(p\) and is in \(\mathrm{GL}(d+1,\Z)\) if and only if \(c_0 = \pm 1\). If all but one of the roots are greater than \(1\) in absolute value, this has a \(d\)-dimensional expanding subspace and \(1\)-dimensional contracting subspace can be taken as, respectively, the physical and internal space. As in the \(2\)--to--\(1\) case, we can easily find a canonical window (an interval with endpoints in \(\Gamma_<\)) that results in the generated cut and project sets being \(L\)-sub.

\begin{figure}
\includegraphics[width=\textwidth,page=1]{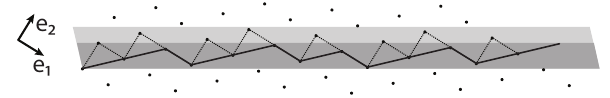}
\caption{The Fibonacci cut and project scheme.}
\label{fig:Fibonacci}
\end{figure}

\subsection{The dimension 1 case}
Taking the inverse of \(M\) in the above, we obtain a matrix with \(1\)-dimensional expanding line and \(d\)-dimensional contracting subspace, which may serve as the physical and internal spaces of a \(k\)--to--\(1\) cut and project scheme (provided \(\Gamma_<\) is dense in \(\intl\)). One may find a window \(W \subset \intl\) which is the attractor to an IFS, with \(W\) a union of \(\Gamma_<\) translated copies of \(A(W)\), so \(W\) is constructable from \(A(W)\), see the proof of Theorem \ref{thm:existence} later. For any compatible window, the resulting \(L\)-sub cut and project sets must be linearly repetitive, so have patch-counting functions which grow like \(\asymp r^d = r\) \cite{LP02}. This is impossible for a polytopal window when \(k \geq 3\), where the complexity must be \(\asymp r^{k-1}\) (see \cite{Wal24,KoiWalII,KoiWal21}), so the windows here are necessarily more complicated, with fractal-like boundary. Such windows are well-known, such as the Rauzy fractal arising from the Tribonacci substitution for \(k = 3\) \cite{Arnoux:PSARF}.

\subsection{Ammann--Beenker and variant examples}

Consider the Ammann--Beenker cut and project scheme \(\mathcal{S}\), for which we have \(\tot = \R^4\), \(\Gamma = \Z^4\) and \(\phy\), \(\intl\) are the expanding and contracting subspaces, respectively, of the matrix
\[
M \coloneqq 
\begin{pmatrix}
1  & 1 & 0 & -1 \\
1  & 1 & 1 & 0 \\
0  & 1 & 1 & 1 \\
-1 & 0 & 1 & 1
\end{pmatrix}
\in \mathrm{GL}(2,\Z) .
\]
We have that \(L \coloneqq M |_{\phy}\) and \(A \coloneqq M |_{\intl}\) are the similarities \(x \mapsto (1+\sqrt{2})x\) and \(x \mapsto (1-\sqrt{2})x\), respectively. For a polytopal window \(W\), FD-invariance is automatic (as it always is, for \(A\) a similarity \(x \mapsto \lambda x\)), so any rational window defines substitutional cut and project sets, for an appropriate power. 

The choice of window for the classical Ammann--Beenker is given by \([0,1]^4_<\), or any translate of this, which is clearly rational. In fact, for the translate \(W = [0,1]^4 - c\), where \(c\) is the centre of the octagon \(W\), it turns out that \(W\) is non-singular, so defines a literal cut and project set \(\cps = \cps(\mathcal{S},W)\). Let us now indicate how to construct this \(W\) from lattice translates of \(A(W)\): The vertices of \(A(W)\) are \(\Gamma_<\) translates of the corresponding vertices of \(W\),  and one may write
\[
W = \bigcup_{g \in X} A(W) - g_< .
\]
Here \(X\) is a set of \(25\) elements of \(\Gamma\): one given by \(\mathbf{0}\), providing the centred \(A(W)\) in the above union,  one $g_<$ for each vertex of $W$, and another two $g_<$ from each edge of $W$ such that the translates of $A(W)$ by these cover the edge. See Figure \ref{fig:translatesAmmann}. 

\begin{figure}
\includegraphics[width=\textwidth]{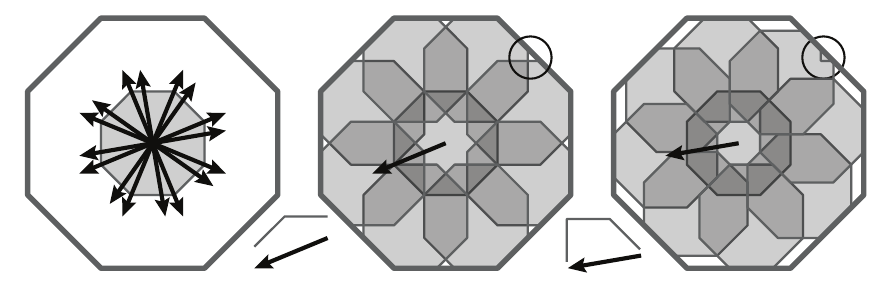}
\caption{This figure indicates how to choose the set $X_<$ so that $W$ is covered by the collection $\{A(W)+g_<\mid g_<\in X)_<$. Choosing the vertices of $W$ does not suffice, and a part of the edge remains uncovered, and we end up with a collection of 25 points.}
\label{fig:translatesAmmann}
\end{figure}

Since the make up of \(W\) as a union of $\Gamma_<$-translates of  $A(W)$ is so simple, we may express the substitution rule as \(\cps \mapsto L\cps + C\), where \(C = X_\vee\) i.e., we inflate and then replace all points with a fixed cluster. This cluster is invariant under the symmetry group \(D_8\) of the Ammann--Beenker patterns, by the choices we've made above.

One may easily come up with other symmetrical windows for the Ammann--Beenker scheme which are attractors to an IFS. For instance, there is a window \(W'\) satisfying
\[
W' = \bigcup_{g \in X'} A(W') - g_< ,
\]
where \(X' = \{\pm e_i\}\), resulting in a window with fractal boundary, given in Figure \ref{fig:AB variant cps}. We show a resulting cut and project set \(\cps\) in Figure \ref{fig:AB variant cps}.

\begin{figure}
\includegraphics[width=\textwidth]{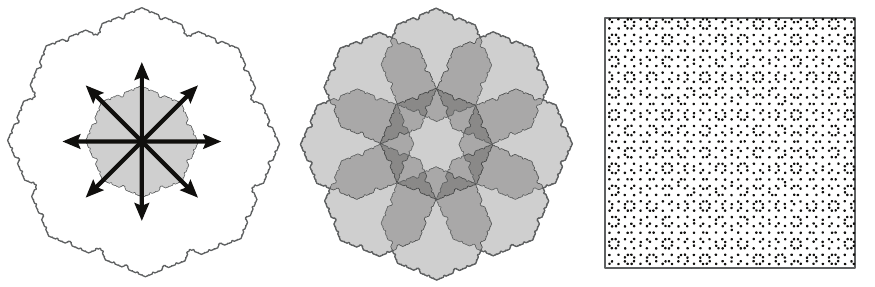}
\caption{In this figure an alternative Ammann-Beenker window with a fractal boundary is given. The resulting cut and project set is depicted on the right.}
\label{fig:AB variant cps}
\end{figure}

Generally, the substitution rules may not be so simple. Consider, for instance, a non-convex polytopal window, such as a star shape. It may be impossible to tile (a translate of) \(W\) with \(A(W)\) (or \(A^m(W)\) for any \(m \in \N\)) about vertices, due to protrusions in \(A(W)\). Nonetheless, if \(W\) is rational, it will still produce substitutional cut and project sets. The construction (see Equation \ref{eq:construction}) relating \(W\) to \(A(W)\) will need to involve also intersections (and maybe also complements), that is, the replacement clusters may need to depend on a finite set of neighbours. Up to adding this information through an MLD recolouring, though, it is not hard to see that one may always at least replace points with clusters only depending on the colour of a point.

\section{Acceptance domains, the limit translation module, and local derivation for cut and project sets}\label{sec:techniques for general windows}
In this section we introduce some constructions and tools needed to prove Theorems \ref{thm:main} and its corollaries from Section \ref{sec:main results}. Namely, we define and check basic properties of acceptance domains, the limit translation module, and local isomorphisms and local derivations, in the context of cut and project sets. In the next section, we then apply these techniques in the proof of Theorem \ref{thm:main} and related results in the case of general windows. Further tools specific to the polytopal window case will be given in the Section \ref{sec:proof polytopal}, where Theorems \ref{thm:polytopal} and \ref{thm:polytopal PSA} are proved.

\subsection{Indicator subsets and acceptance domains}
Fix a Euclidean cut and project scheme \(\mathcal{S}\) and, for now, non-singular window \(W \subset \intl\). Let \(\cps = \cps(\mathcal{S},W)\). Recall that, for bounded \(K \subset \phy\), we have \(K\)-patches \(\cps[x,K]\) for \(x \in \cps\), defined as \(\cps[x,K] = \cps\ang{K+x} - x\). For any \(K\)-patch \(P\) of \(\cps\), we have the \textbf{indicator set}
\begin{equation} \label{eq:patch indicator}
\cps_P \coloneqq \{x \in \cps \mid \cps[x,K] = P\}\ 
\end{equation}
i.e., the centres of occurrences of \(P\). This is a slight abuse of notation since \(\cps_P\) may depend on \(K\) as well as \(P\); one may like to think of \(P\) as a patch which records the existence but also the absence of any points potentially being in \(K\). By repetitivity, \(\cps_P\) is also a Delone set. In fact, it turns out to also be a cut and project set with window \(W_P\) constructable from \(W\), which we will call an acceptance domain.

Sometimes, especially for reducing complexity in practical applications, it can be convenient to work with a more flexible notion of patches. We define an \textbf{indicator} to be a set of pairs \(P = (P_i^\mathrm{in},P_i^\mathrm{out})_{i=1}^\ell\) of finite sets \(P_i^\mathrm{in}\), \(P_i^\mathrm{out} \subset \phy\), indexed over the colours \(i = \{1,\ldots,\ell\}\) of the window (in the single-coloured case, we just drop the indices \(i\)). We define its associated \textbf{indicator set} to be
\[
\cps_P \coloneqq \{x \in \cps \mid x+y \in \cps_i \text{ for each } y \in P_i^\mathrm{in} \text{ and } x+y' \notin \cps_i \text{ for each } y' \in P_i^\mathrm{out}\} .
\]
In other words, for \(x \in \cps\), we determine if \(x \in \cps_P\) simply by checking, for each colour \(i\), whether certain displacements from \(x\) are in \(\cps\) and colour \(i\), or that certain such displacements do not occur. Thus, it is obvious that \(\cps \LD \cps_P\). In the single-coloured case, we instead have just one pair \(P^\mathrm{in}\) and \(P^\mathrm{out}\).

Without loss of generality (since \(\cps - \cps \subset \Gamma_\vee\)), we may assume that each \(P_i^\mathrm{in} \subset \Gamma_\vee\). Indeed, otherwise, \(\cps_P = \emptyset\). Similarly, we assume that \(P_i^\mathrm{out} \subset \Gamma_\vee\) as, otherwise, \(x+y' \notin \cps\) for any \(x \in \cps\) so no relevant information is added. Moreover, if \(x \in \cps\), \(x+y \in \cps\) then \(x^\star \in W\) and \(x^\star + y^\star \in W\), thus \(x^\star \in W-y^\star \subset W-W\). Thus, we assume that \((P_i^\mathrm{in})^\star \subset W-W\), since otherwise \(\cps_P = \emptyset\). Similarly, we assume that \((P_i^\mathrm{out})^\star \subset W-W\), as otherwise \(x+y \notin \cps\) for any \(x \in \cps\) and again this adds no information in this case. In summary, we will assume that each \(P_i^\mathrm{in}\), \(P_i^\mathrm{out} \subset \mathscr{C}_\vee\) for the \textbf{cylinder} \(\mathscr{C} = \{\gamma \in \Gamma \mid \gamma_< \in W-W\}\) of lattice points contained in the thickened physical space \(\phy +(W-W)\). We denote the lifted versions of these sets as
\[
\Gamma_{P_i}^\mathrm{in} \coloneqq \{\gamma \in \Gamma \mid \gamma_\vee \in P_i^{\mathrm{in}}\}, \ \Gamma_{P_i}^\mathrm{out} \coloneqq \{\gamma \in \Gamma \mid \gamma_\vee \in P_i^{\mathrm{out}}\} \subset \mathscr{C} .
\]
Given our restrictions, to define an indicator it is thus equivalent to choose the finite lifted subsets of \(\mathscr{C}\). In the special case of a \(K\)-patch \(P = \cps[x,K] = (P_i)_{i=1}^\ell\), we define its associated indicator by
\[
\Gamma^\mathrm{in}_{P_i} = \{\gamma \in \Gamma_K \mid \gamma_\vee \in P_i\}, \ \Gamma^\mathrm{out}_{P_i} = \{\gamma \in \Gamma_K \mid \gamma_\vee \notin P_i\}, \ \text{ where } \Gamma_K \coloneqq \{ \gamma \in \mathscr{C} \mid \gamma_\vee \in K\} .
\]
By the above discussion, it is clear that the indicator set \(\cps_P\) results in the same definition as before, since \(P_i = (\Gamma_{P_i}^\mathrm{in})_\vee\) and \((\Gamma_{P_i}^\mathrm{out})_\vee\) excludes all possible points which \emph{could} be in \(P\), but are not. Note that \(\Gamma_K\) is a finite set, since it is contained in the bounded set \(K + (W-W)\), which is the disjoint union \(\Gamma_K = \Gamma^\mathrm{in}_{P_i} \cup \Gamma^\mathrm{out}_{P_i}\).

\begin{definition}\label{def:acceptance domains}
For an indicator (or patch) \(P\), defined by \(\Gamma_{P_i}^\mathrm{in}\), \(\Gamma_{P_i}^\mathrm{out} \subset \mathscr{C}\), we define its \textbf{acceptance domain} as
\[
W_P \coloneqq \bigwedge_{i=1}^\ell\left( \left(\bigwedge_{\gamma \in \Gamma_{P_i}^\mathrm{in}} W_i - \gamma_<\right) \wedge \left(\bigwedge_{\gamma \in \Gamma_{P_i}^\mathrm{out}} \neg W_i - \gamma_<\right)\right) \wedge W,
\]
using the Boolean operators \(\wedge\), \(\neg\) from Notation \ref{not:boolean operations}.
\end{definition}

\begin{remark}
In the single-coloured case, we have the slightly simpler formula
\[
W_P \coloneqq \left(\bigwedge_{\gamma \in \Gamma_P^\mathrm{in}} W - \gamma_<\right) \wedge \left(\bigwedge_{\gamma \in \Gamma_P^\mathrm{out}} \neg W - \gamma_<\right) \wedge W .
\]
We may remove the final \(\wedge W\) term if \(P\) is a \(K\)-patch with \(\mathbf{0} \in K\), since we always have \(\mathbf{0} \in \Gamma^\mathrm{in}_P\) in this case (so \(\wedge W\) is already included in the first intersection). In several proofs below, the single- and multi-coloured cases are proved very similarly, so to save on technical notation we will often only present the single-coloured case (or give brief notes on how the proofs are generalised), as the changes needed are otherwise obvious. However, our definitions and results will all be given for both the single- and multi-coloured cases.
\end{remark}

The lemma below gives the reassuring result that, in practice, we can essentially replace the Boolean operations, as used above, with the standard ones. The advantage of using their modifications, from Notation \ref{not:boolean operations}, is that the resulting acceptance domains are compact and topologically regular, so are themselves windows according to Definition \ref{def: cps} (and even polytopes, in the case \(W\) is polytopal).

\begin{lemma}\label{lem:Booleans same on NS}
If \(W' \in \mathscr{B}(W)\) then \(\partial W' \subseteq \nonsing\), explicitly, \(\partial W'\) is contained in the finite union of boundaries \(\partial W_{ij} = \partial(W)+g_{ij}\), where the \(g_{ij} \in \Gamma_<\) are those used to express \(W'\) in terms of \(W\), as in Equation \ref{eq:construction}. Let \(W''\) be given by replacing each operation \(\wedge\), \(\vee\) and \(\neg\) in this expression with its standard Boolean counterpart. Then, for \(x \in \nonsing\), we have that \(x \in W'\) if and only if \(x \in W''\).
\end{lemma}

\begin{proof}
It is not hard to show that, for any \(U\), \(V \in \mathscr{B}\), we have that \(\partial(U \vee V) \subseteq \partial(U) \cup \partial(V)\), \(\partial(U \wedge V) \subseteq \partial(U) \cup \partial(V)\) and \(\partial(\neg U) = \partial U\). Applying this observation repeatedly to an expression \(W' \in \mathscr{B}(W)\), as in Equation \ref{eq:construction}, we see that \(\partial W'\) is contained in the finite union of boundaries \(\partial W + g_{ij}\), where the \(g_{ij} \in \Gamma_<\) are the translations of \(W\) (or \(\neg W\)) used in the expression. Thus, in particular, \(\partial W' \subset \nonsing\). The same is true, by the same argument, for \(W''\). One may check that, for any \(U\), \(V \in \mathscr{B}\), we have that \(U \vee V\), \(U \wedge V\) and \(\neg U\) agree with \(U \cup V\), \(U \cap V\) and \(\intl \setminus U\), respectively, on all points \(x \notin \partial(U) \cup \partial(V)\). Thus, given \(x \in \nonsing\), it follows that \(x \in W'\) if and only if \(x \in W''\) since, inductively under each application of an operation, \(x\) does not belong to the boundaries of any of the intermediate sets.
\end{proof}

The following lemma shows that, when non-empty, indicator sets are themselves cut and project sets, with windows constructable from the original:

\begin{lemma}\label{lem:acceptance domains indicate patches}
Let \(\cps = \cps(\mathcal{S},W)\) with \(W\) non-singular and \(P\) be a patch (or, more generally, an indicator). Then \(x \in \cps_P\) if and only if \(x^\star \in W_P\), equivalently, \(\cps_P = \cps(\mathcal{S},W_P)\). In particular, \(\cps_P \neq \emptyset\) if and only if \(W_P \neq \emptyset\), in which case \(\cps_P\) is a cut and project set with scheme \(\mathcal{S}\) and window \(W_P\) constructable from \(W\).
\end{lemma}

\begin{proof}
We just present the single-coloured case, with the multi-coloured version being similar. By definition, \(\cps_P \subset \cps \subset \Gamma_\vee\). So suppose that \(x \in \Gamma_\vee\). Then \(x \in \cps\) if and only if \(x^\star \in W\). In this case, \(x \in \cps_P\) if and only if \(x+y \in \cps\) for each \(y = \gamma_\vee \in P_\mathrm{in} = (\Gamma_\mathrm{in}^P)_\vee\) and \(x+y \notin \cps\) for each \(y \in P_\mathrm{out} = (\Gamma_\mathrm{out}^P)_\vee\). Applying the star map (which is a homomorphism) these are equivalent to \(x^\star+y^\star \in W\) i.e., \(x^\star \in W-y^\star = W-\gamma_<\), and \(x^\star \notin W-\gamma_<\), respectively, that is (with \(x^\star \in W\)),
\[
x^\star \in \left(\bigcap_{\gamma \in \Gamma^P_\mathrm{in}} W - \gamma_<\right) \cap \left(\bigcap_{\gamma \in \Gamma^P_\mathrm{out}} \intl \setminus (W - \gamma_<) \right) \cap W .
\]
Since \(x^\star \in \nonsing\), by Lemma \ref{lem:Booleans same on NS} the standard Boolean operations may be replaced with the modified ones of Notation \ref{not:boolean operations}, so this is equivalent to \(x^\star \in W_P\), by Definition \ref{def:acceptance domains}, as required. By definition, \(W_P \in \mathscr{B}(W)\), so it is constructable from \(W\), and compact since \(W_P \subseteq W\). If \(\cps_P \neq \emptyset\) then, since \(\cps_P = \cps(\mathcal{S},W_P)\), we have \(W_P \neq \emptyset\). Conversely, if \(W_P \neq \emptyset\), we have \(\intr(W) \neq \emptyset\) (since \(W_P \in \mathscr{B}\)), so \(\cps_P \neq \emptyset\) by density of \(\Gamma_<\).
\end{proof}

\begin{remark}
The requirement that \(\cps = \cps(\mathcal{S},W)\) for \(W\) non-singular allows for the given simpler statement above. For a singular cut and project set, one needs to be careful in defining which acceptance domain \(x^\star\) should belong to when it hits boundary points, which will be determined by the way in which \(\cps\) is given as a limit of non-singular cut and project sets.
\end{remark}

\begin{notation}
For bounded \(K \subset \phy\), let \(\mathscr{A}_K = \mathscr{A}_K(W)\) denote the set of all acceptance domains of \(K\)-patches. In the case that \(K = B(\mathbf{0},r)\), that is, when \(P\) is an \(r\)-patch, we call \(W_P\) an \textbf{\(r\)-acceptance domain} and let \(\mathscr{A}_r\) denote the set of \(r\)-acceptance domains over all \(r\)-patches.
\end{notation}

\begin{lemma}\label{lem:ADs tile}
The set \(\mathscr{A}_K\) of acceptance domains of \(K\)-patches tile \(W\), in the sense their union is \(W\) and their interiors are disjoint.
\end{lemma}

\begin{proof}
By the previous lemma, \(x^\star \in W_P\) for any \(x \in \cps\), where \(P = \cps[x,K]\). Moreover, since \(x \in \cps\) if and only if \(x^\star \in W\), all points of \((\Gamma_<) \cap W\) belong to some \(K\)-acceptance domain. We have that \(\Gamma_<\) is dense in \(W\), so the acceptance domains cover a dense subset of \(W\). Moreover, the union of acceptance domains is closed, since \(\Gamma_K\) is clearly finite and each \(W_P\) is closed. So the union of the \(K\)-acceptance domains is closed, and thus equal to \(W\).

Now suppose that \(\intr(W_P) \cap \intr(W_{P'}) \neq \emptyset\) for distinct \(K\)-patches \(P\) and \(P'\). Since \(\Gamma_<\) is dense, there must exist \(\gamma \in \Gamma\) with \(\gamma_< \in \intr(W_P) \cap \intr(W_{P'})\). But then, by the previous lemma, \(P = \cps[\gamma_\vee,K] = P'\), so \(W_P = W_{P'}\), as required.
\end{proof}

The following is an elementary fact, and we omit the proof. 

\begin{lemma}\label{lem: arbitrarily small ADs}
For any \(\epsilon > 0\), for sufficiently large \(r > 0\) each \(W_P \in \mathscr{A}_r\) has diameter at most \(\epsilon\).
\end{lemma}

Our main theorems apply to Euclidean cut and project schemes where the windows can be multi-coloured. The below shows that there is no loss of generality, though (at least for those referencing the MLD-invariant notions of \(L\)-(S)sub, LIDS etc.), in working with single-coloured windows.

\begin{proposition}\label{prop:reducing to one colour}
For a Euclidean cut and project scheme \(\mathcal{S}\) and multi-coloured window \(W = (W_i)_{i=1}^\ell\), there is another, single-coloured window \(W'\) that is mutually constructable from \(W\). If each \(W_i\) is polytopal then so is \(W'\) and, in this case, \(W\) is rational and FD-invariant with respect to a linear map \(A \colon \intl \to \intl\) if and only if \(W'\) is.
\end{proposition}

\begin{proof}
Given a multi-coloured window \(W = (W_i)_{i=1}^\ell\), the window \((W_i + (\gamma_i)_<)_{i=1}^\ell\) is clearly mutually constructable from \(W\), for any choice of \(\gamma_i \in \Gamma\). In particular, by taking appropriate \(\gamma_i\), we may assume, up to MLD equivalence, that the \(W_i\) are all disjoint. If we then drop colours from \(W\), this clearly produces a single-coloured window \(W'\) that, by definition, is constructable from \(W\). Conversely, using Lemma \ref{lem: arbitrarily small ADs}, find an \(r\)-acceptance domain \(W'_P\) of \(W'\) with diameter \(\epsilon\), where \(2\epsilon\) is strictly smaller than the minimal distance between two of the colour components (which, as above, we assume is strictly positive). By density of \(\Gamma_<\), we may cover each \(W_i\) with translates \(W'_P + g\), for \(g \in Z_i\), for some finite set \(Z_i \subset \Gamma_<\), where each \((W'_P+g) \cap W_i \neq \emptyset\). By definition, the \(W'_P + g\) are constructable from \(W'\), thus so is each \(W_i = \bigcup_{g \in Z_i} (W' \cap (W'_P + g))\), as required.

By construction, since we only moved window components by lattice vectors in the above argument, it is clear that the properties of being polytopal, rational and FD-invariant are preserved.
\end{proof}

\subsection{The limit translation module for Euclidean cut and project sets}

The following definition can be found in \cite[Section 5.1.2]{AOI} and will be crucial in determining structure on the cut and project scheme from an \(L\)-sub cut and project set.

\begin{definition}
For a Delone (multi)set \(\Lambda\) and bounded \(K \subset E\), let
\[
D_K = D_K(\Lambda) \coloneqq \left\{(y-x) \in E \mid \Lambda[x,K] = \Lambda[y,K] \text{ for some } x, y \in \cps \right\} .
\]
That is, \(D_K\) is the set of displacements between translation-equivalent \(K\)-patches in \(\Lambda\). In the particular case of \(K = B(\mathbf{0},r)\), we write \(D_r \coloneqq D_K\). We let \(\Delta_K = \Delta_K(\Lambda) = \langle D_K \rangle_\Z\) be the \(\Z\)-module generated by \(D_K\), denoted \(\Delta_r\) for \(K = B(\mathbf{0},r)\). Clearly, for \(K \subseteq K'\), we have \(D_K \supseteq D_{K'}\) and thus \(\Delta_K \supseteq \Delta_{K'}\). Similarly, for any two bounded \(K\), \(K' \subset E\) we have that \(\Delta_{K \cup K'} \subseteq \Delta_K \cap \Delta_{K'}\). The {\bf limit translation module} \(\Delta = \Delta(\Lambda)\) is given by the inductive limit of this diagram of Abelian groups, that is
\[
\Delta = \bigcap_{\mathrm{bdd} K \subset E} \Delta_K = \bigcap_{r > 0} \Delta_r\ .
\]
\end{definition}

In the literature, it is more usual to define \(D_K\) by letting the \(x\), \(y \in E\), not just \(\cps\). However, this does not affect the limit translation module. Indeed, let \(D'_K\) denote this standard version of \(D_K\). Obviously, \(D'_K \supseteq D_K\). On the other hand, let \(r > 0\) be such that every \(x \in E\) is within \(r\) of some point of \(\cps\). Then, if \(\cps[x,K+B_r] = \cps[y,K+B_r]\), we also have \(\cps[x+z,K] = \cps[y+z,K]\) for some \(z \in B_r\), where \(x+z \in \cps\). Provided \(\mathbf{0} \in K\), we have \(y \in \cps\) too and thus, \(D'_{K+B_r} \subseteq D_K\).

The following result (see \cite{AOI}) is a simple exercise:

\begin{proposition} \label{prop: LTM invariant under LI and MLD}
If \(\Lambda' \LI \Lambda\) then \(\Delta(\Lambda') \leq \Delta(\Lambda)\). Similarly, if \(\Lambda \LD \Lambda'\) then \(\Delta(\Lambda) \leq \Delta(\Lambda')\). In particular, if \(\Lambda \LIs \Lambda'\) or \(\Lambda \MLD \Lambda'\) then \(\Delta(\Lambda) = \Delta(\Lambda')\).
\end{proposition}

We now make some simple observations regarding the return vectors and the limit translation module of an \(L\)-Sub pattern, and then a cut and project set. The below follows directly from Proposition \ref{prop: LTM invariant under LI and MLD} and Definition \ref{def:L-sub}:

\begin{corollary} \label{cor: sub preserves LTM}
Suppose that \(\cps\) is \(L\)-Sub. Then \(\Delta = L(\Delta)\).
\end{corollary}

Similarly, if \(\cps\) is \(L\)-sub we at least have \(L(\Delta) < \Delta\) when \(\cps\) is repetitive or \(L\) is expansive (as, by Theorem \ref{thm:recognisability}, we have that \(\cps \LIs \cps'\), its predecessor, in this case). The following slightly more specific version of this type of result will be used elsewhere:

\begin{lemma}\label{lem:action of expansion on return vectors}
Suppose that \(\cps\) is \(L\)-sub, with predecessor \(\cps' \LIs \cps\) (for instance, \(L\) is expansive or \(\cps\) is repetitive). Then there exists some \(\kappa \geq 0\) so that, for all bounded \(K\) and \(v \in D_K\), we have that \(L(v) \in D_{K'}\), provided that \(K' \subseteq L(K)_{-\kappa}\), where \(L(K)_{-\kappa}\) is the set of points whose \(\kappa\)-balls are fully contained in \(L(K)\).
\end{lemma}

\begin{proof}
We have a predecessor \(\cps'\) for which \(L(\cps') \LD \cps \LIs \cps'\). From \(\cps' \LIs \cps\), we have \(D_K = D_K(\cps) = D_K(\cps')\) (and similarly for \(K'\) satisfying the statement). So, given \(v \in D_K\), there exist \(x\), \(y \in \cps'\) with \(v = x-y\) and \(\cps'[x,K] = \cps'[y,K]\). Then, clearly, \((L\cps')[L(x),L(K)] = (L\cps')[L(y),L(K)]\). Let \(L\cps' \LD \cps\) have derivation radius \(c\). It follows that \(\cps[L(x),L(K)_{-c}] = \cps[L(x),L(K)_{-c}]\). Let \(r > 0\) be such that all points of the ambient space are within distance \(r\) of a point of \(\cps\). Then there exists some \(z\) with \(\|z\| \leq r\) with \(\cps[L(x)+z,L(K)_{-(c+r)}] = \cps[L(y)+z,L(K)_{-(c+r)}]\) and \(L(x)+z\), \(L(y)+z \in \cps\). Taking \(\kappa = c+r\), we see that \((L(x)+z)-(L(y)+z) = L(v) \in D_{L(K)_{-\kappa}}\), as required.
\end{proof}

\begin{corollary}\label{cor:action of L on return vectors}
There exists some \(C > 0\) so that, for sufficiently large \(r\), if \(v \in D_r\) then \(L(v) \in D_{Cr}\), where we may take \(C > 1\) if \(L\) is expansive. In particular, for some sufficiently large \(r > 0\), for all \(v \in D_r\) we have that \(L(v) \in D_0 = \cps - \cps\).
\end{corollary}

\begin{proof}
Applying Lemma \ref{lem:action of expansion on return vectors} to the case of \(K = B(\mathbf{0},r)\) leads directly to the result, noting that, for any \(\kappa > 0\), we have that \(B(\mathbf{0},Cr + \kappa) \subseteq L(B(\mathbf{0},r))\) for sufficiently large \(r\) and any \(C < \lambda\), where \(\lambda > 0\) is such that \(B(\mathbf{0},\lambda) \subseteq L(B(\mathbf{0},1))\). Thus, for sufficiently large \(r\), if \(v \in D_r\) then \(Lv \in D_{Cr} \subseteq D_0 = \cps - \cps\). In the case that \(L\) is expansive, we may take a norm with \(\lambda > 1\), and thus also \(C > 1\).
\end{proof}

Next, we consider return vectors for Euclidean cut and project sets. In this setting, we will see that \(\Delta(\cps) = \Gamma_\vee\), and in particular determines the total space dimension \(k\), from \(\Gamma_\vee \cong \Z^k\).

\begin{lemma}\label{lem:return vectors of cps}
Let \(\cps \in \Omega(\mathcal{S},W)\). For all bounded \(K \subset \phy\), we have that \(D_K \subset \Gamma_\vee\). For each \(r > 0\), there exists some \(\epsilon > 0\) such that, for all \(\gamma \in \Gamma\), if \(\gamma_\vee \in D_r\), we have that \(\|\gamma_<\| \leq \epsilon\). Conversely, for all \(\epsilon > 0\), there exists some \(r > 0\) so that, if \(\gamma \in \Gamma\) with \(\|\gamma_<\| \geq \epsilon\), then \(\|\gamma_\vee\| \notin D_r\).
\end{lemma}

\begin{proof}
Assume, without loss of generality, by repetitivity, \(\cps = \cps(\mathcal{S},W)\) for non-singular \(W\). Then, by definition, each \(D_K \subset \cps - \cps \subset \Gamma_\vee - \Gamma_\vee = \Gamma_\vee\). Take any \(r > 0\) and any \(r\)-acceptance domain \(W_P \in \mathscr{A}_r\). We have that \(W_P\) contains a ball \(B\), say of radius greater than some \(\epsilon > 0\). Take any \(\gamma \in \Gamma\) with \(\|\gamma_<\| \leq \epsilon\). Then, by density of \(\Gamma_<\), there exists some \(g \in \Gamma\) with \(g_< \in B\) and \(g_<+\gamma_< \in B\). Then \(g_\vee\), \(g_\vee + \gamma_\vee \in \cps\) and, by Lemma \ref{lem:acceptance domains indicate patches}, we have that \(\cps[g_\vee,r] = \cps[g_\vee + \gamma_\vee,r]\). Hence, \((g_\vee + \gamma_\vee) - g_\vee = \gamma_\vee \in D_r\), as required. Conversely, given any \(\epsilon > 0\), by Lemma \ref{lem: arbitrarily small ADs} we may choose \(r > 0\) so that all \(r\)-acceptance domains all have diameter no greater than \(\epsilon\). Suppose that \(\gamma \in \Gamma\) and take any \(x \in \cps\). Then if \(\cps[x,r] = \cps[x+\gamma_\vee,r] = P\), by Lemma \ref{lem:acceptance domains indicate patches}, \(x^\star\), \(x^\star + \gamma_\vee \in W_P\) for some \(r\)-patch \(P\) so \(\|\gamma_<\| = \|(x^\star + \gamma_\vee) - x^\star\| < \epsilon\), as required.
\end{proof}

The above shows that return vectors for `large' patches are projections of lattice vectors that lie close to the physical space. In combination with the following simple lemma, we see that the limit translation module is precisely \(\Gamma_\vee\).

\begin{lemma}\label{lem:small_basis}
For all \(\epsilon > 0\) there is a finite set \(A \subset \Gamma\) with \(\ang{A}_\Z = \Gamma\) and \(A_< \subset B(\mathbf{0},\epsilon)\).
\end{lemma}

\begin{proof}
Take \(b_1\), \ldots, \(b_{k-d} \in \Gamma\) so that each \((b_i)_< \in B(\mathbf{0},\epsilon/k)\) and, further, \(\{(b_i)_<\}_{i=1}^{k-d}\) spans a lattice in \(\intl\) which is \(\epsilon/k\)-dense. Clearly, since \(\Gamma_<\) is dense, this can be done. Choose further \(b_{k-d+1}'\), \ldots, \(b_k' \in \Gamma\) so that \(\{b_1,\ldots,b_{k-d},b_{k-d+1}',\ldots,b_k'\}\) is linearly independent. Since \(\ang{b_1,\ldots,b_{k-d}}_\Z\) is \(\epsilon/k\)-dense, for each \(b_i'\) there is some \(b_i = b_i' + (c_1^i b_1 + \cdots + c_{k-d}^i b_{k-d})\), for \(c_j^i \in \Z\), with \(b_i \in B(\mathbf{0},\epsilon/k)\).

Thus, \(b = \{b_1,\ldots,b_k\}\) is a basis of \(\tot\) of elements in \(\Gamma\) with the property that each \((b_i)_< \in B(\mathbf{0},\epsilon/k)\). However, \(b\) may only have \(\Z\)-span \(\ang{b}_\Z = G \lneq \Gamma\) with index \(N > 1\). In this case, the parallelotope \(P\) generated by \(b\) is a fundamental domain for \(G\), so we may find \(g_1\), \ldots, \(g_N \in P\) such that, with \(A \coloneqq b \cup \{g_1,\ldots,g_N\}\), we have \(\ang{A}_\Z = \Gamma\). Since each \((b_i)_< \in B(\mathbf{0},\epsilon/k)\), we have \(P_< \subseteq B(\mathbf{0},\epsilon)\) and, in particular, \(a_< \in B(\mathbf{0},\epsilon)\) for each \(a \in A\), as required.
\end{proof}

\begin{corollary}\label{cor:LTM is projected lattice}
Let \(\cps \in \Omega(\mathcal{S},W)\). Then \(\Delta(\cps) = \Gamma_\vee\).
\end{corollary}

\begin{proof}
By Lemma \ref{lem:return vectors of cps}, we have \(\Delta < \Gamma_\vee\). Conversely, by the same lemma, for any \(r > 0\) there exists some \(\epsilon > 0\) so that \(\gamma_\vee \in D_K\) for all \(\gamma \in \Gamma\) with \(\|\gamma_<\| < \epsilon\). By Lemma \ref{lem:small_basis}, such \(\gamma\) generate all of \(\Gamma\) and thus \(\Gamma_\vee = \ang{D_r}_\Z\). Since this holds for all \(r > 0\), we have that \(\Gamma_\vee = \ang{D_r}_\Z\).
\end{proof}

Combining the above with Corollary \ref{cor: sub preserves LTM}, we obtain:

\begin{corollary}\label{cor:projected lattice invariant under L}
Let \(\Lambda \in \Omega(\mathcal{S},W)\) be \(L\)-Sub. Then \(L(\Gamma_\vee) = \Gamma_\vee\).
\end{corollary}

\subsection{Local isomorphism and local derivations for cut and project sets}
The concepts of local isomorphism and local derivations have simple descriptions in the context of cut and project sets. First, we consider those which are locally isomorphic. Note that, given our definitions, \(\cps \LI \cps'\) if and only if \(\cps \LIs \cps'\) for cut and project sets \(\cps\) and \(\cps'\), since all our cut and project sets are repetitive.

\begin{lemma} \label{lem: LI <=> equal windows}
Let \(\mathcal{S}\) be a cut and project scheme, \(W\) and \(W'\) be windows, \(\cps \in \Omega(\mathcal{S},W)\) and \(\cps' \in \Omega(\mathcal{S},W')\). Then \(\cps \LIs \cps'\) if and only if \(W = W'+s\) for some \(s \in \intl\).
\end{lemma}

\begin{proof}
It also follows directly from our definitions that translating the window does not affect \(\Omega(\mathcal{S},W)\), so clearly if \(W = W'+s\) then \(\Omega(\mathcal{S},W) = \Omega(\mathcal{S},W')\) and all these cut and project sets are locally isomorphic. Conversely, suppose that \(\cps \LIs \cps'\). By repetitivity, without loss of generality \(\cps = \cps(\mathcal{S},W)\) and \(\cps' = \cps(\mathcal{S},W')\) with \(W\) and \(W'\) non-singular and containing the origin. By density of \(\Gamma_<\) and \(W\) being compact and equal to the closure of its interior, for arbitrarily small \(\epsilon > 0\), there is a sufficiently large \(r > 0\) so that \(P_r = \cps\ang{B(\mathbf{0},r)} \subset \cps\) is such that the \(\epsilon\)-neighbourhood \(N_\epsilon\) of \(P_<\) covers \(W\). Since \(P\) also appears in \(\Lambda'\), up to translation, the same properties hold for \(W' + s_\epsilon\) for some \(s_\epsilon = \gamma_< \in \Gamma_<\). Thus, \(P_< \subset W, W+s_\epsilon \subset N_\epsilon\), where \(N_\epsilon\) is an \(\epsilon\)-neighbourhood of \(P_<\). It follows that the Hausdorff distance between \(W\) and \(W'+s_\epsilon\) converges to \(0\) as \(\epsilon \to 0\), and it easily follows that \(s_\epsilon \to s\) for some \(s \in \intl\), as \(\epsilon \to 0\), with \(W = W'+s\).
\end{proof}

We now want to show that constructability is necessary and sufficient for local derivability. However, it is also useful to see how a local derivation can be explicitly described by a construction between windows. We do this through the following lemma and proposition. In these results, if \(W\) is multi-coloured, we let \(W\) also denote the union of its colour components, where in each instance this should be clear, given the context.

\begin{lemma}\label{lem:constructables from ADs}
Let \(A \in \mathscr{B}(W)\), that is, \(A\) is constructable from \(W\), where we assume that \(A \subseteq W\). Then \(A = W_P\) for some indicator \(P\).
\end{lemma}

\begin{proof}
Since \(A \in \mathscr{B}(W)\), we may express \(A\) in terms of \(\Gamma_<\) translates of \(W\), as in Equation \ref{eq:construction}. Namely, we may write \(A\) as a union of sets \(X\) of the form
\[
X = \bigwedge_{j=1}^m W_j ,
\]
where each \(W_j = W+(g_j)_<\) or \(\neg W+(g_j)_<\) for \(g_j \in \Gamma\). Without loss of generality (since we assume that \(A \subseteq W\) and thus each \(X \subseteq W\)) we may include \(-\wedge W\) into the above. Thus, similarly to in the discussion to the definition of indicators, we may assume without loss of generality that each \(g_j\) belongs to the cylinder \(\mathscr{C}\). Hence, defining \(P\) by letting \(\Gamma^\mathrm{in}_P\) be the set of \(g_j\) with \(W_j = W+(g_j)_<\) and \(\Gamma^\mathrm{out}_P\) be the set of \(g_j\) with \(W_j = \neg W + (g_j)_<\), clearly \(X = W_P\), as required. The case of multiple colours is proved analogously.
\end{proof}

\begin{proposition}\label{prop:LD from construction}
Consider a cut and project scheme \(\mathcal{S}\) and two windows \(W\) and \(W'\) for it, in non-singular position. Denote \(\cps \coloneqq \cps(\mathcal{S},W)\) and \(\cps' \coloneqq \cps(\mathcal{S},W')\). Suppose that we may write each component of \(W'\) as
\begin{equation}\label{eq:LD from construction}
W'_j = \bigcup_{i=1}^m A(i) + Y(i)_< ,
\end{equation}
where each \(Y(i) \subset \Gamma\) is finite and each \(A(i) \subseteq W\) with \(A(i) \in \mathscr{B}(W)\), that is, \(A(i)\) is constructable from \(W\). Then \(\cps \LD \cps'\). The local derivation may be described, explicitly, as follows. Express each \(A(i)\) as a union \(A(i) = \bigcup_{P \in P(i)} W_P\) of acceptance domains, for a finite set of indicators \(P(i)\). Then we may construct \(\cps'_j\) by replacing each \(x \in \cps\) with the set \(x+\bigcup_{i \in \mathcal{I}_x} Y(i)_\vee\), where \(\mathcal{I}_x = \{i \mid x \in \cps_P \text{ for some } P \in P(i)\}\).
\end{proposition}

\begin{proof}
By Lemma \ref{lem:constructables from ADs}, we can indeed write each \(A(i) = \bigcup_{P \in P(i)} W_P\) as stated. Define \(\cps''\) as the proposition claims \(\cps'\) can be written, that is, \(\cps'' \coloneqq \bigcup_{x \in \cps}x+D(x)\), where we let \(D(x) \coloneqq \bigcup_{i \in \mathcal{I}_x} Y(i)_\vee\). Since \(\cps \LD \cps_P\) for any indicator \(P\), we may also determine \(\mathcal{I}_x\) from \(x\) locally, thus also the cluster \(D(x)\), so it is clear that \(\cps \LD \cps''\). Thus, we just need to show that \(\cps' = \cps''\).

Since \(\cps'\), \(\cps'' \subset \Gamma_\vee\) (the latter due to each \(Y(i) \subset \Gamma\)), we may restrict attention to \(\Gamma_\vee\). For \(x \in \Gamma_\vee\), say \(x = \gamma_\vee\) for \(\gamma \in \Gamma\), we have \(x \in \cps'\) if and only if \(x^\star \in W'\), equivalently (by Equation \ref{eq:LD from construction}) \(x^\star  = \gamma_< \in A(i) + Y(i)_<\) for some \(i\). Since \(A(i) = \bigcup_{P \in P(i)} W_P\), this is in turn equivalent to \(\gamma_< \in W_P + Y(i)_<\), for some \(i\) and \(P \in P(i)\). Projecting down, this is equivalent to \(x = \gamma_\vee \in \cps(\mathcal{S},W_P) + Y(i)_\vee\). By Lemma \ref{lem:acceptance domains indicate patches}, \(\cps(\mathcal{S},W_P) = \cps_P\), so we see that \(x \in \cps'\) if and only if \(x \in \cps_P + Y(i)_\vee\), for some \(i\) with \(P \in P(i)\), that is, \(x \in  \cps''\), as required.
\end{proof}

For a single-coloured window, in particular, if we may take each \(A(i) = W\) in the above (that is, \(W'\) is a finite union of \(\Gamma_<\) translates of \(W\)), then we obtain a local derivation \(\cps \LD \cps'\) by replacing each \(x \in \cps\) with \(x + D\), for some fixed cluster \(D = Y_\vee\) of projected lattice points. More generally, we just need to replace each \(x \in \cps\) with a cluster \(x+D(x)\) that can depend on points of \(\cps\) within a bounded radius of \(x\). In particular, if the \(A(i)\) may be defined just through intersections, we only need to check a finite set of displacements from \(x\) which are \emph{in} the cut and project set. If we also need to use complements, we also use information of which points are \emph{out} of the cut and project set. It is largely due to this variety of potentially simpler situations, which may add computational advantages in practice, that we extended from just patches to indicators when defining acceptance domains.

\begin{theorem} \label{thm: MLD <=> constructable}
Let \(\cps = \cps(\mathcal{S},W)\) with \(W\) in non-singular position, and let \(\cps' \subset \Gamma_\vee\) be any Delone (multi)set in \(\phy\). Then \(\cps \LD \cps'\) if and only if \(\cps' = \cps(\mathcal{S},W')\) for a window \(W'\) that is constructable from \(W\) (and also non-singular). In particular, in this case, we may write \(W'\) (or each colour component of \(W'\)) as a finite union of \(\Gamma_<\) translated acceptance domains of \(W\), as in Equation \ref{eq:LD from construction}, and Proposition \ref{prop:LD from construction} applies. We have \(\Lambda \MLD \Lambda'\) if and only if \(\cps\) and \(\cps'\) may be expressed as cut and project sets of the same scheme \(\mathcal{S}\) with mutually constructable windows.
\end{theorem}

\begin{proof}
First, suppose that \(\cps' = \cps(\mathcal{S},W')\) with \(W'\) constructable from \(W\)  (and, as usual in these proofs, everything is single-coloured). It easily follows from this that \(W'\) is also non-singular (since \(\mathscr{B}(W) \supseteq \mathscr{B}(W')\), see also Lemma \ref{lem:Booleans same on NS}). Since \(W\) has non-empty interior, \(W'\) is compact and \(\Gamma_<\) is dense, it is clear that we may write
\begin{equation}\label{eq:LD<=>constructable}
W' = \bigcup_{z \in Z} ((W+z_<) \wedge W') = \bigcup_{z \in Z} ((W \wedge (W'-z_<))+z_<) = \bigcup_{z \in Z} W_z + z_<
\end{equation}
for some finite set \(Z \subset \Gamma\) (we just take enough lattice elements \(z\) so that the \(W+z_<\) cover \(W'\)), where we denote \(W_z \coloneqq W \wedge (W'-z_<)\). Clearly each \(W_z \subseteq W\) and is constructable from \(W\), since \(W'\) is, so Proposition \ref{prop:LD from construction} applies and \(\cps \LD \cps'\), as required. For a multi-coloured window, we just repeat the above for each colour component \(W'_j\) of \(W'\).

For the converse, assume that \(\cps \LD \cps'\) with derivation radius \(c \geq 0\), where \(\cps' \subset \Gamma_\vee\). Again, to remove technicalities, we will present the proof in the single-coloured case but indicate the simple changes needed otherwise afterwards. Let \(r > 0\) be such that all points of \(\cps'\) are within distance \(r\) of a point of \(\cps\) and \(\kappa \coloneqq r + c\). Given \(x \in \cps\), let \(D(x) \coloneqq \cps'[x,r]\). From the local derivation, if \(\cps[x,\kappa] = \cps[y,\kappa]\) then \(D(x) = D(y)\). That is, each \(D(x) = D_P\) only depends on the patch \(P = \cps[x,\kappa]\) and thus we may write
\[
\cps' = \bigcup_{\kappa\text{-patches } P} \cps_P + D_P .
\]
By Lemma \ref{lem:acceptance domains indicate patches}, each \(\cps_P = \cps(\mathcal{S},W_P)\) where, by definition, each \(W_P\) is constructable from \(W\). Since we assume that \(\cps' \subset \Gamma_\vee\), we also have \(D_P \subset \Gamma_\vee - x = \Gamma_\vee\) (since each \(x \in \Gamma_\vee\)), so one quickly verifies \(\cps_P + z = \cps(\mathcal{S},W_P) + z = \cps(\mathcal{S},W_P+z^\star)\) for each \(z \in D_P\). Thus, \(\cps' = \cps(\mathcal{S},W')\), where \(W'\) is a union of elements \(W_P + g_<\) for \(g \in \Gamma\). Of course, each \(W_P + g_<\) is still constructable from \(W\), so \(W'\) is still constructable from \(W\), as required. Since \(\mathscr{B}(W') \subseteq \mathscr{B}(W)\), it again follows that \(W'\) must be in non-singular position. In the case that \(\cps\) and \(\cps'\) are multi-coloured, a similar proof works, we just need to consider (multi-coloured) sets \(D(x)_i\), in place of \(D(x)\), to construct each component \(W_j'\) for \(W'\). The equivalence for \(\cps \MLD \cps'\) follows by applying the above in both directions.
\end{proof}

\begin{remark}
Perhaps in contrast to other references, we do not assume that \(\cps'\) is a cut and project set in the above, the proof above derives this fact (something which we need later). For simplicity, we do assume that \(\cps' \subset \Gamma_\vee\), although an inspection of the proof of the second direction reveals the extent to which this is necessary. If this is dropped then, instead, \(\cps'\) is a union of translates of cut and project sets (note that such an object may not be constructable from a cut and project scheme, by choosing algebraically incommensurate shifts in the physical space). One could, in principle, handle this situation by allowing different translates to be applied to the cut and project set for each colour component of a window. However, this will not be required later, so we again do not say much more on this technicality.
\end{remark}

\section{Proofs of main theorems}\label{sec:proof for general windows}

In this section we will prove the main structural Theorem \ref{thm:main}, which we break down into a few smaller steps. This section also contains the proofs of Theorems \ref{thm:existence}, \ref{thm:classification}, \ref{thm:symmetries}, and \ref{thm:PSA}. 

\begin{lemma}\label{lem:rescaling of cps is cps}
Let \(\cps = \cps(\mathcal{S},W)\) for non-singular \(W\) and \(L \colon \phy \to \phy\) be a linear automorphism. Suppose there exists a linear automorphism \(M \colon \tot \to \tot\) satisfying (1--3) of Theorem \ref{thm:main}. Then \(L\cps = \cps(\mathcal{S},A(W))\) with \(A(W)\) non-singular.
\end{lemma}

\begin{proof}
By definition, both \(\cps\) and \(\cps' \coloneqq \cps(\mathcal{S},A(W)) \subset \Gamma_\vee\) and \(A(W)\) is non-singular, since if \(\gamma_\vee \in A(W)\) for \(\gamma \in \Gamma\) then \(M^{-1}(\gamma)_< = A^{-1}(\gamma_<) \in A^{-1}(A(W)) = W\) and \(M^{-1}(\gamma) \in \Gamma\), contradicting \(W\) being non-singular. Then, for arbitrary \(\gamma \in \Gamma\), we have
\[
(\gamma_\vee)^\star \in \cps \iff \gamma_< \in W \iff M(\gamma)_< \in A(W) \iff (L(\gamma_\vee))^\star \in A(W) \iff L(\gamma_\vee) \in \cps' .
\]
Since \(\cps\), \(\cps' \subset \Gamma_\vee\), we have that \(x \in \cps\) if and only if \(L(x) \in \cps'\), as required. If \(W\) is multi-coloured, then the above holds for each colour component so the argument is analogous.
\end{proof}

\begin{corollary}\label{cor:constructable window from rescaling}
Suppose that (1--3) of Theorem \ref{thm:main} hold with respect to \(M \colon \tot \to \tot\). Then, if the elements of \(\Omega(\mathcal{S},W)\) are \(L\)-Sub, (4) holds too, that is, a translate of \(W\) is mutually constructable from \(A(W)\).
\end{corollary}

\begin{proof}
Take any \(\cps' = \cps(\mathcal{S},W)\). Without loss of generality, \(W\) is in non-singular position, as translating \(W\) does not affect \(\Omega = \Omega(\mathcal{S},W)\) or property (4). Assuming that one (equivalently, by Lemma \ref{lem:L-sub stable under LIs}, any) element of \(\Omega\) is \(L\)-Sub, recall that we have a subdivision (local derivation) map \(S \colon L\Omega \to \Omega\). Then \(L(\cps') \MLD \cps\), where \(\cps \coloneqq S(L\cps')\) with \(\cps \LIs \cps'\). Without loss of generality, applying an extra shift to \(S\) if required (which keeps \(S\) a local derivation map), we can assume that \(\cps \subset \Gamma_\vee\) too. We have that \(L\cps' = \cps(\mathcal{S},A(W))\), by Lemma \ref{lem:rescaling of cps is cps}. Since \(L\cps' \MLD \cps\), by Theorem \ref{thm: MLD <=> constructable} we have that \(\cps = \cps(\mathcal{S},W')\) for a window \(W'\) that is mutually constructable from \(A(W)\). Finally, since \(\cps \LIs \cps'\), by Lemma \ref{lem: LI <=> equal windows} we have that \(W' = W-s\) for some \(s \in \intl\), as required.
\end{proof}

By the above result, to establish the first main direction of Theorem \ref{thm:main}, we just need to show the existence of the linear automorphism \(M \colon \tot \to \tot\) satisfying (1--3). We start with the reverse direction though, that (1--4) implies \(L\)-Sub, which is more easily checked.

\begin{proof}[Proof of equivalence in Theorem \ref{thm:main}]
Suppose that (1--3) hold and take \(\cps' = \cps(\mathcal{S},W) \in \Omega = \Omega(\mathcal{S},W)\). We assume, without loss of generality, that \(W\) is non-singular, since translating \(W\) does not affect \(\Omega\) (and thus also does not affect whether its elements are \(L\)-Sub), nor properties (1--4). Then, by Lemma \ref{lem:rescaling of cps is cps}, we have that \(L\cps' = \cps(\mathcal{S},A(W))\). Now, suppose also (4a), that \(W-s\) is constructable from \(A(W)\) for some \(s \in \intl\). By Theorem \ref{thm: MLD <=> constructable} we have that \(L(\cps') \LD \cps(\mathcal{S},W-s) \eqqcolon \cps\). By Lemma \ref{lem: LI <=> equal windows}, we have that \(\cps \LIs \cps'\). Thus, we have shown that there exists some \(\cps \in \Omega(\mathcal{S},W)\) which is \(L\)-sub. If \(L\) is expansive, then it is also \(L\)-Sub. If \(L\) is not expansive, but we have the stronger (4), that is, \(A(W)\) is also constructable from \(W-s\), then these are mutually constructable, so by Theorem \ref{thm: MLD <=> constructable}, \(L\cps' \MLD \cps\) and \(\cps\) is again \(L\)-Sub, as required.

For the converse direction, suppose that \(\cps \in \Omega(\mathcal{S},W)\) is \(L\)-Sub, so \(L\cps' \MLD \cps \LIs \cps'\) where, without loss of generality, \(\cps' = \Omega(\mathcal{S},W)\) for non-singular \(W\). By Corollary \ref{cor:projected lattice invariant under L}, \(L\) induces an isomorphism \(L|_{\Gamma_\vee} \colon \Gamma_\vee \to \Gamma_\vee\). This lifts to an automorphism of \(\Gamma\) and hence to a linear automorphism \(M \colon \tot \to \tot\) in the obvious way, with \(M(\Gamma) = \Gamma\) and \(M(g)_\vee = L(g_\vee)\) for all \(g \in \Gamma\). In particular, (1) of Theorem \ref{thm:main} holds for this choice of \(M\).

Next we show (3), that \(M(\intl) = \intl\). Let \(r > 0\) be such that every point of \(\tot\) is within distance \(r\) of \(\Gamma\). Fix a basis \(b\) of \(\intl\) and take an arbitrary \(R > 0\). For each \(b_i \in b\), choose \(\gamma_i \in \Gamma\) within radius \(r\) of \(R \cdot b_i\). Then \(\|(\gamma_i)_\vee\| \leq r\) (we may choose a norm on \(\phy\) and \(\intl\), and max norm on \(\tot\)) and thus \(M(\gamma_i)_\vee = L((\gamma_i)_\vee)\) has norm at most \(Cr\), with \(C\) only depending on \(L\). We see that, by increasing \(R > 0\), we may make the subspace \(H\) spanned by the \(\gamma_i\) as close to \(\intl\) as desired, with \(M(H)\) remaining arbitrarily close. It follows that \(M(\intl) = \intl\).

Let us show (2), that \(M(\phy) = \phy\) and \(M |_{\phy} = L\). Choose lines \(\ell_1\), \ldots, \(\ell_d < \phy\) with \(\ell_1 + \cdots + \ell_d = \phy\). We wish to show that \(M(\ell_i) < \phy\), from which it follows \(M(\phy) = \phy\). By Corollary \ref{cor:action of L on return vectors}, there exists some \(r > 0\) so that, for all \(v \in D_r\), we have that \(L(v) \in D_0\). By Lemma \ref{lem:return vectors of cps}, there exists some \(\epsilon > 0\) so that, for all \(\gamma \in \Gamma\), we have that \(\gamma_\vee \in D_r\) provided \(\|\gamma_<\| < \epsilon\). Thus, for such \(\gamma\), we have that \(L(\gamma_\vee) \in D_0\), that is, there exist \(x\), \(x+L(\gamma_\vee) \in \cps\) and so \(x^\star\), \(x^\star+L(\gamma_\vee) \in W\), thus \(L(\gamma_\vee) \in W - x^\star \subset W - W\).

Hence, we have shown that for all \(\gamma \in \Gamma\), whenever \(\|\gamma_<\| < \epsilon\), we have \(\|M(\gamma)_<\| \leq \rho\), where \(W-W \subset B(\mathbf{0},\rho)\). For arbitrary \(R > 0\), take \(\gamma_1\), \ldots, \(\gamma_d \in \Gamma\) with each \(\|\gamma_i\| \geq R\) and each \(\gamma_i\) within radius \(\epsilon\) of \(\ell_i\). Clearly such \(\gamma_i\) exist, as the line spanned by \(\ell_i\) is dense in the torus \(\mathbb{T} = \tot/\Gamma\), or at least some sub-torus that contains the origin. As we can take \(R > 0\) arbitrarily large, we can make the line through \(\gamma_i\) lie as close as we wish to \(\ell_i\). Moreover, each \(\|M(\gamma_i)\| \geq cR\), where \(c\) only depends on \(M\), so by taking \(R\) large we may ensure that the line through \(M(\gamma_i)\) is as close to \(\phy\) as we wish. By continuity, it follows that each \(M(\ell_i) \subset \phy\), as required, so \(M(\phy) = \phy\).

Now, given \(g \in \Gamma\), we have \(M(g)_\vee = M(g_\vee + g_<)_\vee = M(g_\vee)_\vee + M(g_<)_\vee = M(g_\vee)_\vee = M(g_\vee)\), since by (3) above \(M(g_<) \in \intl\) (so that \(M(g_<)_\vee = \mathbf{0}\)) and \(M(g_\vee) \in \phy\) (so that \(M(g_\vee)_\vee = M(g_\vee)\)). It follows that \(M(g_\vee) = M(g)_\vee = L(g_\vee)\), the latter equality by definition of \(M\), for all \(g \in \Gamma\). Since \(\Gamma_\vee\) contains a basis for \(\phy\), we have \(M |_{\phy} = L\), so have established (2).

Since we have constructed \(M \colon \tot \to \tot\) satisfying (1--3), (4) also holds by Corollary \ref{cor:constructable window from rescaling}, so we have shown that the elements of \(\Omega\) being \(L\)-Sub implies (1--4), thus (1--3) and the weaker (4a), as required.
\end{proof}

To complete the proof of Theorem \ref{thm:main}, we have the following. Note that the proof makes use of Theorem \ref{thm:classification}, given a little later.

\begin{proposition}
The map \(A\) of Theorem \ref{thm:main} is contractive, if \(L\) is expansive.
\end{proposition}

\begin{proof}
Suppose that \(\Omega\) consists of \(L\)-sub patterns, recall that substitution \(\sub \colon \Omega \to \Omega\) is given by \(\sub(\cps) = S(L\cps)\), where \(S\) is a local derivation map, with derivation radius some \(c \geq 0\). As usual, take \(\lambda > 1\) so that \(LB \supseteq \lambda B\) (where \(B\) is the unit ball in \(\phy\)), and let \(1 < \lambda' < \lambda\). Then, for \(r\) large enough so that \(\lambda' r \leq \lambda r - c\), if \(\cps_1[\mathbf{0},r] = \cps_2[\mathbf{0},r]\) then \((L\cps_1)[\mathbf{0},\lambda r] = (L\cps_2)[\mathbf{0},\lambda r]\) and thus \(\sub(\cps_1)[\mathbf{0},\lambda' r] = \sub(\cps_2)[\mathbf{0},\lambda' r]\). Thus, the patches about the origin of such \(\cps_1\) and \(\cps_2\) agree on larger and larger balls, under repeated application of \(\sub\).

In particular, consider the set \(\Omega[U]\) of all \(\cps \in \Omega\) with \(\tau(\cps) \in [U]\), where \(U \subseteq \intl\) is a ball centred at the origin, small enough to be wholly contained in \(W_P\) for \(P = \cps'[\mathbf{0},r]\), for the (without loss of generality) non-singular \(\cps' = \cps(\mathcal{S},W)\). These are given by translating the lattice by an element of \(U\), cutting and projecting (with a bit more care for singular patterns). Clearly, even for the singular such patterns, these all have the same \(r\)-patch over the origin as \(\cps'\). Thus, by the above and iterating substitution, we see that \(\sub^n(\Omega[U])\) can be made arbitrarily small in \(\Omega\) (with respect to the tiling metric/uniformity). Thus, \(\tau(\sub^n(\Omega[U])) = T^n(\tau(\Omega[U])) = T^n[U] = M^n_\mathbb{T}[U] + [u]\) is contained in an arbitrarily small ball (centred at \([u]\)), equivalently, \(M^n_\mathbb{T}[U] = [M^n(U)]\) is contained in an arbitrarily small ball (centred at the origin), by making \(n \in \N\) sufficiently large. Since the projection \(\pi_< \colon \tot \to \mathbb{T}\) is a local isomorphism, we see that \(M^n(U)\) can be made arbitrarily small, say, much smaller in radius than the fixed ball \(U\). It follows that \(M^n\) is a contraction on \(\intl\), that is, all its eigenvalues are strictly less than \(1\) in modulus. Thus, the same must be true of \(M\), so \(A = M|_{\intl}\) is a contraction, as required.
\end{proof}

The existence result, Theorem \ref{thm:existence}, follows quickly from Theorem \ref{thm:main}:

\begin{proof}[Proof of Theorem \ref{thm:existence}]
As in the statement of the result, let \(\tot\), \(\phy\), \(\Gamma\) and any expansion \(L \colon \phy \to \phy\) be given. First, assume that we may find an internal space \(\intl\) and window \(W\) for which \(\Omega(\mathcal{S},W)\) consists of \(L\)-Sub patterns. Thus, in particular, (1--3) of Theorem \ref{thm:main} hold, that is, there exists a linear map \(M \colon \tot \to \tot\) with \(M(\Gamma) = \Gamma\) and \(M(\phy) = \phy\) with \(M |_{\phy} = L\). Also by Theorem \ref{thm:main}, we have that \(A \coloneqq M|_{\intl}\) is contractive (and so \(\intl\) is always uniquely defined as the contractive subspace), so \(M\) is the required hyperbolic map.

Conversely, suppose that a hyperbolic automorphism \(M \colon \tot \to \tot\) is given, for which \(M(\Gamma) = \Gamma\), \(M(\phy) = \phy\) and \(L \coloneqq M|_{\phy}\) is expansive. Let \(\intl\) be given by the complementary contractive subspace of \(M\). Note that \(\Gamma \cap \intl = \{\mathbf{0}\}\). Indeed, otherwise, \(\gamma \in \intl\) for a non-zero \(\gamma \in \Gamma\). But then \(M^n(\gamma) = A^n(\gamma) \to \mathbf{0}\) are other such points, arbitrarily close to the origin, contradicting that \(\Gamma\) is a lattice. By assumption, \(\Gamma \cap \phy = \{\mathbf{0}\}\) and \([\phy]\) is dense in \(\mathbb{T}\), equivalently, \(\Gamma_<\) is dense in \(\intl\), so this data defines a cut and project scheme \(\mathcal{S}\) satisfying (1--3) of Theorem \ref{thm:main}. So we just need to demonstrate the existence of some compact and topologically regular window \(W \subset \intl\) also satisfying (4). 

Let \(B \subset \intl\) be the unit ball and take a sufficiently large but finite set \(Z \subset \Gamma\) of lattice points for which \(B \subseteq \bigcup_{z \in Z} A(B) + z_<\). This defines an iterated function system (see Remark \ref{rem:IFS}) in \(\intl\), with contractions \(x \mapsto A(x) + z_<\). It follows from e.g. \cite[Theorem 9.1]{Falconer03} that there is a unique, compact, non-empty $W$ such that 
\[
W \coloneqq \bigcup_{z \in Z} A(W) + z_< .
\]
In particular, \(W\) is constructable from \(A(W)\), indeed it satisfies an equation of the form of Equation \ref{eq:construction} (using only a union and not the \(\wedge\) or \(\neg\) Boolean operators). So we just need to show that \(W\) is topologically regular.

To see this, define the operator \(X\) on non-empty, compact subsets of \(\intl\), given by
\[
X(K) \coloneqq \bigcup_{z \in Z} A(K) + z_< .
\]
By the choice of $Z$, we have that \(B \subseteq X(B)\). Clearly, for \(U \subseteq V\), we have \(X(U) \subseteq X(V)\), so it follows that \(B \subseteq X(B) \subseteq X^2(B) \subseteq X^3(B) \subseteq \cdots\). From the proof of \cite[Theorem 9.1]{Falconer03}, we know that \(X^n(B) \to W\) as \(n \to \infty\), with respect to the Hausdorff distance, so that the sets \(X^n(B) \subset W\) are dense in \(W\). Moreover, each \(X^n(B)\) is a finite union of topological balls, so has dense interior. Hence, \(W\) also has dense interior, that is, it is topologically regular, as required.
\end{proof}

Next, we wish to prove the result classifying all substitutions \(\sub \colon \Omega \to \Omega\) on a hull of Euclidean cut and project sets in terms of the corresponding map on the torus. Firstly, we observe that the existence of the toral automorphism \(M_\mathbb{T}\) defined by Theorem \ref{thm:main} already imposes a lot of structure (also note that \(M_\mathbb{T}\) is uniquely defined as the continuous extension of \([x] \mapsto [Lx]\) on \([\phy]\), which is dense in \(\mathbb{T}\), so \(M\) is always uniquely defined, given \(L\), when it exists). Recall, for the below, that given \([u] \in \mathbb{T}\), we let \(T_u[x] = M_\mathbb{T}[x] + [u]\) and that all substitution maps are assumed to be invertible (see the comment before Theorem \ref{thm:classification}).

\begin{proposition}\label{prop:subs on torus}
Consider a hull of Euclidean cut and project sets \(\Omega = \Omega(\mathcal{S},W)\). Let \(M \colon \tot \to \tot\) satisfying (1--3) of Theorem \ref{thm:main} be given, which thus defines a toral automorphism \(M_\mathbb{T} \colon \mathbb{T} \to \mathbb{T}\). Then, for any substitution map \(\sub \colon \Omega \to \Omega\) with inflation \(L = M|_{\phy}\), we have \(\tau \circ \sub = T_u \circ \tau\) for some \([u] \in \mathbb{T}\). Moreover, we may express this \([u] \in \mathbb{T}\) as
\[
[u] = \tau(\sub\cps) - M_\mathbb{T}(\tau(\cps)) 
\]
for any \(\cps \in \Omega\). We have that \(\nonsing+[u] = \nonsing\) and \(\sing+[u] = \sing\). Any other substitution map \(\sub' \colon \Omega \to \Omega\) with \(\tau \circ \sub' = T_u \circ \tau\) must have \(\sub' = \sub\).
\end{proposition}

\begin{proof}
Let \(\cps \in \Omega\) be arbitrary, \([a] \coloneqq \tau(\cps)\) and \([b] \coloneqq \tau(\sub \cps) \in \mathbb{T}\). Since \(\sub = S \circ L\) satisfies \(\sub(\cps+x) = \sub(\cps) + Lx\),
\[
\tau(\sub(\cps+x)) = \tau(\sub(\cps) + Lx) = \tau(\sub(\cps)) + [Lx] = [b] + M_\mathbb{T}[x] 
\]
for all \(x \in \phy\). On the other hand,
\[
T_u(\tau(\cps+x)) = T_u([a] + [x]) = M_\mathbb{T}[a+x] + ([b] - M_\mathbb{T}[a]) = M_\mathbb{T}[x] + [b] 
\]
so \(\tau \circ \sub = T_u \circ \tau\) on the translational orbit \(\cps + \phy \subset \Omega\) for the above choice of \([u] = [b]-M_\mathbb{T}[a]\). Since this orbit is dense, and \(\tau\), \(T_u\) and \(\sub\) are continuous, \(\tau \circ \sub = T_u \circ \tau\) on all of \(\Omega\), as required.

Now, suppose that \(\cps \in \Omega\) is singular, so there exists some other \(\cps' \neq \cps\) with \(\tau(\cps) = \tau(\cps') = [x] \in \mathbb{T}\), where \(x \in \sing\). By the above, \(\tau(\sub \cps) = T_u(\tau \cps) = T_u(\tau \cps') = \tau(\sub \cps')\). Since \(\sub\) is injective, we have \(\tau(\sub \cps) = \tau(\sub \cps')\) for distinct \(\sub \cps \neq \sub \cps'\), so each is singular. Thus, \(\sub\) maps singular points to singular points. Similarly, \(\sub\) maps non-singular points to non-singular points. Indeed, if \(\cps\) is non-singular but \(\sub \cps\) is not, we have \(\tau(\sub \cps) = \tau(\cps')\) for some \(\cps' \neq \sub \cps\). We may write \(\cps' = \sub \cps''\), since \(\sub\) is surjective. From \(\tau(\sub \cps) = \tau(\sub \cps'')\), applying \(T_u \circ \tau = \tau \circ \sub\), we have \(T_u(\tau \cps) = T_u(\tau \cps'')\). But \(\cps \neq \cps''\) (or else \(\sub \cps = \sub \cps'' = \cps'\)). Since \(T_u\) is injective, we see that \(\tau \cps = \tau \cps''\), contradicting that \(\cps\) is non-singular.

We see that \(\sub\), and thus \(T_u[x] = M_\mathbb{T}[x] + [u]\), preserves the singular and non-singular points of \(\mathbb{T}\). Since \(M\phy = \phy\) and \(M\Gamma = \Gamma\), we have \(M_\mathbb{T}(\nonsing) = \nonsing\), so \(\nonsing + [u] = \nonsing\) and similarly \(M_\mathbb{T}(\sing) = \sing\) so \(\sing + [u] = \sing\), as required.

Finally, suppose that \(\sub'\) also satisfies \(\tau \circ \sub' = T_u \circ \tau\). Let \(\cps \in \Omega\) be non-singular (thus, by the above, \(\sub \cps\) and \(\sub' \cps\) are also non-singular). Then
\[
(\tau \circ \sub)(\cps) = (T_u \circ \tau)(\cps) = (\tau \circ \sub')(\cps) .
\]
Since \(\sub\cps\) is non-singular and \(\tau(\sub \cps) = \tau(\sub' \cps)\), we must have that \(\sub \cps = \sub' \cps\). But then \(\sub(\cps+x) = \sub(\cps) + Lx = \sub'(\cps) + Lx = \sub'(\cps+x)\), so \(\sub = \sub'\) on the dense orbit \(\cps+\phy\) in \(\Omega\). By continuity, \(\sub = \sub'\).
\end{proof}

\begin{proof}[Proof of Theorem \ref{thm:classification}]
Suppose that \(\Omega = \Omega(\mathcal{S},W)\) consists of \(L\)-Sub patterns. Thus, in particular, (1--3) of Theorem \ref{thm:main} hold with respect to a linear automorphism \(M \colon \tot \to \tot\), which induces the toral automorphism \(M_\mathbb{T} \colon \mathbb{T} \to \mathbb{T}\). Thus, Proposition \ref{prop:subs on torus} applies and we have \(\tau \circ \sub = T_u \circ \tau\) for some \([u] \in \mathbb{T}\), which uniquely characterises it, and satisfies \(\nonsing + [u] = \nonsing\) and \(\sing + [u] = \sing\). So we just need to show that the set of such \([u]\) associated with a substitution are precisely those for which \(W-u_<\) is mutually constructable from \(A(W)\) (or just constructable, when \(L\) is expansive).

So, firstly suppose that \(W-s\) is constructable from \(A(W)\), for \(s = u_< \in \intl\). Equivalently, \(W-s-A(x)\) is constructable from \(A(W-x) = A(W) - A(x)\) for any \(x \in \intl\). Choose \(x \in \intl\) so that \(W-x\) is non-singular. Then, by Lemma \ref{lem:rescaling of cps is cps}, we have \(L\cps' = \cps(\mathcal{S},A(W-x))\), where \(\cps' \coloneqq \cps(\mathcal{S},W-x)\). Since \(W-s-A(x)\) is constructable from \(A(W-x)\), we have a local derivation \(L\cps' \LD \cps\), where \(\cps \coloneqq \cps(\mathcal{S},W-s-A(x))\). We have that \(\cps \LIs \cps'\), so \(\cps\) is \(L\)-sub, with predecessor \(\cps'\), and we have an associated substitution map \(\sub \colon \Omega \to \Omega\) with \(\sub(\cps') = \cps\). Moreover, if \(L\) is expansive then \(\sub\) is injective and \(\cps \MLD L\cps'\) (by Theorem \ref{thm:recognisability}), or we similarly obtain this by simply assuming that \(W-s\) is \emph{mutually} constructable from \(A(W)\).

As discussed in Subsection \ref{sec:torus parametrisation}, for \(z \in \intl\) and \(W-z\) non-singular, we have \(\tau(\cps_z) = z\) for \(\cps_z = \cps(\mathcal{S},W-z) \in \Omega\). Thus, in the above, we have \(\tau(\cps) = s+A(x)\) and \(\tau(\cps') = x\). By Proposition \ref{prop:subs on torus}, the substitution constructed satisfies \(\tau \circ \sub = T_s \circ \tau\) since \(\tau(\sub \cps') - M_\mathbb{T}(\tau \cps') = [s+A(x)] - [Mx] = [s+A(x)] - [A(x)] = [s]\). Thus, we have constructed the required substitution when \(u = s\).

If, instead, we merely have \(u_< = s\), write \(u = v+s\) for \(v = u_\vee \in \phy\). Then we may define the substitution map \(\sub' \colon \Omega \to \Omega\) as \(\sub'(\cps) \coloneqq \sub(\cps) + v\) (where \(\cps \in \Omega\) is now arbitrary). Then
\[
\tau(\sub' \cps) = \tau((\sub \cps) + v) = \tau(\sub \cps) + [v] = T_s(\tau \cps) + [v] = T_{(v+s)}(\tau \cps) = T_u(\tau \cps) ,
\]
so \(\sub'\) satisfies \(\tau \circ \sub' = T_u \circ \tau\), as required.

Conversely, suppose that we have an (invertible) substitution map \(\sub \colon \Omega \to \Omega\) satisfying \(\tau \circ \sub = T_u \circ \tau\), so each element of \(\Omega\) is \(L\)-Sub. Take any non-singular \(\cps' \in \Omega\), say \(\tau\cps' = [a]\) with \(a \in \intl\). Thus, \(\cps' = \cps(\mathcal{S},W-a)\). By Proposition \ref{prop:subs on torus}, \(\cps \coloneqq \sub \cps'\) is also non-singular. Thus, at least after shifting by some \(v \in \intl\), we may write \(\cps-v = \cps(\mathcal{S},W-b)\) for some \(b \in \intl\), so \(\tau(\cps-v) = [b]\) and \(\tau(\cps) = [b+v]\). Hence,
\[
[b+v] = \tau(\cps) = \tau(\sub \cps') = T_u(\tau \cps) = T_u[a] = M_\mathbb{T}[a] + [u] = [A(a)] + [u],
\]
so \(u-v = b - A(a) + \gamma\) for some \(\gamma \in \Gamma\). Projecting to internal space, \(u_< = b - A(a) + \gamma_<\).

Now, since \(L\cps' \MLD \cps\), we also have \(L\cps' \MLD \cps-v\). Since \(L\cps' = \cps(\mathcal{S},A(W-a))\) (Lemma \ref{lem:rescaling of cps is cps}), by Theorem \ref{thm: MLD <=> constructable} we have that \(W-b\) is mutually constructable from \(A(W-a) = A(W) - A(a)\). Equivalently, shifting everything by \(A(a)\), we have \(W-(b-A(a))\) is mutually constructable from \(A(W)\). Since shifting by a projected lattice vector does not affect constructability, we also have that \(W-(b-A(a)+\gamma_<) = W-u_<\) is mutually constructable from \(A(W)\), as required.
\end{proof}

Next, we prove the characterisations of symmetries and LIDS points of the hull:

\begin{proof}[Proof of Theorem \ref{thm:symmetries}]
Suppose that \(L\) is a symmetry of \(\Omega = \Omega(\mathcal{S},W)\) where, without loss of generality, \(W\) is non-singular. By repetitivity, since \(L\) is a symmetry, we have that \(\cps \LIs L\cps\) for all \(\cps \in \Omega\) so, in particular, each \(\cps \in \Omega\) is \(L\)-Sub, so (1--4) of Theorem \ref{thm:main} hold. Let \(\cps \coloneqq \cps(\mathcal{S},W) \in \Omega\). By Lemma \ref{lem:rescaling of cps is cps}, \(L\cps = \cps(\mathcal{S},A(W))\). Since \(L\cps \LIs \cps\), by Lemma \ref{lem: LI <=> equal windows}, we have that \(A(W) = W-s\) for some \(s \in S\), so (4c) holds, as required. Conversely, suppose that (1--3) and (4c) hold. Taking \(\cps = \cps(\mathcal{S},W)\), again, \(L\cps = \cps(\mathcal{S},A(W))\), by the Lemma \ref{lem:rescaling of cps is cps}. Since we assume that \(A(W) = W-s\), we have \(L\cps \LIs \cps\), from the reverse direction of Lemma \ref{lem: LI <=> equal windows}, so \(L\) is a symmetry, as required.
\end{proof}

\begin{proof}[Proof of Theorem \ref{thm:PSA}]
Clearly (A) implies (B). Assuming (B), we have \(\cps \in \Omega\), with \(\tau \cps = [v+s]\), is LIDS with respect to \(L^m\). Thus, there is some (invertible) substitution map \(\sub \colon \Omega \to \Omega\), with expansion \(L^m\) and \(\sub(\cps) = \cps\). By Theorem \ref{thm:classification}, we have that \(\tau \circ \sub = T_u \circ \tau\) for some \(u \in \tot\), where \(T_u[x] = M_\mathbb{T}^m[x] + [u]\). Since \(\sub(\cps) = \cps\), we have that
\[
[v+s] = \tau \cps = \tau(\sub \cps) = T_u(\tau \cps) = T_u[v+s] = [M^m(v+s)] + [u] = [L^m(v) + A^m(s) + u] .
\]
Thus, \((L^m(v) + A^m(s) + u) - (v+s) \in \Gamma\). Projecting to \(\intl\), we obtain \(A^m(s) + u_< - s \in \Gamma_<\). By Theorem \ref{thm:classification}, we have that \(W-u_<\) is mutually constructable from \(A^m(W)\). Shifting a set by a projected lattice element does not affect constructability, so \(W-u_< + (A^m(s) + u_< - s) = W + A^m(s)-s\) is mutually constructable from \(A(W)\). Shifting both by \(-A^m(s)\), we see that \(W-s\) is mutually constructable from \(A^m(W)-A^m(s) = A^m(W-s)\), so (4c) holds, and thus so does the weaker (4d).

Finally, suppose that (4c) holds for power \(p\). Using \(M^p\) in Theorem \ref{thm:main}, we have that the elements of \(\Omega\) are \(L^p\)-Sub. By the same theorem, this also holds given the weaker (4d) if \(L\) is expansive. Thus, \(W-s\) is constructable from \(A^p(W-s) = A^p(W)-A^p(s)\), so \(W-(s-A^p(s))\) is constructable from \(A^p(W)\). By Theorem \ref{thm:classification}, for any \([u] \in \mathbb{T}\) with \(u_< = s-A^p(s)\), there is a substitution \(\sub \colon \Omega \to \Omega\) which satisfies \(\tau(\sub \cps) = T_u(\tau \cps)\) for all \(\cps \in \Omega\). Then, given any \(\cps \in \Omega\) with \(\tau \cps = [v+s]\) for some \(v \in \phy\), take in particular \([u] = [(v-L^p(v))+(s-A^p(s))]\) and let \(\sub\) be the associated substitution. Then
\[
\tau(\sub \cps) = T_u(\tau \cps) = [M^p(v+s)] + [u] = [L^p(v)+A^p(s) + (v-L^p(v)+s-A^p(s))] = [v+s] ,
\]
so \(\tau(\sub \cps) = \tau \cps\). Thus, \(\sub(F) \subseteq F\) for the fibre \(F = \tau^{-1}[v+s]\). By assumption, \(F\) is finite. Moreover, since \(\sub\) is bijective, some power \(\sub^i\) is the identity on \(F\), thus each element of \(F\) is fixed by the (invertible) substitution \(\sub^i\), which has inflation map \(L^m\) for some multiple \(m\) of \(p\) and (A) holds, as required.
\end{proof}

\section{Substitutional polytopal windows}\label{sec:proof polytopal}

In this section we prove Theorem \ref{thm:polytopal}, giving a simple necessary and sufficient condition for a cut and project scheme with polytopal window to produce substitutional cut and project sets, as well as Theorem \ref{thm:polytopal PSA}, which characterises the LIDS (equivalently, pseudo self-affine) points in terms of the rational points. Recall, from Section \ref{sec:main results}, the finite sets \(\sH\) and \(\sH_0\) of supporting hyperplanes and subspaces, respectively. The following concept will be useful, especially when analysing the complexity of polytopal cut and project sets:

\begin{definition}
A subset \(f \subseteq \sH_0\) of exactly \(n = \dim(\intl)\) elements is called a \textbf{flag} if 
\[
\bigcap_{V \in f} V = \{\mathbf{0}\} .
\]
Similarly, a subset \(f \subseteq \sH\) of \(n\) elements is called a \text{flag} when the intersection is a singleton, equivalently, \(V(f)\) is a flag. 
\end{definition}

The following stabiliser subgroups will be particularly useful in establishing that (1) implies (2) in Theorem \ref{thm:polytopal}.

\begin{definition}
For a subspace \(X \leqslant \intl\), we define the \textbf{stabiliser} subgroup
\[
\Gamma^X \coloneqq \{\gamma \in \Gamma \mid X + \gamma_< = X\} = \{\gamma \in \Gamma \mid \gamma_< \in X\} \leqslant \Gamma .
\]
In particular, for each \(V \in \sH_0\) we have a stabiliser \(\Gamma^V\). For a translation \(Y \subset \intl\) of a subspace \(V(Y)\), we similarly define the \textbf{stabiliser} subgroup
\[
\Gamma^Y \coloneqq \{\gamma \in \Gamma \mid Y + \gamma_< = Y\} = \Gamma^{V(Y)} .
\]
In particular, for each \(H \in \sH\) we have a stabiliser \(\Gamma^H\).
\end{definition}

Throughout this Section, as in Section \ref{sec:results for polytopal windows}, we assume that the necessary conditions (1--3) of Theorem \ref{thm:main} (see also Theorem \ref{thm:existence}) hold, relative to a hyperbolic map \(M \colon \tot \to \tot\), restricting to an expansion \(L \colon \phy \to \phy\) on the physical space. A marked difference to the non-polytopal case is the second conclusion of the following:

\begin{lemma} \label{lem:invariant subspaces}
If \(W\) is FD-invariant then \(A(X) = X\) for any subspace \(X\) that is an intersection of subspaces from \(\sH_0\). In particular, \(A\) is diagonalisable and has only real, irrational eigenvalues of modulus strictly less than one.
\end{lemma}

\begin{proof}
Since each \(V \in \sH_0\) is invariant under \(A\), the same is true of intersections \(X\) of such subspaces. Intersecting all but one of the elements of a flag in \(\sH_0\) (of which there must be one, or else \(W\) would not be compact) defines a \(1\)-dimensional subspace \(Y\), and these subspaces span \(\intl\). Since each such subspace is invariant, we see that \(A\) is diagonalisable with real eigenvalues. They must be irrational, or else \(M\) has a rational eigenvalue too, which consequently has corresponding eigenvector \(\gamma\) for some \(\gamma \in \Gamma\) (this is most easily seen by reparameterising to \(\Gamma = \Z^k\), so that \(M\) is represented by an integer matrix). Since \(\intl\) is the contracting subspace of \(M\), it follows that \(\gamma \in \intl\), contradicting that \(\intl \cap \intl = \{\mathbf{0}\}\).
\end{proof}

\subsection{Proof that face direction invariant and rational implies substitutional}
In this subsection we prove the first direction of Theorem \ref{thm:polytopal}, establishing that a cut and project scheme is substitutional when it is FD-invariant and rational. So, in this subsection we will assume:
\begin{enumerate}
	\item face direction invariance with respect to \(A^p\), for some \(p \in \N\);
	\item rationality.
\end{enumerate}
For the proof below we assume for simplicity (and without loss of generality, as neither FD-invariance, rationality or being substitutional depends on the translate of \(W\) used) that \(W\) is in \emph{singular} position, with a vertex \(v\) over the origin. Moreover, by Proposition \ref{prop:reducing to one colour}, we may assume throughout that \(W\) is single-coloured.

\begin{theorem}\label{thm:polytopal direction one}
For \(W\) as above, we have that \(W\) is constructable from \(A^m(W)\) for some \(m \in \N\). Thus, the elements of \(\Omega\) are \(L^m\)-Sub.
\end{theorem}

\begin{proof}
By rationality (and our choice of translate above), the vertices of \(W\) belong to \(\Q\Gamma_<\), that is, \(\mathrm{Ver}(W) \subset \frac{1}{N}\Gamma_<\) for some \(N \in \N\). We also have that \(A\) maps \(\Gamma_<\) isomorphically to itself, so similarly for \(\frac{1}{N}\Gamma_<\). Then \(A\) quotients to an isomorphism \(A \colon Q \to Q\) on the quotient group \((\frac{1}{N}\Gamma_<) / \Gamma_<\), which is finite since \(\frac{1}{N}\Gamma_<\) is finite index in \(\Gamma_<\). It follows that \(A^m\) acts as the identity on \(Q\) for a sufficiently high power \(m=jp\), with \(j \in \N\). Moreover, by assumption, \(A^m\) acts as the identity on \(\sH_0\). Note that \(A^m\) permutes the connected components of \(\intl \setminus \bigcup_{H \in \sH_0}\). By taking an even larger power, if necessary, we may assume that \(A^m\) acts as the identity on this finite set. Henceforth we set \(m = 1\), to simplify notation, by starting with some such power of \(M\).

Construct \(B \in \mathscr{B}(A(W))\) containing a ball around the origin, but also with very small diameter \(\kappa\). For \(v \in \mathrm{Ver}(W)\), define \(B'_v \coloneqq (B+\gamma_<) \wedge A(W)\) for \(\gamma_<\) very close to \(Av \in \mathrm{Ver}(AW)\). It is not hard to see that, if \(\kappa\) is small enough and \(\gamma_<\) is close enough to \(Av\), then \(B_v \coloneqq B'_v + (v-Av) \subset W\) and \(B_v = W \cap U\) on some ball \(U\) centred at \(v\). Indeed, by our choice of the power of \(A\), a small region around the vertex at \(v\) transforms, under \(A\), to a translate of this region in \(A(W)\) about \(Av\). Since \(v-Av \in \Gamma_<\), we also have \(B_v \in \mathscr{B}(A(W))\). Thus, by taking the union of elements \(B_v\) over all \(v \in \mathrm{Ver}(W)\), we define a set \(B_0 \in \mathscr{B}(A(W))\) with \(B_0 \subseteq W\) which contains a neighbourhood of the vertex set in \(W\).

The above may now be repeated on higher dimensional cells. For example, for each edge \(e\) between vertices \(v\), \(w \in \mathrm{Ver}(W)\), start with \(B_v\) as above. By rationality, \(N(w-v) \in \Gamma_<\). Moreover, by Lemma \ref{lem:invariant subspaces}, so is \(A^i(N(w-v))\) for all \(i \in \N\), since the line parallel to \(e\) (which is an intersection of hyperplanes of \(\sH_0\)) is invariant under \(A\). Since \(A\) is contractive, it follows that the subspace parallel to \(e\) contains a dense subset in \(\Gamma_<\). Thus, by taking sufficiently many translates of \(B_v'\) with elements of \(\Gamma_<\) lying on this line, we may cover a neighbourhood of the edge \(e \subseteq W\) with a union of such elements, which by construction is some subset \(B_e \in \mathscr{B}(A(W))\). Moreover, if we initially intersect \(B_v\) with a much smaller diameter set containing \(v\) (which we may always easily construct), we may easily ensure that \(B_e\) does not intersect \(W\) across any other face not incident with \(e\) (except perhaps near the vertices \(v\) and \(w\), but these are already correctly covered by \(B_v\) and \(B_w\)). Thus, taking the union \(B_1\) of these sets and \(B_0\), we have constructed a subset \(B_1 \in \mathscr{B}(A(W))\) with \(B_1 \subseteq W\) and covering a neighbourhood of its \(1\)-skeleton.

The above argument may be repeated to construct \(B_2\), \ldots, \(B_n \in \mathscr{B}(W)\), where \(B_n = W\), where in each step we observe that, for each cell of \(W\), we have that \(\Gamma_<\) is dense in the parallel subspace, using Lemma \ref{lem:invariant subspaces}. Thus, \(W \in \mathscr{B}(A(W))\) and we have proved (4a). Since we assume (1--3) throughout this section, and \(L\) is expansive, we have that \(\Omega\) consists of \(L^m\)-sub patterns for the appropriate power \(m\) noted at the start of the proof, by Theorem \ref{thm:main}.
\end{proof}

Since \(W\) (without any translation) is constructable from \(A^m(W)\) for a suitable power \(m \in \N\), it follows quickly from Theorem \ref{thm:PSA} that all rational translates (in the internal space) are LIDS. In fact, we have the following more specific result; note that an analogous statement and proof works whenever \(\tau \colon \Omega \to \mathbb{T}\) has finite fibres, and \(\tau \circ \sub = T \circ \tau\) where \(T = M_\mathbb{T}^m\).

\begin{proposition}\label{prop:polytopal PSA1}
There is an (injective) substitution map \(\sub \colon \Omega \to \Omega\), with inflation \(L^m\) for some \(m \in \N\). Moreover, this satisfies \(\tau \circ \sub = T \circ \tau\), where \(T[x] = [M^m(x)]\). We have that every \(\cps \in \Omega\) with \(\tau \cps = [u]\), for \(u_< \in \Q\Gamma_<\), satisfies \(\sub^\ell(\cps) = \cps + e\) for some power \(\ell \in \N\) and \(e \in \phy\). In particular, all such elements are LIDS.
\end{proposition}

\begin{proof}
Since, by the above proof, \(W\) is constructable from \(A^m(W)\), by Theorem \ref{thm:classification} we have a substitution map \(\sub\) for which \(\tau \circ \sub = T \circ \tau\), since we may take \(s = u = \mathbf{0}\) in that result. Here, we take the power \(M^m\) in place of \(M\), so the inflation map is \(L^m\) on \(\phy\). To simplify notation, we replace \(M\) with \(M^m\) at this point (so, for instance, we replace \(A^m\) with \(A\)).

Take \(\cps\) as in the statement and write \(u = v+s\), where \(v = u_\vee\) and \(s = u_\vee \in \Q\Gamma_<\). So \(s = \frac{1}{N}\gamma_<\) for some \(\gamma \in \Gamma\) and \(N \in \N\). As in the previous proof, consider a power \(A^\ell\) of \(A\) which acts as the identity on the quotient \(\frac{1}{N}\Gamma_< / \Gamma_<\). Thus, in particular, \(A^\ell(s) - s \in \Gamma_<\) so
\[
T^\ell[u] = [M^\ell(u)] = [L^\ell(v) + A^\ell(s)] = [L^\ell(v) + s] = [L^\ell(v) - v] + [v + s] = [v'] + [u] ,
\]
where \(v' = L^\ell(v) - v \in \phy\). Then, for any \(\cps' \in \Omega\) with \(\tau \cps' = [u]\), we have that
\[
\tau(\sub^\ell(\cps') - v') = \tau(\sub^\ell \cps') - [v'] = T^\ell(\tau \cps') - [v'] = [u] ,
\]
so the map \(\cps' \mapsto \sub^\ell(\cps') - v'\) sends the (finite) fibre \(\tau^{-1}[u]\) to itself. Since this map is injective, a power is the identity, so it follows that all such patterns are fixed under some power of \(\sub^\ell\), followed by a translation and thus they are LIDS.
\end{proof}

We will see a converse to the above in a later result, that the LIDS points \(\cps \in \Omega\) satisfy \(\tau \cps = [u]\) for \(u_< \in \Q\Gamma_<\).

\subsection{Proof that substitutional implies face direction invariant and rational}

Throughout this subsection we assume that \(\Omega = \Omega(\mathcal{S},W)\), for \(W\) polytopal (and, without loss of generality by Proposition \ref{prop:reducing to one colour}, single-coloured), consists of \(L\)-Sub cut and project sets, where \(L\) is expansive. If they are all, instead, merely \(L^m\) substitutional for some \(m > 1\) then we pass to the corresponding power \(M^m\). This does not affect the properties of FD-invariance or rationality, nor conditions (1--3) required of the linear automorphism \(M^m \colon \tot \to \tot\). Thus, the elements of \(\Omega\) are assumed to be \(L\)-Sub, so (1--4) of Theorem \ref{thm:main} hold.

\begin{lemma} \label{lem: substitutional => face invariant}
We have that \(W\) is FD-invariant with respect to some power \(A^\ell\).
\end{lemma}

\begin{proof}
Since the elements of \(\Omega\) are \(L\)-Sub, \(W\) is constructable from some translate of \(A(W)\), by Theorem \ref{thm:main}. By Lemma \ref{lem:Booleans same on NS}, we may thus cover \(\partial W\) by a finite set of translates of \(\partial A(W)\). Since \(A\) is a linear automorphism it is bijective as a function \(V \mapsto A(V)\) on codimension \(1\) subspaces. In particular, the supporting subspaces \(\sH_0\) of \(W\) are in bijective correspondence with those of \(A(W)\) under the map \(A\). Then, if \(A(V') \notin \sH_0\) for some \(V' \in \sH_0\), then there must be some \(V \in \sH_0\) with \(V \neq A(V')\) for all \(V' \in \sH_0\). Then it is impossible to cover \(\partial W\) using a finite number of translates of \(\partial A(W)\). Thus, \(A\) defines a bijection \(A \colon \sH_0 \to \sH_0\) and, since \(\sH_0\) is finite, some power \(A^\ell\) acts as the identity on \(\sH_0\).
\end{proof}

Given the above lemma, we will henceforth assume FD-invariance throughout the remainder of this subsection.  Thus, \(A\) is diagonalisable with real, irrational eigenvalues strictly less than \(1\) in modulus (by Lemma \ref{lem:invariant subspaces}). To simplify constructions, we will now position the window so that \(W\) it is constructable from \(A(W)\), rather than merely a translate of it being so.

\begin{lemma}\label{lem:notranslation}
For an appropriate translate of \(W\), we have that \(W\) is mutually constructable from \(A(W)\).
\end{lemma}

\begin{proof}
By Theorem \ref{thm:polytopal} (4) we already know that \(W - s\) is mutually constructable from \(A(W)\) for some \(s \in \intl\). Thus, \(W - (s + A(x))\) is mutually constructable from \(A(W-x) = A(W) - A(x)\) for all \(x \in \intl\), so we just need \(x \in \intl\) with \(s+A(x) = x\), equivalently \((\Id - A)(x) = s\). We have that \(\Id - A\) is invertible (since \(A\) is contractive, so \(1\) is not an eigenvalue), so such an \(x \in \intl\) exists, as required.
\end{proof}

Our proof of this direction of Theorem \ref{thm:polytopal} will end up establishing that \(\Omega\) consists of low complexity patterns, in the sense below. In the case that \(L\) is a similarity, we would know that the elements of \(\Omega\) are linearly repetitive, in which case the property would hold automatically. However, we are able to show this without restricting to the case of \(L\) being a similarity.

\begin{definition}
We say that \(\Omega\) has \textbf{low complexity} (and say it has property \textbf{C}, for short) if there is some \(C > 0\) so that for some (equivalently any) \(\cps \in \Omega\), the number of \(r\)-patches is bounded above by \(Cr\) for all \(r \geq 1\).
\end{definition}

The proof of rationality now breaks into two steps. Firstly, we will show that the stabilisers \(\Gamma^V\) of \(V \in \sH_0\) have sufficiently high rank to force \textbf{C}, making use of results from \cite{KoiWalII} (see also \cite{Wal24} for the case of non-convex windows). Amongst other properties, this will imply that singleton intersections of \(\Gamma_<\) translates of the supporting subspaces are rational, that is, they are contained in some \(\frac{1}{N}\Gamma_<\). We then show that all supporting hyperplanes are in fact rational translates of the supporting subspaces.

\begin{lemma}\label{lem:stabilisers contain basis}
Each \(V \in \sH_0\) has a basis in \(\Gamma^V_<\).
\end{lemma}

\begin{proof}
Fix an arbitrary \(V \in \sH_0\), with parallel supporting hyperplanes \(\{H_1, \ldots, H_\ell\} \subset \sH\), where we denote \(H \coloneqq H_1\). By FD-invariance, we have that \(\{A(H_1), \ldots, A(H_\ell)\}\) are the supporting hyperplanes of \(A(W)\) parallel to \(V\).

By Lemma \ref{lem:notranslation}, \(W \in \mathscr{B}(A(W))\), so by  Lemma \ref{lem:Booleans same on NS} we have that \(\partial W \subset \bigcup_{g \in Z} \partial A(W) + g_<\) for a finite set \(Z \subset \Gamma\). In particular, for a generalised `face' \(F = \partial W \cap H\), which is a compact set with non-empty interior in \(H\), we have \(F \subset \bigcup_{g \in S} \partial A(W) + g_<\). We may write \(\partial A(W)\) as the union of closed sets \(F_{H'} \coloneqq \partial A(W) \cap A(H')\), over all \(H' \in \sH\). Of course, \((F_{H'} + x) \cap F\) has zero measure for \(A(H') + x \neq H\), so we must have that \(F\) is contained in the union of \(F_{H_i} + g_<\) with \(A(H_i) + g_< = H\). By passing to a sufficiently large power, if necessary, we may ensure that the \(F_{H_i}\) have very small diameter relative to some ball contained in \(F \subset H\), small enough so that we need at least \(n\) distinct translates \(F_{H_i} + (g_j)_< \subset H\) used in covering \(F\), for some fixed \(i\) and \(j \in \{1,\ldots,n\}\), where \(\{(g_j-g_1)_< \mid j \in \{2,\ldots,n\}\}\) is a basis for \(V\) and each \(g_j \in \Gamma\). Thus, the projected lattice points \((g_j-g_1)_< \in \Gamma_< \cap V = \Gamma_<^V\) span \(V\), as required.
\end{proof}

Since each \(\Gamma^V_<\) contains a basis for \(V\), it is rank at least \(n-1\). In fact, by FD-invariance and the fact that \(A\) is a contraction, we in fact know that it must have strictly higher rank than this. We aim to show that their ranks are, in fact, generally as high as they can be, which will prove \textbf{C}, the low complexity condition. To this end, we introduce the following technical lemma:

\begin{lemma} \label{lem: subspaces spanned by stabilisers}
Let \(f = \{V_1,\ldots,V_n\} \subseteq \sH_0\) be a flag and let \(G_m \coloneqq \Gamma^{V_m}\). For \(m = 0\), \ldots, \(n\), we also define \(V_m' \coloneqq \bigcap_{i=1}^m V_i\) (with \(V_0' \coloneqq \intl\)), \(G_m' \coloneqq \bigcap_{i=1}^m G_i\) (with \(G_0' \coloneqq \Gamma\)), \(X_m \coloneqq \langle G_m \rangle_\R\) and \(X_m' \coloneqq \langle G_m' \rangle_\R\) (hence \(X_0' = \tot\)). Then
\begin{enumerate}
	\item \(X_i' = X_1 \cap X_2 \cap \cdots \cap X_i\) for \(i=0\), \ldots, \(n\);
	\item \(X_i' + X_{i+1} = \tot\) for \(i=0\), \ldots, \(n-1\).
\end{enumerate}
\end{lemma}

\begin{proof}
We have that (1) is true for \(i=0\) since \(X_0' = \tot\), which is also the vacuous intersection of subspaces of \(\tot\) in (1), by convention. Moreover, \(X_0' + X_1 = \tot + X_1 = \tot\), so (2) is trivial.

So suppose that both (1) and (2) hold for \(i=m-1\), with \(m \leq n\). By a standard dimension formula,
\[
\dim(X_1 \cap \cdots \cap X_{m-1}) + \dim(X_m) = \dim((X_1 \cap \cdots \cap X_{m-1}) \cap X_m) + \dim((X_1 \cap \cdots \cap X_{m-1})+X_m) .
\]
By induction, we thus have \(\dim(X_{m-1}') + \dim(X_m) = \dim(X_1 \cap \cdots \cap X_m) + k\). By definition, \(\dim(X_m) = \rk(G_m)\) and \(\dim(X_{m-1}') = \rk(G_{m-1}')\) so that
\begin{equation} \label{eq: dim formula}
\rk(G_{m-1}') + \rk(G_m) = \dim(X_1 \cap \cdots \cap X_m) + k .
\end{equation}
On the other hand, using the short exact sequence
\[
0 \to G_{m-1}' \cap G_m \to G_{m-1}' \times G_m \to G_{m-1}' + G_m \to 0 ,
\]
we have the rank formula \(\rk(G_{m-1}') + \rk(G_m) = \rk(G_m') + \rk(G_{m-1}' + G_m)\). Combining with Equation \ref{eq: dim formula}, we see that
\begin{equation} \label{eq: rk formula}
\rk(G_m') + \rk(G_{m-1}' + G_m) = \dim(X_1 \cap \cdots \cap X_m) + k .
\end{equation}
Since each \(G_i \leqslant X_i\), it is clear that \(\rk(G_m') \leq \dim(X_1 \cap \cdots \cap X_m)\). Similarly, since \(G_{m-1}' + G_m \leqslant \Gamma\), we have that \(\rk(G_{m-1}' + G_m) \leq k\). So both of these inequalities must in fact be equalities to ensure Equation \ref{eq: rk formula} holds, so that
\[
\rk(G_m') = \dim(X_1 \cap \cdots \cap X_m) \ \ \text{ and } \ \ \rk(G_{m-1}' + G_m) = k .
\]
Since \(X_m' \coloneqq \langle G_m' \rangle \leqslant X_1 \cap \cdots \cap X_m\) and \(\dim(X_m') = \rk(G_m') = \dim(X_1 \cap \cdots \cap X_m)\), we have \(X_m' = X_1 \cap \cdots \cap X_m\), as required by claim (1) of the lemma for \(i=m\).

To prove (2), now suppose additionally that \(m < n\). So (1) holds for \(i=m\), by the above, and (2) holds for \(i = m-1 < n-1\). We have that \(G_m' + G_{m+1}\) is a lattice in its span, denoted
\[
Z \coloneqq \langle G_m' + G_{m+1} \rangle = \langle G_m' \rangle + \langle G_{m+1} \rangle = X_m' + X_{m+1} = (X_1 \cap \cdots \cap X_m) + X_{m+1}\ .
\]
Recall (Lemmas \ref{lem:invariant subspaces} and  \ref{lem: substitutional => face invariant}) that \(A\) is diagonalisable, with eigenvectors that may be chosen to lie on the lines given as the intersections of all but one of the hyperplanes of the flag \(f\). Take any such line \(Y\), where \(A(Y) = Y\), say \(A(y) = \lambda y\), where \(|\lambda| < 1\) since \(A\) is contractive. We claim that, for any \(V_i \in f\), if \(Y \subseteq V_i\) then \(Y \subseteq X_i\). Indeed, take any \(y \in Y\). By Lemma \ref{lem:stabilisers contain basis}, \(\Gamma_<^{V_i}\) contains a basis for \(V_i\) and hence there is at least some \(x \in X_i\) so that \(x_< = y\). We may write \(x = v + y\) for \(v = x_\vee \in \phy\). Then \(M^{-n}(\lambda^n x) = \lambda^n L^{-n}(v) + \lambda^n A^{-n}(y) = \lambda^n L^{-n}(v) + y\). Since \(L\) is expansive, \(L^{-n}\) is contractive so that \(L^{-n}(v) \to \mathbf{0}\), and so of course \(\lambda^n L^{-n}(v) \to \mathbf{0}\) too as \(n \to \infty\). Thus, \(y_n \coloneqq M^{-n}(\lambda^n x) \to y\) as \(n \to \infty\). But \(X_i\) is invariant under \(M\) and \(M^{-1}\), so each \(y_n \in X_i\). Since \(X_i\) is closed, we have that \(y \in X_i\), as required.

Thus, for each eigenline \(Y \subseteq \intl\) defined by the flag, we must have that \(Y \subseteq X_1 \cap X_2 \cap \cdots \cap X_m = X_m'\), or \(Y \subseteq X_{m+1}\) (since every such \(Y\) belongs to all but one of the \(X_i\) by the above). These lines span \(\intl\), so we see that \(\intl \subseteq Z = X_m' + X_{m+1}\). Now, since \(Z = (X_1 \cap \cdots \cap X_m) + X_{m+1}\), and each \(X_i\) is invariant under \(M\), we also have that \(Z\) is invariant under \(M\). Moreover, the sub-lattices \(G_m' = G_1 \cap \cdots \cap G_m\) and \(G_{m+1}\) are invariant under \(M\), and \(G_m + G_{m+1}'\) is a lattice in \(Z\) by the above. It follows that \(M|_Z\) has determinant, or product of eigenvalues, of modulus \(1\). Since \(Z \supseteq \intl\), to achieve this it must also contain all of the other eigenvalues (with multiplicity) of \(M\), hence \(Z = \tot\). Since \(Z = X_m' + X_{m+1}\), we have shown (2) for value \(i=m\), as required. The lemma thus holds, by induction.
\end{proof}

\begin{theorem} \label{thm:low complexity}
The cut and project scheme satisfies property \textbf{C}.
\end{theorem}

\begin{proof}
Take an arbitrary flag \(f \subseteq \sH_0\). Then using a standard dimension formula and iteratively applying Lemma \ref{lem: subspaces spanned by stabilisers}, we have
\begin{align*}
  & \rk(G_1) + \rk(G_2) + \rk(G_3) + \cdots + \rk(G_n)\\
= & \dim(X_1) + \dim(X_2) + \dim(X_3) + \cdots + \dim(X_n) \\
= & \dim(X_1 + X_2) + \dim(X_1 \cap X_2) + \dim(X_3) + \cdots + \dim(X_n) \\
= & \dim(X_1' + X_2) + \dim(X_2') + \dim(X_3) + \cdots + \dim(X_n) \\
= & k + \dim(X_2' + X_3) + \dim(X_2' \cap X_3) + \dim(X_4) + \cdots + \dim(X_n)\\
= & k + k + \dim(X_3') + \dim(X_4) + \cdots + \dim(X_n) \\
= & \cdots \\
= & (n-2)k + \dim(X_{n-1}') + \dim(X_n) \\
= & (n-2)k + \dim(X_{n-1}' + X_n) + \dim(X_{n-1}' \cap X_n) \\
= & (n-2)k + k + \dim(X_n') = (n-1)k \ .
\end{align*}
In the last line, note that \(X_n' = \langle G_1 \cap \cdots \cap G_n \rangle\) and that \(G_1 \cap \cdots \cap G_n = \{\mathbf{0}\}\), since no element of \(\Gamma\) can stabilise all elements of \(f\) simultaneously (since it is a flag). Since the flag was arbitrary, the sum of stabiliser ranks is equal to \((n-1)k\) for each of them and so \textbf{C} holds by \cite[Theorem 5.7]{KoiWal21} (or see \cite[Corollary 2.11]{Wal24} for the case of arbitrary, non-convex polytopal windows).
\end{proof}

\begin{corollary} \label{cor:rational flag intersections}
There exists an \(M \in \N\), such that the following holds: for any flag \(f = \{V_1, \ldots, V_n\} \subseteq \sH_0\), any \(N \in \N\) and any \(\gamma_1\), \ldots, \(\gamma_n \in \frac{1}{N}\Gamma\), the `vertex'
\[
\{v\} \coloneqq \bigcap_{i=1}^n (V_i + (\gamma_i)_<)
\]
is such that \(v \in \frac{1}{M} \Gamma_<\).
\end{corollary}

\begin{proof}
Replacing the lattice \(\Gamma\) with \(\frac{1}{N}\Gamma\) does not affect property \textbf{C}, by the complexity formula for \cite[Theorem 6.1]{KoiWal21} since the formula only concerns the stabiliser ranks. Then the result follows directly from \cite[Proposition 2.20]{Wal24} (also see \cite{KoiWalII}).
\end{proof}

For every vertex \(v \in \mathrm{Ver}(W)\), we have that \(\{v\} = \bigcap_{i = 1}^n H_i\) for a flag \(f = \{H_1,\ldots,H_n\}\) of supporting hyperplanes \(H_i \in \sH\). Therefore, the above result implies rationality of the window if we can ensure that each supporting hyperplane is a rational translate of a supporting subspace. We show this in the following results:

\begin{notation}
We let \(\approx\) be the equivalence relation on affine hyperplanes of \(\intl\) given by letting \(H_1 \approx H_2\) if there is some \(g \in \Gamma_<\) with \(H_1 = H_2 + g\). We let \(\mathfrak{X}\) denote the equivalence classes of all translates of supporting hyperplanes of \(W\), that is
\[
\mathfrak{X} \coloneqq \{H + x \mid H \in \sH_0, x \in \intl \} / \approx .
\]
For \([H] \in \mathfrak{X}\), we let \(A [H] \coloneqq [A(H)]\), which gives a well-defined bijection \(A \colon \mathfrak{X} \to \mathfrak{X}\). Indeed, if \([H_1] = [H_2]\) then \(H_1 = H_2 + g\) for \(g \in \Gamma_<\), so \(A(H_1) = A(H_2+g) = A(H_2) + A(g)\). Thus, since \(A(g) \in \Gamma_<\) (since \(A(\Gamma_<) = \Gamma_<\)), we have \(A(H_1) \approx A(H_2)\) and the map is well-defined. Given \([H+x] \in \mathfrak{X}\) we have \(A[A^{-1}(H+x)] = [H+x]\), where \([A^{-1}(H+x)] \in \mathfrak{X}\), by FD-invariance, so the map is surjective. Finally, if \(A[H_1] = A[H_2]\) then \(A(H_1) = A(H_2) + g\) for some \(g \in \Gamma_<\), thus \(H_1 = H_2 + A^{-1}(g)\) hence \([H_1] = [H_2]\), since \(A^{-1}(g) \in \Gamma_<\), so the map is injective.
\end{notation}

\begin{lemma}
The map \(A \colon \mathfrak{X} \to \mathfrak{X}\) has finite orbits when applied to elements of \(\sH/\approx\).
\end{lemma}

\begin{proof}
We recall that we choose a translation of the window so that \(W\) is constructable from \(A(W)\) and that \(A\) is FD-invariant (see Lemmas \ref {lem: substitutional => face invariant}, \ref{lem:notranslation} and comments below them). As seen in the proof of Lemma \ref{lem: substitutional => face invariant}, it follows from constructability and Lemma \ref{lem:Booleans same on NS} that, for each \(H \in \sH\), we have \(A(H) + g \in \sH\) for some \(g \in \Gamma_<\). Thus, \(A \colon \mathfrak{X} \to \mathfrak{X}\) restricts to a map \(A \colon \sH /\approx \to \sH / \approx\). Since \(\sH\) is finite and \(A\) is a bijection, it follows that \(A\) has finite orbits on \(\sH / \approx\).
\end{proof}

By the above, there is a power of \(A\) so that \(A \colon \mathfrak{X} \to \mathfrak{X}\) is the identity on \(\sH / \approx\). So, without loss of generality, by passing to this power we now assume that, for all \(H \in \sH\), we have \(A(H) = H + g_H\) for some \(g_H \in \Gamma_<\).

\begin{corollary}\label{cor:rational hyperplanes}
All supporting hyperplanes are rational, that is, for each \(H \in \sH\) we have that \(H = V(H) + \frac{1}{N}g\) for some \(g \in \Gamma_<\) and \(N \in \N\).
\end{corollary}

\begin{proof}
Let \(H \in \sH\) be arbitrary and \(V \coloneqq V(H)\). By the above, \(A(H) = H + g\) for some \(g \in \Gamma_<\). Since \(A\) is contractive, \(\Id - A \colon \Gamma_< \to \Gamma_<\) is an injection so that \((\Id - A)(\Gamma_<)\) is a finite index subgroup of \(\Gamma_<\). It follows that \(x \coloneqq (\Id - A)^{-1}(g) \in \frac{1}{N}\Gamma_<\) for some \(N \in \N\).

Now, \(A(H+x) = A(H) + A(x) = H + g + A(x) = H+x\), so \(H+x\) is fixed by \(A\). By FD-invariance \(H+x = V\), since if \(H+x\) does not contain the origin then the parallel \(A(H+x)\) is closer to the origin, since \(A\) is a contraction, contradicting \(H+x = A(H+x)\). So \(H = V-x\) for \(x \in \Q\Gamma_<\), as required.
\end{proof}

Since each supporting hyperplane is a rational translate of a subspace, the result now follows from \textbf{C} and earlier results on the intersection points of flags of low complexity schemes:

\begin{corollary}\label{cor:window rational}
We have that \(\mathrm{Ver}(W) \subset \Q\Gamma_<\). In particular, the window is rational.
\end{corollary}

\begin{proof}
By Corollary \ref{cor:rational hyperplanes}, each \(H \in \sH\) may be written as \(H = V(H) + z_H\) for some \(z_H \in \frac{1}{N}\Gamma_<\). By Lemma \ref{cor:rational flag intersections}, an intersection point of a rational translates of supporting subspaces gives rational points in \(\frac{1}{M}\Gamma_<\) for some \(M \in \N\). Since each supporting hyperplane is a rational translate of a supporting subspace, the vertices in particular are rational.
\end{proof}

We have thus proved that a polytopal substitutional cut and project set is FD-invariant and rational, completing the second direction of Theorem \ref{thm:polytopal}. With a similar argument to the above, we also have the following which, along with Proposition \ref{prop:polytopal PSA1}, proves Theorem \ref{thm:polytopal PSA}. For the result below, we now revert to the convention that \(W\) is positioned so that \(\mathbf{0} \in \mathrm{Ver}(W)\).

\begin{proposition}\label{prop:polytopal PSA2}
Any vertex of a polytope in \(\mathscr{B}(W)\) is rational, that is, it belongs to \(\Q\Gamma_<\). Thus, if \(\cps \in \Omega(\mathcal{S},W)\) is LIDS, with respect to some power \(L^m\) (\(m \in \N\)) of the inflation, then \(\tau(\cps) = [u]\) with \(u_< \in \Q\Gamma_<\).
\end{proposition}

\begin{proof}
Recall that, since Lemma \ref{lem:notranslation} in this subsection, we have been using a repositioning \(W-s\) of the window (although dropping \(s\) from our notation) for which \(W-s\) is mutually constructable from \(A(W-s)\), instead of a positioning as in Theorem \ref{thm:polytopal PSA} for which \(\mathbf{0} \in \mathrm{Ver}(W)\). However, we saw in Corollary \ref{cor:window rational} that also \(\mathrm{Ver}(W-s) \subset \Q\Gamma_<\), so \(s \in \Q\Gamma_<\). By an analogous argument to that proof, it follows that all polytopes of \(\mathscr{B}(W)\) (i.e., those constructable from \(W\)) have vertices in \(\Q\Gamma_<\).

Take any \(\cps \in \Omega\), say with \(\tau(\cps) = [u]\). By Theorem \ref{thm:PSA}, if \(\cps\) is LIDS for some power of the inflation, it is necessary that \(W-s\) is (mutually) constructable from \(A^p(W-s)\) for some \(p \in \N\), where now \(s \coloneqq u_<\). We have \(-A^p(s) \in \mathrm{Ver}(A^p(W-s))\), since \(\mathbf{0} \in \mathrm{Ver}(A^p(W))\) and, as above, all polytopes in \(\mathscr{B}(W-s)\) have vertices in \(\Q\Gamma_< - s\). Thus, \(-A^p(s) \in \Q\Gamma_< - s\), that is, \(s - A^p(s) = (\Id - A^p)(s) \in \Q\Gamma_<\). Now, since \(A^p\) is contractive, we have that \(\Id - A^p\) is invertible (on \(\intl\)) and, since \((\Id - A^p)(\Gamma_<) \leq \Gamma_<\) with finite index, we have that \((\Id - A^p)^{-1}(\Gamma_<) \leq \frac{1}{N}\Gamma_<\) for some \(N \in \N\). Hence, \(s \in (\Id - A^p)^{-1}(\Q\Gamma_<) \leq \Q\Gamma_<\), as required.
\end{proof}

Before considering some examples, we show that under these assumptions \textbf{LR} must also hold. This follows from the characterisation of \textbf{LR} for polytopal windows (\cite{KoiWalII} and \cite{Wal24}) and a theorem of Perron \cite{Per21} on systems of numbers preserved by a linear automorphism:

\begin{theorem}
Suppose that the conditions of Theorem \ref{thm:polytopal} hold. Then the cut and project sets are \textbf{LR}.
\end{theorem}

\begin{proof}
By Theorem \ref{thm:low complexity}, \textbf{C} holds. By Corollary \ref{cor:rational hyperplanes}, \(W\) is also \textbf{weakly homogeneous} meaning that, for some translate of the window and some \(N \in \N\), for each \(H \in \sH\) we have \(H = V(H) + g_H\) for \(g_H \in \frac{1}{N}\Gamma_<\), see \cite{KoiWalII}. By \cite[Theorem A]{KoiWalII} (or \cite[Theorem A]{Wal24} for general, non-convex polytopes), \textbf{LR} is equivalent to \textbf{C} and a Diophantine condition on the lattice holding. In the decomposable case (with stabilisers of different hyperplanes of different ranks), this condition requires some care to state, although by \cite[Remark 3.16]{Wal24} it is equivalent to the following: for any flag \(f \subseteq \sH_0\), consider the \(1\)-dimensional subspaces \(Y_1\), \ldots \(Y_n\), each given by intersecting all but one of the supporting hyperplanes in \(f\). Define \(G_i \coloneqq \Gamma^{Y_i}\), that is, \(G_i\) is the subgroup of \(\Gamma\) projecting to \(Y_i\). Let \(k_i \coloneqq \rk(G_i)\). Then the necessary (and sufficient) Diophantine condition to ensure \textbf{LR} is that there exists some \(C > 0\) so that
\[
\|\gamma_<\| \geq C\|\gamma\|^{-(k_i-1)}
\]
for each non-zero \(\gamma \in G_i\), which we require for each \(G_i\). Equivalently, take a basis of \((G_i)_<\) which identifies \(Y_i \cong \R\) by sending the \(k_i\)th basis vector to \(1\), and the others to \(\alpha_1\), \ldots, \(\alpha_{k_i-1} \in \R\). Then the necessary Diophantine condition (on \(G_i\)) is that \(\{\alpha_i\}_{i=1}^{k_i-1}\) is a badly approximable system of numbers in the usual sense. The linear isomorphism \(A \colon Y_i \to Y_i\) induces an automorphism of \(\langle \alpha_1, \ldots, \alpha_{k_i-1}, 1 \rangle_\Z\). It now follows by a theorem of Perron \cite{Per21} that the system is badly approximable.
\end{proof}

\section{Definition of classical tiling substitutions} \label{sec:L-sub versus classical}\label{sec:compare tilings}

In this final section, we compare the notion of a pattern being \(L\)-sub to standard notions, for Delone (multi)sets and tilings, being generated by substitution rules. Throughout this section, \(L \colon E \to E\) will be assumed to be expansive.

\subsection{\(L\)-sub patterns versus substitutional Delone (multi)sets}

The following is a widely used notion \cite{LW03, Lee23} of a Delone (multi)set being substitutional:

\begin{definition} \label{def:SDM}
A Delone (multi)set \(\cps = (\cps_i)_{i \leq m}\) is called a \textbf{substitution Delone multiset} (\textbf{SDM}) if \(\cps\) is a Delone (multi)set and there exists an expansive map \(L \colon \R^d \to \R^d\) and finite sets \(\mathcal{D}_{ij}\) for \(i\), \(j \leq m\) such that, for each \(i \leq m\),
\begin{equation}\label{eq:SDM}
\Lambda_i = \bigcup_{j=1}^\ell (L \Lambda_j + \mathcal{D}_{ij})\ ,
\end{equation}
where the unions on the right-hand side are disjoint.
\end{definition}

Firstly, we observe that an SDM, with expansion \(L\), is clearly \(L\)-sub, where Equation \ref{eq:SDM} provides a local rule \(L\cps \LD \cps\). More precisely, suppose the maximal norm of element over all \(\mathcal{D}_{ij}\) is \(c\). Suppose that \((L\cps)[x,c] = (L\cps)[y,c]\). Then, for each \(j\) and \(z \in B_c\), we have \(x-z \in L\cps_j\) if and only if \(y-z \in L\cps_j\). Thus, for all \(i\) and \(j\), we have \(x \in L\cps_j + D_{ij}\) if and only if \(y \in L\cps_j + D_{ij}\), that is, \(x \in \cps_i\) if and only if \(y \in \cps_i\), showing that \(L\cps \LD \cps\), with derivation radius \(c\).

Of course, it is not true that every Delone (multi)set \(\cps\) which is \(L\)-sub is an SDM, since the former captures the idea that \(\cps\) is any element in a hull of substitutional patterns, whereas Definition \ref{def:SDM} captures the notion of a pseudo self-affine point, substituting to itself under a (very particular kind of) substitution rule. However, we will show that a repetitive \(L\)-sub pattern, at least up to a `local recolouring' (see Definition \ref{def:recolouring}), is one generated in the hull of an SDM.

For the result below recall that, generally, for a space \(\Omega\) of patterns (in this case, Delone (multi)sets), a substitution is a map \(\sub \colon \Omega \to \Omega\), defined by \(\sub = S \circ L\) where \(S\) is a local derivation map. All \(L\)-sub patterns (for repetitive patterns, or \(L\) expansive, as here) uniquely define such a substitution map.

\begin{proposition}\label{prop:fixed point of any sub}
For any FLC pattern space, with substitution \(\sub \colon \Omega \to \Omega\), we have that \(\sub^n\) has a fixed point for some \(n \in \N\).
\end{proposition}

For a proof of the above, see \cite{Wal25}, which applies to non-repetitive pattern spaces too (and defines a notion of \(\Omega\) being FLC, although for the repetitive examples considered here this is automatic).

\begin{definition}\label{def:recolouring}
For Delone (multi)sets \(\cps\) and \(\cps'\), we say that \(\cps'\) is a \textbf{local recolouring} of \(\cps\) if \(\mathrm{sup}(\cps) = \mathrm{sup}(\cps')\) and \(\cps \MLD \cps'\).
\end{definition}

Usually, we will recolour by labelling points with their \(r\)-patches, which is clearly an MLD relabelling. Of course, a local relabelling extends over all Delone sets in a hull (or one may define the notion for a space of such patterns, analogously to definition of a derivation map).

\begin{definition}\label{def:colour sub}
We say that a substitution rule \(\sub = S \circ L \colon \Omega \to \Omega\) on a hull of Delone (multi)sets is \textbf{colour-defined} (or \textbf{CD}, for short) if the subdivision rule \(S\) may defined by
\begin{equation}\label{eq:colour sub}
S(L\cps)_i = \bigcup_{j=1}^\ell (L \Lambda_j + \mathcal{D}_{ij}) .
\end{equation}
If the unions on the right-hand side are disjoint, we call a CD substitution rule \textbf{non-overlapping}.
\end{definition}

\begin{remark}\label{rem:CD subs versus standard}
In other words, a substitution rule is CD and non-overlapping when it is a rule of the form given in Definition \ref{def:SDM}. We use the non-overlapping condition since Definition \ref{def:SDM} is how it is often stated in the literature, although note that it seems quite natural to not require this condition generally, analogously to how for substitutions of tilings (such as the Penrose or Ammann--Beenker tile substitutions) one often considers rules which, when applied to the tilings, can have non-disjoint substitutes of different tiles. The only restriction in a substitution being colour-defined is that the rule must only replace points with clusters depending just on the colour of points i.e., not using information from other nearby points. Thus, by adding extra information to points, using local patches, it is not hard to see why the following holds:
\end{remark}

\begin{proposition}\label{prop:simplicity of subs}
For any substitution rule \(\sub \colon \Omega \to \Omega\) on a hull of Delone (multi)sets, by applying a local recolouring of points (by labelling points by their \(r\)-patches, for sufficiently large \(r\)), the resulting substitution map on the hull of recoloured Delone (multi)sets may be described by a CD and non-overlapping substitution.
\end{proposition}

\begin{proof}
We begin by showing we may enhance the colouring to define Delone (multi)sets with (non-overlapping) colours allowing for a CD substitution rule. We will then refine the argument to show it can be made non-overlapping.

Let \(\kappa > 0\) be such that, for any \(\cps' \in \Omega\), every point of \(\cps\) is within distance \(\kappa\) of a point of \(L\cps'\), where \(\cps = \sub(\cps') = S(L\cps')\). Let \(c\) be the derivation radius of \(S \colon L\Omega \to \Omega\). Recall that \(\lambda > 1\) is such that \(\lambda B \subseteq LB\), where \(B\) is the unit ball. Then, for any \(r > 0\) and any \(\cps_1\), \(\cps_2 \in \Omega\), if \((L\cps_1)[\mathbf{0},r+\kappa+c] = (L\cps_2)[\mathbf{0},r+\kappa+c]\) we quickly derive that
\begin{equation}\label{eq:colour derivation}
S(L\cps_1)[z,r] = S(\cps_2)[z,r] \text{ for all } z \in B_\kappa .
\end{equation}
Given any \(\cps' \in \Omega\), and \(\cps = \sub \cps'\), say that \(x \in \cps\) is \textbf{covered by} \(Ly \in L\cps'\) if \(Ly-x \in B_\kappa\). By our assumption on \(\kappa\), each point of \(\cps\) is covered by a point (perhaps more than one) from \(L\cps'\). Moreover, by Equation \ref{eq:colour derivation}, if \(Ly\), \(Ly' \in L\cps'\) are centred at the same \((r+\kappa+c)\)-patches, then \(Ly\) controls a point \(Ly+z\), with colour \(i\), if and only if \(Ly'\) controls the point \(Ly'+z\), also with colour \(i\). It easily follows that, if we colour the Delone (multi)sets using labels defined by their \(r\)-patches, the substitution rule is CD, provided that the rule is well-defined, which requires that any \(r\)-patch \(\cps'[y,r]\) determines the patch \(\cps[Ly,r+\kappa+c]\) (since then, if \(Ly\) and \(Ly'\) are labelled with the same \(r\)-patches in \(L\cps'\), then they determine the labels of the \(r\)-patches they control in \(\cps\)).

Now, if \(\cps_1[y,r] = \cps_2[y',r]\), clearly \((L\cps_1)[Ly,LB_r] = (L\cps_2)[Ly',LB_r]\). Thus, the claim holds provided that \(r\) is large enough so that \(LB_r \supseteq r+\kappa+c\). Recall that \(\lambda > 1\) is defined so that \(LB_r \supseteq B_{\lambda r}\), so the above all holds once \(\lambda r \geq r + \kappa + c\), that is, \(r \geq \frac{\kappa + c}{\lambda-1} > 0\). Thus, by labelling each point with its \(r\)-patch for \(r \geq \frac{\kappa + c}{\lambda-1}\), the resulting substitution map will be CD. Explicitly, fixing some \(\cps' \in \Omega\) and \(\cps \coloneqq \sub \cps\), we may define
\[
\mathcal{D}_{ij} = \{x-y \mid y \in L\cps_j', x \in \cps_j \text{ with } x-y \in B_\kappa\}
\]
which, by the above, does not depend on the choice of \(\cps' \in \Omega\). Here, each \(\cps_j\) is the set of points of \(\cps\) at the centres of \(r\)-patches of type \(j\).

Since a point \(x \in \sub(\cps')\) may be controlled by multiple points of \(L\cps'\), though, the rule thus defined may be overlapping. To correct this, begin by letting \(\mathcal{D} \coloneqq \bigcup_{i,j = 1}^\ell \mathcal{D}_{ij}\), which is a finite set. Put an arbitrary total order \(\prec\) on \(\mathcal{D}\). We now say that a point \(x \in \cps = \sub\cps'\) is \textbf{controlled by} \(Ly \in L\cps'\) if \(v = x-Ly \in \mathcal{D}\) and that, if \(v' = x-Ly' \in \mathcal{D}\) for any other \(Ly' \in L\cps'\), then \(v \preceq v'\). Thus, since each \(x \in \cps\) is controlled by \(Ly = v-x \in L\cps'\) for a minimal \(v \in \mathcal{D}\), it is clear that each \(x \in \cps\) is controlled by a unique point of \(L\cps'\). Moreover, if the enhanced label, given by \(\cps'[y,r]\), of the point \(Ly \in L\cps'\) is able to determine the other positions points \(Ly' \in L\cps'\) which cover a point \(x\) that \(Ly\) also covers (that is, with \(x-Ly \in B_\kappa\)), then it can also determine whether or not it controls it. Thus, it is sufficient that \(\cps'[y,r]\) contains the information of which points of \(L\cps'\) are occupied within radius \(\kappa\) of \(x\). All such points are at most distance \(2\kappa\) from \(y\), so it is sufficient, by a similar argument to above, that \(\lambda r \geq 2\kappa\), so this holds provided we have coloured the point with a sufficiently large \(r\), as required.
\end{proof}

\begin{corollary}
Any SDM, with expansion \(L\), is \(L\)-sub. Conversely, given an FLC hull \(\Omega\) of Delone (multi)sets, equipped with an expansive substitution map \(\sub \colon \Omega \to \Omega\) (e.g., one induced by a particular \(L\)-sub Delone (multi)set \(\cps\) being \(L\)-sub, for some expansive \(L\)), at least after applying a local recolouring, \(\cps\) is an SDM for some \(\cps \in \Omega\), with expansion \(L^n\) for some \(n \in \N\).
\end{corollary}

\begin{proof}
The fact that an SDM is \(L\)-sub is trivial and was discussed earlier in this section. Conversely, given a substitution map \(\sub \colon \Omega \to \Omega\), take \(n \in \N\) so that \(\sub^n\) has a fixed point \(\cps \in \Omega\), using Proposition \ref{prop:fixed point of any sub}. We have that \(\sub^n = S' \circ L^n\) is also an expansive substitution map (where \(S'\) may be thought of as applying, \(n\) times, the map \(S\), see for instance \cite{Wal25}), with inflation \(L^n\) and, by Proposition \ref{prop:simplicity of subs}, is given by a CD and non-overlapping substitution, perhaps after applying a local recolouring. Applying Definition \ref{def:colour sub}, we have \(\cps_i = S'(L^n \cps)\) on the left-hand side of Equation \ref{eq:colour sub}, where the unions are disjoint on the right-hand side, thus Definition \ref{def:SDM} applies and \(\cps\) is SDM, as required.
\end{proof}

The above shows how the notion of \(L\)-sub considered here is, in an essential sense, a natural extension of being in the hull of an SDM. Moreover, the types of substitutions one gets only differ insofar as points of SDMs must be replaced by (non-overlapping) clusters with a rule utilising only colours of points, rather than local surroundings to a bounded radius (see Remark \ref{rem:CD subs versus standard}). However, the latter is a very helpful flexibility in practice, as we have seen in many previous results and, in particular, allowing a more natural approach in the classification problem for substitutional polytopal windows. That said, knowing that a substitution rule is given by a (non-overlapping) CD substitution has the following interesting consequence:

\begin{theorem}\label{thm:recolour to GIFS}
Suppose that \(\Omega = \Omega(\mathcal{S},W)\) is a hull of \(L\)-sub cut and project sets. Then (1--4) of Theorem \ref{thm:main} apply and, in particular, we have a corresponding contraction \(A \colon \intl \to \intl\).

Suppose that, up to a translation, \(W\) is the attractor of a GIFS (see Definition \ref{def:GIFS}), with contraction maps \(x \mapsto A(x) + g\) for \(g \in \Gamma_<\). Then there is a CD substitution rule \(\sub \colon \Omega \to \Omega\). If the union in the GIFS, defining \(W_i\) from a union of sets of the form \(A(W_j) + X_{ij}\), have mutually disjoint interiors, then the substitution map is also non-overlapping.

Conversely, if \(\sub \colon \Omega \to \Omega\) is an (expansive, with inflation \(L\)) CD substitution map, then the support of (a translate of) \(W\) may be given as a finite union of subsets \(W_1'\), \ldots, \(W_\ell' \in \mathscr{B}(W)\), so that \(W' = (W_i')_{i=1}^\ell\) is the attractor of a GIFS, with contraction maps \(x \mapsto A(x) - g\) for \(g \in \Gamma_<\), where the elements of \(\Omega(\mathcal{S},W')\) are a local recolouring of those from \(\Omega\) (in particular, they are MLD). If the substitution is non-overlapping then the sets \(A(W_j) +  X_{ij}\), whose union gives \(W_i\), have disjoint interiors.
\end{theorem}

\begin{proof}
Firstly, suppose that \(W+s\) is the attractor of a GIFS as described. Choose, without loss of generality, a non-singular translate of \(W+s\) and let \(\cps' \coloneqq \cps(\mathcal{S},W+s)\). By Lemma \ref{lem:rescaling of cps is cps}, we have that \(L\cps_j = \cps(\mathcal{S},A(W_j+s))\), with window \(A(W_j+s)\) still in non-singular position. For each \(\gamma \in \Gamma\), we have that \(L\cps_j + \gamma_\vee = \cps(\mathcal{S},A(W_j+s)+\gamma_<)\). In particular, taking \(\mathcal{D}_{ij}\) as the union of \(\gamma_\vee\), where the \(\gamma_<\) are the \(g \in X_{ij} \in \Gamma_<\) for which
\[
W_i = \bigcup_{j=1}^\ell A(W_j) + X_{ij}
\]
in the GIFS defining \(W_i\) from the \(W_j\) (that is, \(\mathcal{D}_{ij}^\star = X_{ij}\)), we see that
\[
\cps(\mathcal{S},W_i+A(s)) = \bigcup_{j=1}^\ell (L\cps_j + \mathcal{D}_{ij}) .
\]
Thus, defining \(\cps\) by \(\cps_i \coloneqq \cps(\mathcal{S},W_i+A(s))\), we have that \(\cps = \cps(\mathcal{S},W+A(s))\). Since \(\cps\) is a cut and project set with translated window, we have \(\cps \LIs \cps'\) by Lemma \ref{lem: LI <=> equal windows}. And clearly, by construction, \(\cps\) is given as a colour-defined substitution of \(\cps'\). Moreover, if the \(A(W_j)+X_{ij}\) have disjoint interiors for distinct \(j\), then the \(L\cps_j + \mathcal{D}_{ij}\) are disjoint for distinct \(j\). Indeed, otherwise, there exists some \(x \in (L\cps_j + v) \cap (L\cps_k + w)\), for \(v \in \mathcal{D}_{ij}\) and \(w \in \mathcal{D}_{ik}\). Applying the star map, \(x^\star \in A(W_j+s) + v^\star\) and \(x^\star \in A(W_k+s) + w^\star\) where \(v^\star \in X_{ij}\) and \(w^\star \in X_{ik}\). Since \(A(W+s)\) is in non-singular position, \(x^\star\) belongs to the interiors of these sets so, shifting by \(-A(s)\), we see that they have intersecting interiors, a contradiction, so the substitution is defined above is non-overlapping.

Conversely, suppose that we have a CD substitution \(\sub \colon \Omega \to \Omega\) with inflation \(L\). Take any \(\cps' = \cps(\mathcal{S},W) \in \Omega\) where, without loss of generality, we assume that \(W\) is in non-singular position. Then, for \(\cps \coloneqq \sub \cps'\), by Definition \ref{def:colour sub},
\begin{equation}\label{eq:GIFS1}
\cps_i = \bigcup_{j=1}^\ell L\cps'_j + \mathcal{D}_{ij} = \bigcup_{j=1}^\ell \cps(\mathcal{S},A(W_j)) + \mathcal{D}_{ij}
\end{equation}
by Lemma \ref{lem:rescaling of cps is cps}. We may assume, by replacing the substitution \(\cps \mapsto \sub \cps\) with \(\cps \mapsto (\sub \cps) - v\) for any \(v \in \mathrm{sup}(\cps)\), that \(\mathrm{sup}(\cps) \subset \Gamma_\vee\), since the same is true of \(\cps'\) and \(\cps \LIs \cps'\) (so they have the same return vectors). Thus, since each \(\cps(\mathcal{S},A(W)) \subset \Gamma_\vee\), we also have that each \(\mathcal{D}_{ij} \subseteq \Gamma_\vee\). It quickly follows that
\begin{equation}\label{eq:GIFS2}
\cps_i = \bigcup_{j=1}^\ell \cps(\mathcal{S},A(W_j) + X_{ij}) = \cps\left(\mathcal{S},\bigcup_{j=1}^\ell A(W_j) + X_{ij}\right)
\end{equation}
where we define \(X_{ij} \coloneqq \mathcal{D}_{ij}^\star\). Thus, \(\cps\) is a cut and project set, with colour \(i\) window given as above. Since \(\cps \LIs \cps'\), it follows from Lemma \ref{lem: LI <=> equal windows} that each \(W_i + s = \bigcup_{j=1}^\ell A(W_j) + X_{ij}\), for some \(s \in \intl\). Since \(A\) is a contraction, by a standard argument we have seen multiple times, we may replace this with \(W_i + s' = \bigcup_{j=1}^\ell A(W_j + s') + X_{ij}\) for an appropriate \(s' \in \intl\), that is, \(W+s'\) is the attractor of a GIFS of the required form.

If \(\sub\) is also non-overlapping, then we can take the unions in Equation \ref{eq:GIFS1}, and thus also the first union of Equation \ref{eq:GIFS2}, to be disjoint. It follows that the \(A(W_j) + X_{ij}\), for a fixed \(i\), have mutually disjoint interiors. Indeed, suppose otherwise. By density of \(\Gamma_<\), there exists \(x \in \Gamma_\vee\) for which \(x^\star \in A(W_j) + X_{ij}\) and \(x^\star \in A(W_k) + X_{ik}\), for \(j \neq k\). But then \(x \in \cps(A(W_j) + X_{ij}\) and \(x \in \cps(A(W_k) + X_{ik}\), a contradiction, so these sets have disjoint interiors, as required.
\end{proof}

\begin{corollary}\label{cor:sub=>GIFS}
Suppose that \(\Omega = \Omega(\mathcal{S},W)\) is a hull of cut and project sets with Euclidean total space which consists of \(L\)-sub patterns, that is, conditions (1--4) of Theorem \ref{thm:main} apply (which, in particular, defines a particular contraction \(A \colon \intl \to \intl\)).

Denote also \(\mathrm{sup}(W) = \bigcup_{i=1}^\ell W_i\). Then, perhaps after taking an appropriate translate of \(W\), there is a window \(W' = (W_j')_{j=1}^{\ell'}\) that is mutually constructable from \(W\), where the \(W_j' \in \mathscr{B}(W)\) have mutually disjoint interiors (which may be taken as the acceptance domains of \(r\)-patches, for sufficiently large \(r\)) with union \(\mathrm{sup}(W)\) and so that \(W' = (W_j')_{j=1}^{\ell'}\) is the attractor of a GIFS of the form
\[
W'_i = \bigcup_{j=1}^{\ell'} A(W'_j) + X_{ij} ,
\]
where each \(X_{ij} \subset \Gamma_\vee\), and where the above union may be chosen to be disjoint.
\end{corollary}

\begin{proof}
By Lemma \ref{prop:simplicity of subs}, by labelling points with their \(r\)-patches, for sufficiently large \(r\), we in fact have a substitution \(\sub \colon \Omega \to \Omega\) which may be defined by a CD and non-overlapping rule. Since the \(r\)-patch of a point is determined by which acceptance domain of \(r\)-patch it projects to (see Lemma \ref{lem:acceptance domains indicate patches}), this recolouring amounts to replacing \(W\) with its acceptance domains (one for each new colour). By definition, the acceptance domains are in \(\mathscr{B}(W)\), and they tile the window (Lemma \ref{lem:ADs tile}), and as a collection this gives a mutually constructable window since adding these labels defines an MLD equivalence. By Theorem \ref{thm:recolour to GIFS}, up to a translation, the window is the attractor of a GIFS of the stated form.
\end{proof}

Thus, whilst it is not true that windows of \(L\)-sub patterns are attractors of some GIFS, it is at least true (up to a translation) once they are decomposed sufficiently into constructable subsets. However, as noted above, both perspectives are valuable, particularly, as we have seen, when considering polytopal windows. Indeed, in practice, it is useful to know results that apply to the a priori simpler notion of \(L\)-sub.

\subsection{\(L\)-sub tilings versus substitution tilings}
In the above section, we showed how a standard notion of a Delone (multi)set being substitutional corresponds rather directly (essentially, ``up to MLD'') to our notion of \(L\)-sub. In this section, we relate \(L\)-sub with the classical notion of being substitutive for tilings.

Given a tile substitution rule \(\sub\) (satisfying standard properties, such as FLC and primitivity), for each tiling \(\cT\) admitted by \(\sub\), it is well-known that there is always a `predecessor' tiling \(\cT' \in \Omega\) for which \(\sub \cT' = \cT\), see for instance \cite{AB98}. Substitution is defined by inflation followed by a subdivision rule \(S\) on tiles, so \(S(L\cT') = \cT\). Since subdivision \(S\) is clearly a local derivation, we have that \(L\cT' \LD \cT\). And since \(\cT'\) belongs to the translational hull of \(\cT\), we have that \(\cT' \LI \cT\), establishing that \(\cT\) is \(L\)-sub.

The relationship in the converse direction is established by a result of Solomyak \cite{Sol97} (and earlier, in the planar case, of Priebe--Frank and Solomyak \cite{PS01}), namely, that a pseudo self-affine tiling is self-affine, up to MLD. Recall that a tiling \(\cT\) is pseudo self-affine if it is repetitive and \(L\cT \LD \cT\), for an expansion \(L\).

So, suppose that \(\cT\) is a repetitive and \(L\)-sub pattern. It may be shown (see, for instance, \cite{Wal25}) that, at least for some power \(n \in \N\), the associated substitution map \(\sub \colon \Omega \to \Omega\) has some fixed point \(\sub^n(\cT') = \cT'\). Thus, \(L^n\cT' \LD \cT'\), so \(\cT' \in \Omega\) is pseudo self-affine. The result of \cite{Sol97} then says that \(\cT' \MLD \cT''\), where \(\cT''\) is fixed by a stone inflation substitution, with expansion perhaps some further power of \(L^n\). Thus, the hull \(\Omega\) is, up to MLD, generated by a standard tile stone inflation rule. With the above, this establishes that our notion of \(L\)-sub only differs from other standard notions up to redecorating our patterns up to MLD.

\bibliographystyle{amsalpha}
\bibliography{biblio.bib}

\end{document}